\documentclass[a4paper,10pt,reqno]{amsart}
\usepackage[latin1]{inputenc}
\usepackage{dsfont}
\usepackage{amsxtra}
\usepackage{amsmath}
\usepackage{amssymb}
\usepackage{amsfonts}
\usepackage[all]{xy}
\usepackage{mathrsfs}
\usepackage{amsthm}

\theoremstyle{plain}
\newtheorem{thm}{Theorem}[section]
\newtheorem{lemma}[thm]{Lemma}
\newtheorem{prop}[thm]{Proposition}
\newtheorem{cor}[thm]{Corollary}
\newtheorem{conj}[thm]{Conjecture}

\theoremstyle{definition}
\newtheorem{df}[thm]{Definition}
\theoremstyle{remark}
\newtheorem{rk}[thm]{Remark}
\newtheorem{ex}[thm]{Example}

\numberwithin{equation}{section}

\def\a{{\alpha}}
\def\l{{\lambda}}

\def\NN{\mathbb{N}}
\def\QQ{\mathbb{Q}}
\def\CC{\mathbb{C}}
\def\GG{\mathbb{G}}

\def\ZZ{\mathbb{Z}}
\def\PP{\mathbb{P}}

\def\aen{\mathfrak{a}}

\def\gen{\mathfrak{g}}
\def\hen{\mathfrak{h}}
\def\len{\mathfrak{l}}

\def\sen{\mathfrak{s}}
\def\ten{\mathfrak{t}}

\def\Hen{\mathfrak{H}}
\def\Sen{\mathfrak{S}}

\def\Ac{\mathcal{A}}
\def\Bc{\mathcal{B}}
\def\Cc{\mathcal{C}}

\def\Fc{\mathcal{F}}

\def\Oc{\mathcal{O}}
\def\Pc{\mathcal{P}}

\def\Rc{\mathcal{R}}
\def\Sc{\mathcal{S}}

\def\Sym{\mathbf\Lambda}
\def\oneb{\mathbf{1}}

\def\sb{\mathbf{s}}

\def\Ab{\mathbf{A}}
\def\Bb{\mathbf{B}}
\def\Fb{\mathbf{F}}
\def\Hb{\mathbf{H}}

\def\Ub{\mathbf{U}}
\def\Vb{\mathbf{V}}

\def\Sb{\mathbf{S}}
\def\Tb{\mathbf{T}}

\def\homo{\operatorname{\it \mathscr{H}\kern-.25em om}}
\def\ext{\operatorname{\it \mathscr{E}\kern-.25em xt}}
\def\edo{\operatorname{\it \mathscr{E}\kern-.25em nd}}
\def\der{\operatorname{\it \mathscr{D}\kern-.25em er}}

\def\Wt{\operatorname{Wt}\nolimits}
\def\SL{\operatorname{SL}\nolimits}
\def\Hom{\operatorname{Hom}\nolimits}
\def\ch{\operatorname{ch}\nolimits}
\def\can{\operatorname{can}\nolimits}
\def\End{\operatorname{End}\nolimits}
\def\G{\operatorname{G}\nolimits}
\def\spe{\operatorname{spe}\nolimits}
\def\top{\operatorname{top}\nolimits}

\def\soc{\operatorname{soc}\nolimits}
\def\ad{\operatorname{ad}\nolimits}
\def\KZ{\operatorname{KZ}\nolimits}
\def\Fr{\operatorname{Fr}\nolimits}
\def\Supp{\operatorname{Supp}\nolimits}

\def\Ext{\operatorname{Ext}\nolimits}
\def\RHom{\operatorname{RHom}\nolimits}
\def\REnd{\operatorname{REnd}\nolimits}
\def\Res{\operatorname{Res}\nolimits}
\def\Ind{\operatorname{Ind}\nolimits}

\def\can{\operatorname{can}\nolimits}
\def\ch{\operatorname{ch}\nolimits}

\def\Im{\operatorname{Im}\nolimits}

\def\Irr{\operatorname{Irr}\nolimits}

\def\CIrr{\operatorname{PI}\nolimits}

\def\Rep{\operatorname{Rep}\nolimits}

\def\Hom{{\rm{Hom}}}
\def\gr{{\rm{gr}}}
\def\dim{{\rm{dim}}}

\def\la{{\langle}}

\def\o{{\omega}}
\def\g{{\gamma}}
\def\eps{{\epsilon}}
\def\l{{\lambda}}

\author{P. Shan, E. Vasserot}
\address{P.S. : Universit\'e Paris 7, UMR CNRS 7586, F-75013 Paris, E.V. : Universit\'e Paris 7, UMR CNRS 7586, F-75013 Paris,}
\email{shan@math.jussieu.fr, vasserot@math.jussieu.fr}

\thanks{2000{\it Mathematics Subject Classification.} }

\title
[Heisenberg and Cherednik]
{Heisenberg algebras and rational double affine Hecke algebras}

\begin{document}
\begin{abstract}
In this paper we categorify the Heisenberg action on the Fock space via the
category $\Oc$ of cyclotomic rational double affine Hecke algebras.
This permits us to relate the filtration by the support on the Grothendieck
group of $\Oc$ to a representation theoretic grading defined using the
Heisenberg action. This implies a recent conjecture of Etingof.
\end{abstract}

\maketitle

\setcounter{tocdepth}{2}

\tableofcontents

\section{Introduction and notation}

\subsection{Introduction}
In this paper we study a relationship between the representation theory of
certain rational double affine Hecke algebras (=RDAHA) and the representation theory
of affine Kac-Moody algebras.
Such connection is not new and appears already at several places in the literature.
A first occurrence is Suzuki's functor \cite{Su} which maps the Kazhdan-Lusztig category of modules
over the affine Kac-Moody algebra $\widehat{\sen\len}_n$ at a negative level to the representation category of the RDAHA of
${\sen\len}_m$. A second one is a cyclotomic version of  Suzuki's functor \cite{VV} which maps a
more general version of the
parabolic category $\Oc$ of $\widehat{\sen\len}_n$ at a negative level to the representation category of the cyclotomic RDAHA. A third one comes from the relationship between the
cyclotomic RDAHA and quiver varieties, see e.g., \cite{Go}, and from the relationship between
quiver varieties and affine Kac-Moody algebras.
Finally, a fourth one, which is closer to our study, comes from the relationship in \cite{S} between the
Grothendieck ring of cyclotomic RDAHA and the level $\ell$ Fock space
$\Fc_{m,\ell}$ of $\widehat{\sen\len}_m$.
 In this paper we focus on a recent conjecture of Etingof
\cite{E} which relates the support of the objects of the category $\Oc$ of
$H(\Gamma_n)$, the RDAHA associated with
the complex reflection group $\Gamma_n=\Sen_n\ltimes(\ZZ_\ell)^n$,
to a representation theoretic grading of
the Fock space $\Fc_\ell=\Fc_{\ell,1}$.
These conjectures yield in particular an explicit
formula for the number of finite dimensional
$H(\Gamma_n)$-modules. This was not known so far.
The appearance of the Fock space $\Fc_\ell$ is not a hazard.
It is due to the following two facts, already noticed in \cite{E}.
First, by level-rank duality, the $\widehat{\sen\len}_m$-module $\Fc_{m,\ell}$
carries a  level $m$ action on $\widehat{\gen\len}_\ell$.
It carries also a level $1$ action on $\widehat{\gen\len}_\ell$, under which it is identified with $\Fc_\ell$.
Next, the category $\Oc$ of
the algebras $H(\Gamma_n)$ with $n\geqslant 1$ categorifies  $\Fc_{m,\ell}$ by \cite{S}.
Our proof consists precisely to
interpret the support of the $H(\Gamma_n)$-modules in terms of the $\widehat{\sen\len}_m$-action
on $\Fc_{m,\ell}$, and then to interpret
this in terms of the $\widehat{\sen\len}_\ell$ action on $\Fc_\ell$.
An important ingredient is a categorification (in a weak sense) of the action of the Heisenberg algebra
on $\Fc_\ell$ and $\Fc_{m,\ell}$. The categorification of the Heisenberg algebra has recently been studied by several authors. We'll come back to this in another publication.

\subsection{Organisation}
The organisation of the paper is the following.

Section 2 is a reminder on rational DAHA. We recall some basic facts concerning
parabolic induction/restriction functors. In particular we describe their
behavior on the support of the modules.

Section 3 contains basic notations for complex reflection groups, for the
cyclotomic rational DAHA $H(\Gamma_n)$ and for affine Lie algebras.
In particular we introduce the category $\Oc(\Gamma_n)$ of
$H(\Gamma_n)$-modules, the functor $\KZ$, Rouquier's equivalence from
$\Oc(\Sen_n)$ to the module category of the $\zeta$-Schur algebra.
Next we recall the categorification of the Fock space representation
of $\widehat{\sen\len}_m$ in \cite{S}, and we describe the filtration by the support
on $\Oc(\Gamma_n)$.

Section 4 is more combinatorial.
We recall several constructions related to 
Fock spaces and symmetric polynomials.
In particular we give a relation between symmetric polynomials and the
representation ring of the group $\Gamma_n$, and we describe several
representations on the level $\ell$ Fock space
(of Heisenberg algebras and of affine Kac-Moody algebras).

Section 5 is devoted to the categorification of the Heisenberg action on the
Fock space, using $\Oc(\Gamma_n)$. Then we introduce a particular class
of simple objects in $\Oc(\Gamma_n)$, called the primitive modules,
and we compute the endomorphism algebra of some modules
induced from primitive modules. Finally we introduce the operators $\tilde a_\l$
which are analogues for the Heisenberg algebra of the
Kashiwara's operators $\tilde e_q$, $\tilde f_q$
associated with Kac-Moody algebras.

Section 6 contains the main results of the paper.
Using our previous constructions we compare the filtration by the support
on $\Oc(\Gamma_n)$ with a representation-theoretic grading on the Fock space.
This confirms a conjecture of Etingof, yieldding in particular the number of
finite dimensional simple objects in $\Oc(\Gamma_n)$
for integral $\ell$-charge
(this corresponds to some rational values of the parameters of $H(\Gamma_n)$).

Finally there are two appendices containing basic facts on Hecke algebras,
Schur algebras, quantum groups, quantum Frobenius homomorphism
and on the universal R-matrix.

\subsection{Notation}
Now we introduce some general notation.
Let $\Ac$ be a $\CC$-category, i.e., a $\CC$-linear additive category.
We'll write $Z(\Ac)$ for the center of $\Ac$, a $\CC$-algebra.
Let $\Irr(\Ac)$ be the set of isomorphism classes of simple objects
of  $\Ac$. If $\Ac=\Rep(\Ab)$, the
category of all finite-dimensional representations of a $\CC$-algebra $\Ab$,
we abbreviate $$\Irr(\Ab)=\Irr(\Rep(\Ab)).$$
For an Abelian or triangulated category
let $K(\Ac)$ denote its Grothendieck group.
We abbreviate $K(\Ab)=K(\Rep(\Ab))$. We set
$$[\Ac]=K(\Ac)\otimes\CC.$$ For an object $M$ of $\Ac$ we write $[M]$
for the class of $M$ in $[\Ac]$.
For an Abelian category $\Ac$
let $D^b(\Ac)$ denote its bounded derived category.
We abbreviate $D^b(\Ab)=D^b(\Rep(\Ab))$.
The symbol $H^*(\PP^{m-1})$
will denote both the complex
$$\langle m\rangle=\bigoplus_{i=0}^{m-1}\CC[-2i]\in
D^b(\CC)$$ and the integer $m$ in $K(\CC)=\ZZ$.
Given two Abelian $\CC$-categories $\Ac$, $\Bc$ which are Artinian (i.e.,
objects are of finite length and Hom's are finite dimensional) we define
the tensor product (over $\CC$)
$$\otimes:\Ac\times\Bc\to\Ac\otimes\Bc$$
as in \cite[sec.~5.1, prop.~5.13]{De}.
Recall that for $\Ac=\Rep(\Ab)$ and $\Bc=\Rep(\Bb)$ we have
$\Ac\otimes\Bc=\Rep(\Ab\otimes\Bb)$.
Given a category $\Ac$ and objects $A,A'\in\Ac$,
we write $\Hom_\Ac(A,A')$ for the collection of morphisms $A\to A'$.
Given categories $\Ac$, $\Bc$ and functors $F,F':\Ac\to\Bc$
we write $\Hom(F,F')$ for the collection of morphisms $F\to F'$.
We denote the identity morphism
$A\to A$ by $\oneb_A$ and the identity morphism
$F\to F$ by $\oneb_F$.
Given a category $\Cc$ and a functor $G:\Bc\to\Cc$
let $G\circ F$ be the composed functor $\Ac\to\Cc$.
For a functor $G':\Bc\to\Cc$ and morphisms of functors
$\phi\in\Hom(F,F')$, $\psi\in\Hom(G,G')$ we write
$\psi\phi$ for the morphism of functors $G\circ F\to G'\circ F'$
given by
$$(\psi\phi)(A)=\psi(F'(A))\circ G(\phi(A))\in
\Hom_\Cc\bigl(G(F(A)), G'(F'(A))\bigr),\quad A\in\Ac.$$

\subsection{Acknowledgments}
We are grateful to I. Losev for a careful reading of a preliminary draft of
the paper.

\section{Reminder on rational DAHA's}

\subsection{The category $\Oc(W)$}
Let  $W$ be any complex reflection group.
Let $\hen$ be the reflection representation of $W$.
Let $S$ be the set of pseudo-reflections in $W$.
Let $c:S\to\CC$ be a map that is constant on the
$W$-conjugacy classes.
The rational DAHA attached to $W$ with parameter $c$ is the
quotient $H(W)$ of the smash product of $\CC W$ and the tensor algebra of
$\hen\oplus\hen^*$ by the relations
$$[x, x' ] = 0, \quad [y, y' ] = 0, \quad
[y, x] = \langle x, y \rangle-\sum_{s\in S}c_s\langle\alpha_s,y\rangle
\langle x,\check\alpha_s\rangle s,$$
for all $x, x'\in\hen^*$, $y, y'\in\hen$.
Here $\langle\bullet,\bullet\rangle$
is the canonical pairing between $\hen^*$ and $\hen$,
the element $\alpha_s$ is a generator of
$\Im(s|_{\hen^*}-1)$ and $\check\alpha_s$ is the generator of
$\Im(s|_{\hen}-1)$
such that $\langle\alpha_s , \check\alpha_s\rangle = 2$.
Let $R_x$, $R_y$ be the subalgebras generated by
$\hen^*$ and $\hen$ respectively.
We may abbreviate
$$\CC[\hen]=R_x,\quad\CC[\hen^*]=R_y.$$
The category $\Oc$ of $H(W)$ is the full subcategory
$\Oc(W)$ of the category of
$H(W)$-modules consisting of objects that are finitely generated as
$\CC[\hen]$-modules and $\hen$-locally nilpotent.
We recall from \cite[sec.~ 3]{GGOR} the following properties
of $\Oc(W)$. It is a quasi-hereditary category.
The standard modules are labeled
by the set $\Irr(\CC W)$ of isomorphism classes of irreducible $W$-modules.
Let $\Delta_\chi$ be the standard module associated with the module
$\chi\in\Irr(\CC W)$. It is the induced module
$$\Delta_{\chi}=\Ind^{H(W)}_{W\ltimes R_y}(\chi).$$
Here $\chi$ is regarded as a $W\ltimes R_y$-module such that $\hen^\ast\subset R_y$ acts by zero.
Let $L_{\chi}$, $P_{\chi}$
denote the top and the projective cover of $\Delta_{\chi}$.

\begin{rk}
\label{rk:2.1}
The definitions above still make sense if $\hen$ is any
faithful finite dimensional $\CC W$-module.
To avoid any confusion we may write
$$\Oc(W,\hen)=\Oc(W),\quad
H(W,\hen)=H(W).$$
\end{rk}

\subsection{The stratification of $\hen$}
\label{sec:support1}
Let $W$ be a complex reflection group.
Let $\hen$ be the reflection representation of $W$.
For a parabolic subgroup $W'\subset W$ let $X_{W'}^\circ$
be the set of points of $\hen$ whose stabilizer in $W$ is conjugate (in $W$) to $W'$. By a theorem of Steinberg, the sets
$X_{W',\hen}^\circ$, when $W'$ runs over a set of representatives
of the $W$-conjugacy classes of parabolic subgroups of $W$,
form a stratification of $\hen$
by smooth locally closed subsets,
see also \cite[sec.~6]{G} and the references there.
Let $X_{W'}$ be the closure of $X_{W'}^\circ$ in $\hen$.
To avoid any confusion we may write
$X_{W',\hen}^\circ=X_{W'}^\circ$
and $X_{W',\hen}=X_{W'}$. The set $X_{W',\hen}$ consists of points of $\hen$ whose stabilizer is $W$-conjugate to $W'$. We have
$$X_{W',\hen}=\bigsqcup X_{W'',\hen}^\circ,$$
where the union is over a set of representatives
of the $W$-conjugacy classes of the parabolic subgroups $W''$ of $W$
which contain $W'$.
Further, the quotient $X_{W',\hen}/W$ is an
irreducible closed subset of $\hen/W$.

\subsection{Induction and restriction functors on $\Oc(W)$}
Fix an element $b\in\hen$. Let $W_b\subset W$ be the stabilizer of $b$, and
$$\pi_b:\hen\to\hen/\hen^{W_b}$$ be the obvious projection onto
the reflection representation of $W_b$.
The parabolic induction/restriction functor relatively to the point $b$
is a functor \cite{BE}
$$\Ind_b:\Oc(W_b,\hen/\hen^{W_b})\to\Oc(W,\hen),\quad
\Res_b:\Oc(W,\hen)\to\Oc(W_b,\hen/\hen^{W_b}).$$
Since the functors $\Ind_b$, $\Res_b$ do not
depend on $b$ up to isomorphism, see
\cite[sec.~3.7]{BE}, we may write
$${}^\Oc\!\Ind^{W}_{W_b}=\Ind_b,\quad{}^\Oc\!\Res^{W}_{W_b}=\Res_b$$
if it does not create any confusion.
The {\it support} of a module $M$ in $\Oc(W,\hen)$ is the support
of $M$ regarded as a $\CC[\hen]$-module.
It is a closed subset
$\Supp(M)\subset\hen.$
For any simple module $L$ in $\Oc(W,\hen)$ we have
$\Supp(L)=X_{W',\hen}$ for some parabolic subgroup $W'\subset W$. For $b\in X^\circ_{W',\hen}$ the module $\Res_b(L)$ is a nonzero finite dimensional module.
See \cite[sec.~3.8]{BE}. The support of a module is the union of the supports of all its constituents. So the support of any module in $\Oc(W,\hen)$ is a union of $X_{W',\hen}$'s.
Let us consider the behavior of the support under restriction.

\begin{prop}\label{prop:ressup}
Let $W'\subset W$ be a parabolic subgroup.
Let $\hen'$ be the reflection representation of $W'$.
Let $X\subset\hen$ be the support of a module $M$ in $\Oc(W,\hen)$.
Let $X'\subset\hen'$ be the support of the module
$M'={}^\Oc\!\Res^W_{W'}(M)$.

(a) We have $M'\neq 0$ if and only if $X_{W',\hen}\subset X$.

(b) Assume that 
$X=X_{W'',\hen}$
with $W''\subset W$ a parabolic subgroup. If $M'\neq 0$ then
$W''$ is $W$-conjugate to a subgroup of $W'$ and we have
$$X'=
\bigcup_{W_1}X_{W_1,\hen'}=
\bigsqcup_{W_1}X_{W_1,\hen'}^\circ,$$
where $W_1$ runs over a set of representatives of the $W'$-conjugacy classes
of parabolic subgroups of $W'$ containing a subgroup $W$-conjugated to $W''$.

\end{prop}

\begin{proof}
Part $(a)$ is immediate from the definition of the restriction,
because  for $b\in\hen$ it implies that
$\Res_b(M)\neq 0$ if and only if $b\in X.$
Now we prove $(b)$. For a parabolic subgroup $W_1\subset W'$ we have
$$\aligned
X_{W_1,\hen'}\subset X'
&\iff {}^\Oc\!\Res^{W'}_{W_1}(M')\neq 0\cr
&\iff {}^\Oc\!\Res^{W}_{W_1}(M)\neq 0\cr
&\iff X_{W_1,\hen}\subset X_{W'',\hen}.
\endaligned$$
Here the first and third equivalence follow from $(a)$,
while the second one follows from the
transitivity of the restriction functor \cite[cor.~2.5]{S}.
Therefore
$X_{W_1,\hen'}^\circ\subset X'$ if and only if
$X_{W_1,\hen'}\subset X'$ if and only if
$W_1$ contains a subgroup $W$-conjugate to $W''$.

\end{proof}

\begin{rk}
For any closed point
$b$ of a scheme $X$ we denote by $X_b^\wedge$ the completion of
$X$ at $b$ (a formal scheme).
Assume that $M'={}^\Oc\!\Res^W_{W'}(M)$ is non zero.
Let $\pi$ be the canonical projection $\hen\to\hen'=\hen/\hen^{W'}$.
For $b\in X_{W',\hen}^\circ$
the definition of the restriction functor yields the following formula
$$0\in\pi^{-1}(X'),\quad
X^\wedge_b=b+\pi^{-1}(X')^\wedge_0.$$
\end{rk}

\vskip3mm

Next, we consider the behavior of the support under induction.
Before this we need the following two lemmas.
The $\CC$-vector space $[\Oc(W)]$
is spanned by the set
$\{[\Delta_\chi];\chi\in\Irr(\CC W)\}$.
Thus there is a unique $\CC$-linear isomorphism
\begin{equation}\label{grothW}
\spe:[\Rep(\CC W)]\to[\Oc(W)],\quad
[\chi]\mapsto[\Delta_\chi].
\end{equation}
The parabolic induction/restriction functor
is exact. We'll need the following lemma \cite{BE}.

\begin{lemma} \label{lem:indresgroth}
Let $W'\subset W$ be a parabolic subgroup.
Let $\hen'$ be the reflection representations of $W'$.
Under the isomorphism (\ref{grothW}) the maps
$${}^\Oc\!\Ind^{W}_{W'}:[\Oc(W',\hen')]\to [\Oc(W,\hen)],\quad
{}^\Oc\!\Res^{W}_{W'}:[\Oc(W,\hen)]\to [\Oc(W',\hen')]$$
coincide with the induction and restriction
$$\Ind^{W}_{W'}:[\Rep(\CC W')]\to [\Rep(\CC W)],\quad
\Res^{W}_{W'}:[\Rep(\CC W)]\to [\Rep(\CC W')].$$
\end{lemma}
\noindent
We'll also need the following version of the
Mackey induction/restriction theorem.
First, observe that for any parabolic subgroup $W'\subset W$ and any $x\in W$
there is a canonical $\CC$-algebra isomorphism
$$\varphi_x:\ H(W')\to H(x^{-1}W'x),\ w\mapsto x^{-1}wx,\ f\mapsto x^{-1}fx,\ f'\mapsto x^{-1}f'x,$$
for $w\in W'$, $f\in R_x$, $f'\in R_y$.
It yields an exact functor
$$\Oc(W')\to\Oc(x^{-1}W'x),\quad M\mapsto {}^xM,$$
where ${}^xM$ is the $H(x^{-1}W'x)$-module obtained by twisting the
$H(W')$-action on $M$ by $\varphi_x$.

\begin{lemma}
\label{lem:mackey}
Let $W',W''\subset W$ be parabolic subgroups.
Let $\hen'$, $\hen''$ be the reflection representations of $W'$, $W''$.
For  $M\in\Oc(W',\hen')$ we have the following formula in $[\Oc(W'',\hen'')]$
\begin{equation}\label{mackey}
{}^\Oc\!\Res^{W}_{W''}\circ\,{}^\Oc\!\Ind^{W}_{W'}([M])=\sum_x
{}^\Oc\!\Ind^{W''}_{W''\cap x^{-1}W'x}\circ\,
{}^x\bigl({}^\Oc\!\Res^{W'}_{xW''x^{-1}\cap W'}([M])\bigr),
\end{equation}
where $x$ runs over a set of representatives of the cosets
in $W'\setminus W /W''$.
\end{lemma}

\begin{proof} Use Lemma \ref{lem:indresgroth} and the usual Mackey induction/restriction theorem
associated with the triplet of groups $W$, $W'$, $W''$.
\end{proof}

\begin{rk}\label{rk:rk00}
For a future use, note that the left hand side of (\ref{mackey}) is zero if and only if each term in the sum of the right hand side is zero, because each of these terms is the class of a module in $\Oc(W'',\hen'')$.
\end{rk}

\noindent Now, we can prove the following proposition.

\begin{prop}\label{prop:indsup}
Let $W''\subset W'\subset W$ be a parabolic subgroups. Let $\hen'$ be the reflection representation
of $W'$. For a simple module $L\in\Oc(W',\hen')$
with $\Supp(L)=X_{W'',\hen'}$, we have
$${}^\Oc\!\Ind^{W}_{W'}(L)\neq 0,\quad
\Supp\bigl({}^\Oc\!\Ind^{W}_{W'}(L)\bigr)=X_{W'',\hen}.$$
\end{prop}

\begin{proof}
First, note that ${}^\Oc\!\Ind^{W}_{W'}(L)\neq 0$ by
Lemma \ref{lem:mackey}, because
$${}^\Oc\!\Res^{W}_{W'}\circ\,{}^\Oc\!\Ind^{W}_{W'}([L])=[L]+[M]$$
for some $M\in\Oc(W',\hen')$ and $[L]\neq 0$.
Therefore ${}^\Oc\!\Ind_{W'}^W(L)\neq 0.$
We abbreviate $M={}^\Oc\!\Ind^{W}_{W'}(L)$.
To compute the support of $M$ 
we first check that
$$X_{W'',\hen}\subset\Supp(M).$$  By Proposition \ref{prop:ressup} we have
$$\aligned
X_{W'',\hen}\subset\Supp(M)
&\iff X_{W'',\hen}^\circ\subset\Supp(M)\cr
&\iff {}^\Oc\!\Res^{W}_{W''}(M)\neq 0.
\endaligned$$
By Remark \ref{rk:rk00} the last equality holds if and only if
$${}^\Oc\!\Res^{W'}_{xW''x^{-1}\cap W'}(L)\neq 0$$
for some $x\in W$. This identity is indeed true for $x=1$ because $W''\subset W'$ and
$$X_{W'',\hen'}=\Supp(L)\Rightarrow {}^\Oc\!\Res^{W'}_{W''}(L)\neq 0.$$
Next we prove the inclusion
$$\Supp(M)\subset X_{W'',\hen}.$$
Any point $b$ of $\hen\setminus X_{W'',\hen}$ is contained in the set $X_{W''',\hen}^\circ$
for some parabolic subgroup $W'''\subset W$ such that $W''$ is not conjugate to a subgroup of $W'''$ :
it suffices to set $W'''=W_b$.
We must check that for such a subgroup $W'''\subset W$ we have
$$X^\circ_{W''',\hen}\not\subset\Supp(M).$$
By Proposition \ref{prop:ressup}
it is enough to check that
$${}^\Oc\!\Res^{W}_{W'''}(M)= 0.$$
Now, by Lemma \ref{lem:mackey} we have the following
formula in $[\Oc(W''',\hen)]$
$${}^\Oc\!\Res^{W}_{W'''}([M])=\sum_x
{}^\Oc\!\Ind^{W'''}_{W'''\cap x^{-1}W'x}\circ\,
{}^x\bigl({}^\Oc\!\Res^{W'}_{xW'''x^{-1}\cap W'}([L])\bigr).$$
Here $x$ runs over a set of representatives
of the cosets in $W'\setminus W /W'''$.
Since $W''$ is not conjugate to a subgroup of $W'''$ it is a fortiori
not conjugate to a subgroup of $xW'''x^{-1}\cap W'$, i.e., we have
$$X_{xW'''x^{-1}\cap W',\hen'}^\circ\cap X_{W'',\hen'}=\emptyset.$$
Therefore Proposition \ref{prop:ressup} yields
$${}^\Oc\!\Res^{W'}_{xW'''x^{-1}\cap W'}(L)=0,$$
because $\Supp(L)=X_{W'',\hen'}.$
This implies that
$${}^\Oc\!\Res^{W}_{W'''}([M])=0.$$
Hence we have also
$${}^\Oc\!\Res^{W}_{W'''}(M)=0.$$
We are done.
\end{proof}

\vskip3mm

\section{The cyclotomic rational DAHA}

\subsection{Combinatorics}
For a sequence $\l=(\l_1,\l_2,\dots)$ of integers $\geqslant 0$ we set
$|\l|=\l_1+\l_2+\cdots$.
Let
$$\Lambda(\ell,n)=\{
\l=(\lambda_1,\lambda_2,\dots\lambda_\ell)\in\NN^\ell\,;\,|\l|=n\}.$$
It is the set of {\it compositions} of $n$ with $\ell$ parts.
Let $\Pc_n$ be the set of {\it partitions} of $n$,
i.e., the set of non-increasing sequences $\l$ of integers
$> 0$ with sum $|\l|=n$. We write $\l'$ for the
transposed partition and $l(\l)$ for its length, i.e., for the
number of parts in $\l$. We write also
\begin{equation}\label{zlambda}
z_\lambda=\prod_{i\geqslant 1}i^{m_i}\,m_i!,\end{equation}
where $m_i$ is the number of parts of $\lambda$ equal to $i$.
Given a positive integer $m$ and a partition $\l$ we write also
$$m\l=(m\l_1,m\l_2,\dots).$$
To any partition we associate a {\it Young diagram},
which is a collection of rows
of square boxes with $\l_i$ boxes in the $i$-th row, $i=1,\dots,l(\l)$.
A box in a Young diagram is called a {\it node}.
The coordinate of the $j$-th box in the $i$-th row
is the pair of integers $(i,j)$.
The {\it content} of the node of coordinate $(i,j)$ is the integer $j-i$.
Let the set $\Pc_0$ consist of a single element, the unique partition of zero,
which we denote by 0.
Let $\Pc=\bigsqcup_{n\geqslant 0}\Pc_n$ be the set of all partitions.
We'll abbreviate $\ZZ_\ell=\ZZ/\ell\ZZ$.
Let $\Pc^\ell$ be the set of {\it $\ell$-partitions}, i.e., the set of all
partition valued functions on $\ZZ_\ell$.
Let $\Pc^\ell_n$ be the subset
of $\ell$-tuples $\l=(\l(p))$ of partitions with
$|\l|=\sum_p|\l(p)|=n$. 
Let $\Gamma$ be the group of the $\ell$-th roots of 1 in $\CC^\times$.
We define the sets $\Pc^\Gamma$, $\Pc_n^\Gamma$ of
partition valued functions on $\Gamma$
in the same way.

\subsection{The complex reflection group $\Gamma_n$}
\label{sec:3.2}
Fix non negative integers $\ell$, $n$. Unless specified otherwise
we'll always assume that $\ell,n\neq 0$.
Let $\frak S_n$ be the
symmetric group on $n$ letters and $\Gamma_{n}$ be the semi-direct product
$\frak S_n\ltimes\Gamma^n$, where $\Gamma^n$ is the Cartesian product of $n$
copies of $\Gamma$.
We write also $\frak S_0=\Gamma^0=\Gamma_0=\{1\}$.
For $\g\in\Gamma$ let $\g_i\in\Gamma^n$ be the element with $\g$ at the
$i$-th place and with 1 at the other ones.
Let $s_{ij}$ be the transposition $(i,j)$ in $\Sen_n$.
We'll abbreviate $s_i=s_{i,i+1}$. Write
$s_{ij}^{\gamma}=s_{ij}\gamma_i\gamma_j^{-1}$ for $\gamma\in\Gamma$,
$i\neq j$. For $p\in\ZZ_\ell$ let
$\chi_p:\Gamma\to\CC^\times$ be the character $\gamma\mapsto\gamma^p$.
The assignment $p\mapsto\chi_p$
identifies $\ZZ_\ell$ with the group of characters of $\Gamma$.
The group $\Gamma_n$ is a complex reflection group.
For $\ell>1$
it acts on the vector space $\hen=\CC^n$ via the reflection representation.
For $\ell=1$
the reflection representation is given by the permutation  of coordinates
on the hyperplane
$$\CC_0^n=\{x_1+\dots+x_n=0\}\subset\CC^n.$$
We'll be interested in the following subgroups of $\Gamma_n$.

\begin{itemize}
\item To a composition
$\nu$ of $n$ we associate the set
$$I=\{1,2,\dots,n-1\}\setminus\{\nu_1,\nu_1+\nu_2,\dots\}.$$
Let
$\Gamma_{\nu}=\frak S_\nu\ltimes\Gamma^n$, where $\frak S_\nu=\frak S_I$ is
the subgroup of $\frak S_n$ generated
by the simple reflections $s_{i,i+1}$ with $i\in I$.

\item For integers $m,n\geqslant 0$ and a composition $\nu$ we set
$\Gamma_{n,\nu}=\Gamma_n\times \frak S_\nu$.
If $\nu=(m^j)$ for some integer $j\geqslant 0$
we abbreviate
$\Gamma_{n,(m^j)}=\Gamma_{n,\nu}$.
We write also $\Gamma_{n,m}=\Gamma_n\times\frak S_m$.
Any parabolic subgroup of $\Gamma_n$ is conjugate to $\Gamma_{l,\nu}$ for some
$l,\nu$ with $l+|\nu|\leqslant n$.
\end{itemize}

\subsection{Definition of the cyclotomic rational DAHA}
Fix a basis $(x,y)$ of $\CC^2$. Let $x_i$,
$y_i$ denote the elements $x,y$ respectively in the $i$-th summand
of $(\CC^2)^{\oplus n}$. The group $\Gamma_n$ acts on $(\CC^2)^{\oplus n}$
such that for distinct $i$, $j$, $k$ we have
$$\g_i(x_i)=\g^{-1} x_i,
\quad \g_i(x_j)=x_j, \quad \g_i(y_i)=\g y_i, \quad
\g_i(y_j)=y_j,$$
$$s_{ij}(x_i)=x_j,
\quad s_{ij}(y_i)=y_j, \quad s_{ij}(x_k)=x_k, \quad
s_{ij}(y_k)=y_k.$$ Fix $k\in\CC$ and $c_\g\in\CC$ for each $\g\in\Gamma$.
Since $\Gamma_n$ is a complex reflection group with the reflection
representation $\hen$ equal to
$(\CC^2)^{\oplus n}$, see above, we can define the algebra
$H(W)=H(W,\hen)$ for $W=\Gamma_n$.
We'll call $H(\Gamma_{n})$ the {\it cyclotomic rational DAHA}.
It is the quotient
of the smash product of $\CC \Gamma_{n}$ and the tensor algebra of
$(\CC^2)^{\oplus n}$ by the relations
$$[y_i,x_i]=-k\sum_{j\neq i}\sum_{\gamma\in\Gamma}s_{ij}^{\gamma}
-\sum_{\gamma\in\Gamma}c_\gamma\gamma_i,\quad c_1=-1,$$
$$[y_i,x_j]=k\sum_{\g\in\Gamma}s_{ij}^{\gamma}
\quad \text{if}\ i\neq j,$$
$$[x_i,x_j]=[y_i,y_j]=0.$$
Let $R_x$, $R_y$ be the subalgebras generated by
$x_1,x_2,\dots,x_n$ and $y_1,y_2,\dots,y_n$ respectively.
We'll write $\hen$, $\hen^*$ for the maximal spectrum of $R_x$, $R_y$.
The $\CC$-vector space $\hen$ is identified with $\CC^n$ in the obvious way.
We'll use another presentation where the
parameters are $h$, $h_p$ with $p\in\ZZ_\ell$ where
$k=-h$
and $-c_\g=\sum_{p\in\ZZ_\ell}\gamma^{-p}h_{p}$.
Note that $1=\sum_ph_p$.

\subsection
{The Lie algebras $\widehat{\sen\len}_\ell$ and $\widetilde{\sen\len}_\ell$}
Given complex numbers $h_p$, $p\in\ZZ_\ell$, with $\sum_ph_p=1$,
it is convenient to consider the following level 1 weight
\begin{equation}\label{lambda}
\Lambda=\sum_ph_p\,\omega_p.
\end{equation} Here the
$\omega_p$'s are the fundamental weights of the affine Lie algebra
$$\widehat{\sen\len}_\ell=
(\sen\len_\ell\otimes\CC[\varpi,\varpi^{-1}])\oplus\CC\oneb,$$
where $\oneb$ is a central element and the Lie bracket is given by
\begin{equation}\label{affbracket}
[x\otimes\varpi^r,y\otimes\varpi^s]=
[x,y]\otimes\varpi^{r+s}+r(x,y)\delta_{r,-s}\oneb,\quad
(x,y)=\tau(xy^t),\end{equation}
where $y\mapsto y^t$ is the transposition and $\tau$ is the trace.
The affine Lie algebra $\widehat{\sen\len}_\ell$ is generated by the symbols
$e_p$, $f_p$, $p=0,\dots,\ell-1,$ satisfying the Serre relations.
For $p\neq 0$ we have
$$e_p=e_{p,p+1}\otimes 1,\quad
e_0=e_{\ell,1}\otimes\varpi,\quad
f_p=e_{p+1,p}\otimes 1,\quad
f_0=e_{1,\ell}\otimes\varpi^{-1},$$
where $e_{p,q}$ is the usual elementary matrix in $\sen\len_\ell$.
We'll also use the extended affine Lie algebras
$\widetilde{\sen\len}_\ell$,
obtained by adding to
$\widehat{\sen\len}_\ell$
the 1-dimensional vector space spanned by the scaling element
$D$ such that
$[D,x\otimes\varpi^r]=r\,x\otimes\varpi^r$
and $[D,\oneb]=0$.
Let $\delta$ denote the dual of $D$, i.e., the smallest
positive imaginary root.
We equip the space of linear forms on the Cartan subalgebra of $\widetilde{\sen\len}_\ell$
with the pairing such that
$$\la\omega_p,\omega_q\rangle=\min(p,q)-pq/\ell,\quad
\langle\omega_p,\delta\rangle=1,\quad
\langle\delta,\delta\rangle=0.$$
Let $U(\widehat{\sen\len}_\ell)$ be the enveloping algebra of $\widehat{\sen\len}_\ell$, and
let $U^-(\widehat{\sen\len}_\ell)$ be the
subalgebra generated by the elements $f_p$ with $p=0,\dots,\ell-1.$
For $r\geqslant 0$ we write $U^-(\widehat{\sen\len}_\ell)_r$
for the subspace of $U^-(\widehat{\sen\len}_\ell)$
spanned by the monomials whose weight is the sum of $r$ negative simple roots.

\subsection{Representations of $\frak S_n$, $\Gamma_{n}$}
\label{sec:rep}
The set of isomorphism classes of irreducible
$\frak S_n$-modules is
$$\Irr(\CC \frak S_n)=\{\bar L_\l;\l\in\Pc_n\},$$
see \cite[sec.~I.9]{M}.
The set of isomorphism classes of irreducible
$\Gamma_n$-modules is
$$ \Irr(\CC \Gamma_{n})=\{\bar L_\l;\l\in\Pc_n^\ell\},$$
where $\bar L_\l$ is defined as follows.
Write $\l=(\l(p))$. The tuple of
positive integers $\nu_\l=(|\l(p)|)$ is a composition in
$\Lambda(\ell,n)$. Let
$$\bar L_{\l(p)}(\chi_{p-1})^{\otimes|\lambda(p)|}\in
\Irr(\CC\Gamma_{{|\l(p)|}})$$
be the tensor product of the
$\frak S_{|\l(p)|}$-module $\bar L_{\l(p)}$ and the one-dimensional
$\Gamma^{|\l(p)|}$-module $(\chi_{p-1})^{\otimes|\l(p)|}$.
The $\Gamma_{n}$-module
$\bar L_\l$ is given by
\begin{equation}
\label{labelsimple}\bar L_\l=\Ind_{\Gamma_{\nu_\l}}^{\Gamma_{n}}\bigl(
\bar L_{\l(1)}\chi_\ell^{\otimes|\l(1)|} \otimes
\bar L_{\l(2)}\chi_1^{\otimes|\l(2)|} \otimes\cdots\otimes
\bar L_{\l(\ell)}\chi_{\ell-1}^{\otimes|\l(\ell)|} \bigr).
\end{equation}

\subsection{The category $\Oc(\Gamma_{n})$}
Consider the $\CC$-algebra $H(\Gamma_n)$ with the parameter $\Lambda$ in
(\ref{lambda}). The category $\Oc$ of $H(\Gamma_n)$ is
the quasi-hereditary category $\Oc(\Gamma_{n})$.
The standard modules  are the induced modules
$$\Delta_{\l}=\Ind^{H(\Gamma_{n})}_{\Gamma_{n}\ltimes R_y}(\bar L_\l),\quad
\l\in\Pc_n^\ell.$$
Here $\bar L_\l$ is viewed as a $\Gamma_{n}\ltimes R_y$-module such
that $y_1,\dots y_n$ act trivially. Let $L_{\l}$, $P_{\l}$
denote the top and the projective cover of $\Delta_{\l}$.
Recall the $\CC$-linear isomorphism
\begin{equation}\label{groth}
\spe:[\Rep(\CC\Gamma_{n})]\to[\Oc(\Gamma_{n})],\quad
[\bar L_\l]\mapsto[\Delta_\l].
\end{equation}
To avoid cumbersome notation for induction/restriction functors
in $$\Oc(\Gamma)=\bigoplus_{n\geqslant 0}\Oc(\Gamma_n)$$ we'll abbreviate
\begin{equation}
\label{abrevO}
\begin{aligned}
&{}^\Oc\!\Ind_{n}={}^\Oc\!\Ind_{\Gamma_{n-1}}^{\Gamma_{n}},
&{}^\Oc\!\Res_{n}={}^\Oc\!\Res^{\Gamma_{n}}_{\Gamma_{n-1}},\cr
&{}^\Oc\!\Ind_{n,(m^r)}={}^\Oc\!\Ind_{\Gamma_{n,(m^r)}}^{\Gamma_{n+mr}},
&{}^\Oc\!\Res_{n,(m^r)}={}^\Oc\!\Res^{\Gamma_{n+mr}}_{\Gamma_{n,(m^r)}},\cr
&{}^\Oc\!\Ind_{n,mr}={}^\Oc\!\Ind_{\Gamma_{n,mr}}^{\Gamma_{n+mr}},
&{}^\Oc\!\Res_{n,mr}={}^\Oc\!\Res^{\Gamma_{n+mr}}_{\Gamma_{n,mr}}.
\end{aligned}
\end{equation}
We write also
$$\gathered
{}^\Oc\!\Ind_{(m^r)}={}^\Oc\!\Ind_{\Sen_m^r}^{\Sen_{mr}}:\Oc(\Sen_m^r)\to\Oc(\Sen_{mr}),\cr
{}^\Oc\!\Res_{(m^r)}={}^\Oc\!\Res_{\Sen_m^r}^{\Sen_{mr}}:\Oc(\Sen_{mr})\to\Oc(\Sen_m^r).
\endgathered$$

\subsection{The functor KZ}
\label{sec:KZ}
For $\zeta\in\CC^\times$ and $v_1,v_2,\dots,v_\ell\in\CC^\times$ let
$\Hb_\zeta(n,\ell)$ be the cyclotomic Hecke algebra associated with
$\Gamma_n$ and the parameters $\zeta,v_1,\dots,v_\ell$, see
Section \ref{app:cyclohecke}.
We'll abbreviate $\Hb(\Gamma_n)=\Hb_\zeta(n,\ell)$.
Assume that
$$\zeta=\exp(2i\pi h),\quad
v_p=v_1\exp\bigl(-2i\pi(h_1+h_2+\cdots+h_{p-1})\bigr).
$$
Then the KZ-functor \cite{GGOR} is a quotient functor
$$\KZ:\Oc(\Gamma_n)\to\Rep(\Hb(\Gamma_n)).$$
Since $\KZ$ is a quotient functor,
it admits a right adjoint functor
$$S:\Rep(\Hb(\Gamma_n))\to\Oc(\Gamma_n)$$
such that $\KZ\circ\, S=\oneb$.
By \cite[thm.~5.3]{GGOR}, for each projective module $Q\in\Oc(\Gamma_n)$
the canonical adjunction morphism $\oneb\to S\circ\KZ$ yields an isomorphism
\begin{equation}\label{isomproj}
Q\to S(\KZ(Q)).
\end{equation}

\subsection{The functor $R$}
\label{sec:induction}
Let $\Hb_\zeta(m)$ be the Hecke $\CC$-algebra of $GL_m$, see
Section \ref{app:cyclohecke}.
Let $\Sb_\zeta(m)$ be the $\zeta$-Schur $\CC$-algebra,
see Appendix \ref{app:B}.
The module categories of
$\Sb_\zeta(m)$, $\Hb_\zeta(m)$
are related through the Schur functor
\begin{equation*}
\Phi^*:\Rep(\Sb_\zeta(m))\to\Rep(\Hb_\zeta(m)).\end{equation*}
Set
$$\Lambda(m)_+=\Lambda(m,m)\cap\ZZ^m_+,
\quad
\ZZ^m_+=\{\l=(\l_1,\l_2,\dots,\l_m)\in\ZZ^m\,;\,
\l_1\geqslant \l_2\geqslant\cdots\geqslant \l_m\}.$$
The category $\Rep(\Sb_\zeta(m))$ is
quasi-hereditary with respect to the dominance order,
the standard objects being the modules $\Delta^S_\l$
with $\l\in\Lambda(m)_+$.
The comultiplication $\Delta$ yields a bifunctor (\ref{tensorschur})
$$\dot\otimes:
\Rep(\Sb_\zeta(m))\otimes\Rep(\Sb_\zeta(m'))\to\Rep(\Sb_\zeta(m+m')).$$
Now, assume that $h$ is a negative rational number with denominator $d$ and
let $\zeta\in\CC^\times$ be a primitive $d$-th root of 1.
Recall that $h$ is the parameter of the $\CC$-algebra $H(\frak S_m)$.
If $h\notin 1/2+\ZZ$ then
Rouquier's functor \cite{R} is an equivalence of quasi-hereditary categories
$$R:\Oc(\frak S_{m})\to\Rep(\Sb_\zeta(m)),\quad \Delta_\l\mapsto \Delta^S_\l,$$
such that
$\KZ=\Phi^*\circ R$. For $m=m'+m''$ we have a canonical equivalence of categories
$\Oc(\frak S_{m'})\otimes\Oc(\frak S_{m''})=
\Oc(\frak S_{m'}\times\frak S_{m''})$
and the induction yields a bifunctor
\begin{equation}\label{ind}
{}^\Oc\!\Ind_{m',m''}:\Oc(\frak S_{m'})\otimes\Oc(\frak S_{m''})
\to\Oc(\frak S_{m}).
\end{equation}
We'll abbreviate
$$\Oc(\frak S)=\bigoplus_{n\geqslant 0}\Oc(\frak S_n),\quad
\Rep(\Sb_\zeta)=\bigoplus_{n\geqslant 0}\Rep(\Sb_\zeta(n)).$$

\begin{prop}\label{prop:tensorproduct}
For $h\notin 1/2+\ZZ$
the functor $R$ is tensor equivalence
$\Oc(\frak S)\to\Rep(\Sb_\zeta).$
\end{prop}

\begin{proof}
We must check that $R$ identifies
the tensor product $\dot\otimes$
with the induction (\ref{ind}).
First, fix two projective objects
$X\in\Oc(\frak S_{m'})$ and $Y\in\Oc(\frak S_{m''})$.
We have
$$\aligned
\Phi^*\bigl(R(X)\dot\otimes R(Y)\bigr)&=
{}^\Hb\Ind_{m',m''}
\bigl(\Phi^* R(X)\otimes \Phi^* R(Y)\bigr)
\cr
&=
{}^\Hb\Ind_{m',m''}
\bigl(\KZ(X)\otimes\KZ(Y)\bigr)
\cr
&=
\KZ\bigl({}^\Oc\!\Ind_{m',m''}(X\otimes Y)\bigr)
\cr
&=
\Phi^* R\bigl({}^\Oc\!\Ind_{m',m''}(X\otimes Y)\bigr).
\endaligned$$
The first equality follows from Corollary \ref{cor:tensor*},
the second one and the fourth one
come from $\KZ=\Phi^*\circ R$, and the third one is the commutation of
$\KZ$ and the induction functors, see \cite{S}.
Since the modules
$R(X)\dot\otimes R(Y)$
and
$R\bigl({}^\Oc\!\Ind_{m',m''}(X\otimes Y)\bigr)$
are projective,
and since $\Phi^*$ is fully faithful on projectives we get that
$$R(X)\dot\otimes R(Y)=R\bigl({}^\Oc\!\Ind_{m',m''}(X\otimes Y)\bigr).$$
Now, since the functors
(\ref{tensorschur}), (\ref{ind})
are exact
and coincide on projective objects,
and since the category
$\Oc(\frak S_m)$ has enough projectives,
the proposition is proved.
\end{proof}

\subsection{The categorification of $\widetilde{\sen\len}_m$}
\label{sec:categ}
Recall that  $Z(\Oc(\Gamma_n))$ is the center of the category $\Oc(\Gamma_n)$.
Let $D_n(z)$ be the polynomial in
$Z(\Oc(\Gamma_n))[z]$ defined in \cite[sec.~4.2]{S}.
For any $a\in\CC(z)$ the projection to the generalized eigenspace of $D_n(z)$
with the eigenvalue $a$ yields an exact endofunctor $Q_{n,a}$ of
$\Oc(\Gamma_n)$.
Next, consider the point
$$b_{n}=(0,0,\dots,0,1)\in\hen,\quad\hen=\CC^n.$$
The induction and the restriction relatively to $b_{n}$ yield functors
\begin{equation*}
\begin{split}
{}^\Oc\!\Ind_{n}:
\Oc(\Gamma_{n-1})\to\Oc(\Gamma_{n}),
\quad
{}^\Oc\!\Res_{n}:\Oc(\Gamma_{n})\to\Oc(\Gamma_{n-1}).
\end{split}
\end{equation*}

\begin{df}\cite[sec.~4.2]{S}
The {\it $q$-restiction} and the {\it $q$-induction} functors
$$e_q:\Oc(\Gamma_n)\to\Oc(\Gamma_{n-1}),\quad
f_q:\Oc(\Gamma_{n-1})\to\Oc(\Gamma_n),\quad
q=0,1,\dots,m-1$$
are given by
$$\gathered
e_q=\bigoplus_{a\in\CC(z)}Q_{n-1,a/(z-\zeta^q)}\circ{}^\Oc\!\Res_n\circ\, Q_{n,a},
\cr
f_q=\bigoplus_{a\in\CC(z)}Q_{n,a(z-\zeta^q)}\circ{}^\Oc\!\Ind_n\circ\, Q_{n-1,a}.
\endgathered$$
\end{df}

\noindent
We'll abbreviate
$$E=e_0\oplus e_1\oplus\dots\oplus e_{m-1},\quad
F=f_0\oplus f_1\oplus\dots\oplus f_{m-1}.$$
Following \cite[sec.~6.3]{S},
for $L\in\Irr(\Oc(\Gamma))$ we set
$$\tilde e_q(L)=\top(e_q(L)),
\quad
\tilde f_q(L)=\soc(f_q(L)),
\quad\tilde e_q(0)=\tilde f_q(0)=0.$$
Now, for each $n$ we choose the parameters of
$H(\Gamma_n)$ in the following way
\begin{equation}\label{h}
h=-1/m,\quad h_p=(s_{p+1}-s_p)/m,
\quad s_p\in\ZZ, \quad p\neq 0.\end{equation}
The following hypothesis is important for the rest of the paper :

{\centerline{\it from now on we'll always assume that $m>1$.}}

\noindent
The $\CC$-vector space $[\Oc(\Gamma)]$ is canonically isomorphic to
the {\it level $\ell$ Fock space
$\Fc_{m,\ell}^{(s)}$
associated with
the $\ell$-charge $s=(s_p)$},
see (\ref{chain}) below
for details. The latter is equipped with an integrable
representation of $\widetilde{\sen\len}_m$ of level $\ell$,
see Section \ref{sec:fockl} below.

\begin{prop}\label{prop:EF}
(a) The functors $e_q$, $f_q$ are exact and biadjoint.

(b) We have $E={}^\Oc\!\Res_n$ and $F={}^\Oc\!\Ind_n.$

(c) For $M\in\Oc(\Gamma_n)$ we have $E(M)=0$ (resp.~$F(M)=0$)
iff $E(L)=0$ (resp.~$F(L)=0$) for any constituent $L$ of $M$.

(d)
The operators $e_q$, $f_q$ equip $[\Oc(\Gamma)]$ with a representation of $\widehat{\sen\len}_m$
which is isomorphic via the map (\ref{chain}) to $\Fc^{(s)}_{m,\ell}$.

(e)
The tuple $(\Irr(\Oc(\Gamma)),\tilde e_q,\tilde f_q)$ has a crystal structure.
In particular, for $L,L'\in\Irr(\Oc(\Gamma))$
we have $\tilde e_q(L), \tilde f_q(L)\in\Irr(\Oc(\Gamma))\sqcup\{0\}$,
and $\tilde e_q(L)=L'$  if and only if  $\tilde f_q(L')=L$.

\end{prop}

\begin{proof}
Parts $(a),(b)$ follows from \cite[prop.~4.4]{S},
part $(e)$ is contained in \cite[thm.~6.3]{S},
part $(c)$ is obvious,
and part $(d)$ is \cite[cor.~4.5]{S}.
\end{proof}

\vskip3mm

\subsection{The filtration of $[\Oc(\Gamma_n)]$ by the support}
\label{sec:stratum}
Fix a positive integer $n$.
Assume that $\ell\geqslant 2$.
In this section we consider the tautological action of
$\Gamma_n$ on $\CC^n$.
For an integer $l\geqslant 0$ and a composition $\nu$
such that $l+|\nu|\leqslant n$ we abbreviate
$X^\circ_{l,\nu}=X^\circ_{W,\hen}$ and
$X_{l,\nu}={X}_{W,\hen}$
where $W=\Gamma_{l,\nu}$.
If $\nu=(m^j)$ for some integer $j\geqslant 0$ such that $l+jm\leqslant n$
we write
$$X^\circ_{l,j}=X^\circ_{l,\nu},\quad {X}_{l,j}={X}_{l,\nu}.$$
Therefore $X_{l,j}$ is the set of the
points in $\CC^n$ with $l$ coordinates equal to zero
and $j$ collections of $m$ coordinates which differ
from each other by $\ell$-th roots of one.
To avoid confusions we may write $X_{l,j,\CC^n}=X_{l,j}$.
Unless specified otherwise, for $l,j,m,n$ as above we'll set
\begin{equation}\label{i} i=n-l-jm. \end{equation}

\vskip3mm

\begin{df}
For $i,j\geqslant 0$
we set
$$\Irr(\Oc(\Gamma_n))_{i,j}=
\{L\in\Irr(\Oc(\Gamma_n))\,;\,\Supp(L)=X_{l,j}\}.$$
\end{df}

\vskip3mm

\begin{df}
For $i,j\geqslant 0$
let $F_{i,j}(\Gamma_n)$
be the $\CC$-vector subspace of $[\Oc(\Gamma_n)]$
spanned by the classes of the modules
whose support is contained in $X_{l,j}$, with $l$ as in (\ref{i}).  If $i<0$ or $j<0$ we write
$F_{i,j}(\Gamma_n)=0.$
\end{df}

\vskip3mm

\begin{df}
We define a partial order on the set of pairs of nonnegative integers
$(i,j)$ such that $i+jm\leqslant n$ given by
$(i',j')\leqslant (i,j)$ if and only if $X_{l',j'}\subset X_{l,j}$,
where $l=n-i-jm$ and $l'=n-i'-j'm$.
\end{df}

\vskip3mm

\noindent
Since the support of a module is the union of the supports
of all its constituents,
the $\CC$-vector space $F_{i,j}(\Gamma_n)$ is
spanned by the classes of the modules in
$\Irr(\Oc(\Gamma_n))$ whose support is contained in $X_{l,j}$, or, equivalently
$F_{i,j}(\Gamma_n)$ is spanned by the classes of the modules in
$$\bigcup_{(i',j')\leqslant (i,j)}\Irr(\Oc(\Gamma_n))_{i',j'}.$$

\begin{rk}
\label{rk:support2}
We have
$\bigcup_{i,j}F_{i,j}(\Gamma_n)=[\Oc(\Gamma_n)].$
Indeed, for $L\in\Irr(\Oc(\Gamma_n))$ we have
$\Supp(L)=X_{l,\nu}$ for some $l,\nu$,
see Section \ref{sec:support1}.
For $b\in X_{l,\nu}^\circ$
the $H(\Gamma_{l,\nu})$-module
$\Res_b(L)$ is finite dimensional.
Thus, since the parameter $h$ of
$H(\Gamma_{l,\nu})$
is equal to $-1/m$
the parts of $\nu$ are all equal to $m$.
Hence we have $\Supp(L)=X_{l,j}$ for some  $l,j$ as above.
\end{rk}

\vskip3mm

\noindent
The subspaces $F_{i,j}(\Gamma_n)$
give a filtration of $[\Oc(\Gamma_n)]$.
Consider the associated graded $\CC$-vector space
$$\gr(\Gamma_n)=\bigoplus_{i,j}\gr_{i,j}(\Gamma_n).$$
Note that 
the $\CC$-vector spaces $F_{i,j}(\Gamma_n)$
and $\gr_{i,j}(\Gamma_n)$
differ slightly from the corresponding objects, denoted by
$\Fb_{i,j}K_0$ and $\gr_{i,j}^\Fb K_0$, in
\cite[sec.~6.5]{E}.The images by the canonical projection
$F_{i,j}(\Gamma_n)\to\gr_{i,j}(\Gamma_n)$
of the classes of the  modules in
$\Irr(\Oc(\Gamma_n))_{i,j}$
form a basis of the $\CC$-vector space $\gr_{i,j}(\Gamma_n)$.
So we may regard $\gr_{i,j}(\Gamma_n)$ as the subspace of
$[\Oc(\Gamma_n)]$
spanned by
$\Irr(\Oc(\Gamma_n))_{i,j}$.
We'll abbreviate
$$\gathered
F_{i,\bullet}(\Gamma_n)=\sum_jF_{i,j}(\Gamma_n),\quad
F_{\bullet,j}(\Gamma_n)=\sum_iF_{i,j}(\Gamma_n),\cr
\gr_{i,\bullet}(\Gamma_n)=\bigoplus_j\gr_{i,j}(\Gamma_n),\quad
\gr_{\bullet,j}(\Gamma_n)=\bigoplus_i\gr_{i,j}(\Gamma_n).
\endgathered$$
Now, let us study the filtration of $[\Oc(\Gamma_n)]$ in details.
The subgroup $\Gamma_{l,(m^j)}$ of $\Gamma_n$
is contained in the subgroups $\Gamma_{l+1,(m^j)}$,
$\Gamma_{l,(m^{j+1})}$ and $\Gamma_{l+m,(m^{j-1})}$
(up to conjugation by an element of $\Gamma_n$) whenever such subgroups exist. Thus we have
the inclusions
$$\gathered
X_{l+1,j},\,X_{l,j+1},\,X_{l+m,j-1}\subset X_{l,j},\cr
F_{i-1,j}(\Gamma_n),\, F_{i-m,j+1}(\Gamma_n),\,F_{i,j-1}(\Gamma_n)
\subset F_{i,j}(\Gamma_n).
\endgathered$$

\begin{prop}\label{prop:order}
(a) We have
$$X_{l',j'}\subsetneq X_{l,j}\iff
X_{l',j'}\subset X_{l+1,j}\cup X_{l,j+1}\cup X_{l+m,j-1}.$$
(b) We have an isomorphism of $\CC$-vector spaces
$$\gr_{i,j}(\Gamma_n)=F_{i,j}(\Gamma_n)/
\bigl(F_{i-1,j}(\Gamma_n)+F_{i-m,j+1}(\Gamma_n)+F_{i,j-1}(\Gamma_n)\bigr).$$
\end{prop}

\begin{proof}
First we prove $(a)$. Recall that $X_{l,j}$ is the set of the points in $\CC^n$ with $l$ coordinates equal to zero
and $j$ collections of $m$ coordinates which differ
from each other by $\ell$-th roots of one.
Therefore we have
\begin{equation}\label{bizarre}X_{l',j'}\subset
X_{l,j}\iff i-i'\geqslant \max\bigl(0,(j'-j)m\bigr).
\end{equation}
In particular this inclusion implies that $l'\geqslant l$.
We must prove that
$$X_{l',j'}\subsetneq X_{l,j}\Rightarrow X_{l',j'}\subset X_{l+1,j}\cup X_{l,j+1}\cup X_{l+m,j-1}.$$

First, assume that $l'=l$.
Since $X_{l',j'}\subsetneq X_{l,j}$ we have  $i> i'$. Then (\ref{i}) implies that
$i-i'=(j'-j)m$, hence that $j'>j$ and $i-i'\geqslant m$.
So $i-i'\geqslant\max(m,(j'-j)m)$,
and (\ref{bizarre})
implies that $X_{l',j'}\subset X_{l,j+1}$.

Next, assume that $l+m>l'>l$. Since $X_{l',j'}\subset X_{l,j}$ we have  $i\geqslant i'$.
Further (\ref{i}) implies that
$i-i'>(j'-j)m$ and $i'-i>(j-j'-1)m$.
Thus $i\geqslant i'$ implies indeed that $i>i'$ and $j'\geqslant j$.
So $i-1-i'\geqslant\max(0,(j'-j)m)$,
and (\ref{bizarre}) implies that $X_{l',j'}\subset X_{l+1,j}$.

Finally, assume that $l'\geqslant l+m$.
Since $X_{l',j'}\subset X_{l,j}$ we have  $i\geqslant i'$.
Further (\ref{i}) implies that
$i-i'\geqslant(j'-j+1)m$.
So $i-i'\geqslant\max(0,(j'-j+1)m)$,
and (\ref{bizarre}) implies that $X_{l',j'}\subset X_{l+m,j-1}$.

Part $(b)$ is a consequence of $(a)$ and of the definition of
the filtration on $[\Oc(\Gamma_n)]$.
\end{proof}

\vskip3mm

\begin{rk}
The sets $X_{l+1,j}$, $X_{l,j+1}$, $X_{l+m,j-1}$ do not contain each other. Indeed, the variety $X_{l,j}$ has the dimension $i+j$.
Thus the codimension of $X_{l+1,j}$, $X_{l,j+1}$, $X_{l+m,j-1}$ in
$X_{l,j}$ are $1,m-1,1$ respectively. However,
since a point in $X_{l,j+1}^\circ$ has only $l$ coordinates equal to 0, we have
$X_{l,j+1}\not\subset X_{l+1,j}$ and
$X_{l,j+1}\not\subset X_{l+m,j-1}.$
\end{rk}

\vskip3mm

\begin{rk} We have $F_{\bullet, 0}(\Gamma_n)=[\Oc(\Gamma_n)]$, because
$(i,j)\leqslant (i+jm, 0)$.
\end{rk}

\vskip3mm

\begin{rk}
\label{rk:support0j}
We have $(i',j')\leqslant (0,j)$ if and only if $i'=0$ and $j'\leqslant j$.
\end{rk}

\vskip3mm

\begin{rk}
\label{rk:support3}
Consider the set
$$F_{i,j}(\Gamma_n)^\circ=F_{i,j}(\Gamma_n)\setminus
\bigl(F_{i-1,j}(\Gamma_n)+F_{i-m,j+1}(\Gamma_n)+F_{i,j-1}(\Gamma_n)\bigr).$$
For $L\in\Irr(\Oc(\Gamma_n))$, by Proposition \ref{prop:order} and Remark \ref{rk:support2} we have
$$\aligned
\  [L] \in F_{i,j}(\Gamma_n)^\circ
&\iff\Supp(L)=X_{l,j}\cr
&\iff L\in\Irr(\Oc(\Gamma_n))_{i,j}.
\endaligned$$
\end{rk}

\vskip3mm

\begin{rk}\label{rk:fdim}
A representation is finite dimensional if and only if its support is zero.
Thus $\Irr(\Oc(\Gamma_n))_{0,0}$
is the set of isomorphism classes of finite dimensional modules in $\Oc(\Gamma_n)$.
Note that $(0,0)\leqslant (i,j)$ for all $(i,j)$.
\end{rk}

\vskip3mm

\begin{rk}\label{rk:3.14} If $\ell=1$ then, by Remark \ref{rk:2.1}
and Section \ref{sec:3.2} we have
$\Oc(\Sen_n)=\Oc(\Sen_n,\CC_0^n)$.
For an integer $j\geqslant 0$  we set $X_j=X_{\Sen_m^j,\CC_0^n}$, i.e.,
$X_j$ is the set of the points in $\CC^n_0$ with $j$ collections of
$m$ equal coordinates. Then, we set $i=n-jm$ and the results of this
section extend in the obvious way. In particular, we have
$$X_{j'}\subset X_j\iff j'\geqslant j,
\quad X_{j'}\subsetneq X_j\iff X_{j'}\subset X_{j+1}.$$
\end{rk}

\vskip3mm

\begin{rk}
\label{rk:support4}
For $\l\in\Pc_r$, $r\geqslant 1$, the support of the module
$L_{m\l}\in\Irr(\Oc(\frak S_{mr}))$ is
$$\Supp(L_{m\l})=X_{\frak S_m^r,\CC^{mr}_0}.$$
Indeed, formula (\ref{formind}) below and Proposition
\ref{prop:indsup} imply that
$$\Supp(L_{m\l})\subset
\Supp\bigl({}^\Oc\!\Ind_{(m^r)}
(L_{(m)}^{\otimes r})\bigr)
=X_{\frak S_m^r,\CC^{mr}_0}.$$
Next, by Remark \ref{rk:support2} there is  $j=0,1,\dots,r$ such that
$$\Supp(L_{m\l})=X_{\frak S_m^j,\CC^{mr}_0}.$$
Finally the inclusion
$X_{\frak S_m^j,\CC^{mr}_0}\subset X_{\frak S_m^r,\CC^{mr}_0}$
implies that $j=r$ by Remark \ref{rk:3.14}.
Note that this equality follows also from the work of Wilcox \cite{W}.
\end{rk}

\subsection{The action of $E$, $F$ on the filtration}
Let $E$, $F$ denote the $\CC$-linear operators
on $[\Oc(\Gamma)]$ induced by the exact functors $E$, $F$.
Recall that the parameters of $H(\Gamma)$ are chosen as in
(\ref{h}).

\begin{prop}
\label{prop:voiciunlemme}
Let $L\in \Irr(\Oc(\Gamma_n))_{i,j}$ and $l=n-i-mj$.

(a) We have
$\Supp(F(L))=X_{l,j,\CC^{n+1}}$.

(b) We have $E(L)=0$ iff $i=0$. We have
$\Supp(E(L))=X_{l,j,\CC^{n-1}}$ if $i>0$.
\end{prop}

\begin{proof}
Recall that
$$\gathered
\Supp(L)=X_{l,j}=X_{l,j,\CC^n},\quad
E(L)={}^\Oc\!\Res_n(L),\quad F(L)={}^\Oc\!\Ind_n(L).
\endgathered$$
Thus by Proposition \ref{prop:ressup} we have
$E(L)=0$ iff $b_n\notin X_{l,j}$.
Since $m>1$ the definition of the stratum
$X_{l,j}$ in Section \ref{sec:stratum} shows that $b_n\notin X_{l,j}$
iff $i=0$. Now, assume that $i>0$.
Then $l+mj\leqslant n-1$, and Proposition \ref{prop:ressup} yields
$$\Supp(E(L))=\bigcup_WX_{W,\CC^{n-1}},$$
where $W$ runs over the parabolic subgroups of $\Gamma_{n-1}$ which are $\Gamma_n$-conjugate to
$\Gamma_{l,(m^j)}$ (inside the group $\Gamma_n$).
We claim that a subgroup $W\subset\Gamma_{n-1}$ as above is
$\Gamma_{n-1}$-conjugate to
$\Gamma_{l,(m^j)}$ (inside the group $\Gamma_{n-1}$).
Therefore, we have
$$\Supp(E(L))=X_{l,j,\CC^{n-1}}.$$
Indeed, fix $b'\in \CC^{n-1}$ such that $W=(\Gamma_{n-1})_{b'}$. For
$b=(b',z)$ with $z\in\CC$ generic we have $(\Gamma_n)_b=W$,
where $W$ is regarded as a subgroup of $\Gamma_n$ via the obvious inclusion $\Gamma_{n-1}\subset \Gamma_n$. Since $W$ is $\Gamma_n$-conjugate to
$\Gamma_{l,(m^j)}$, there is an element $g\in\Gamma_n$ such that
the first $l$ coordinates of $g(b)$ are 0, the next $mj$ ones consist of $j$
collections of $m$ coordinates which differ from each other by $\ell$-th roots of one, and the last $i$ coordinates of $g(b)$ are in generic position. We'll abbreviate
$$g(b)\in 0^l(m)^j*^i.$$ Since $z$ is generic it is taken by $g$ to one of the
coordinates of $g(b)$ in the
packet $*^i$. Composing $g$ by an appropriate reflection in $\Sen_n$ we get an element
$g'\in\Gamma_{n-1}$ such that
$$g'(b)=(g'(b'),z)\in 0^l(m)^j*^{i}.$$ Thus we have also
$$g'(b')\in 0^l(m)^j*^{i-1}.$$
This implies the claim.
Hence, we have $$\Supp(E(L))=X_{l,j,\CC^{n-1}}.$$
Finally, since
$\Supp(L)=X_{l,j,\CC^n}$, Proposition \ref{prop:indsup} implies that
$$\Supp(F(L))=X_{l,j,\CC^{n+1}}.$$
\end{proof}

\begin{cor}
\label{cor:kerE}
(a) We have
$E(F_{i,j}(\Gamma_n))\subset F_{i-1,j}(\Gamma_{n-1})$.
If $i\neq 0$ we have also
$E(F_{i,j}(\Gamma_n)^\circ)\subset F_{i-1,j}(\Gamma_{n-1})^\circ$.

(b)
For $M\in\Oc(\Gamma_n)$ with
$[M]\in F_{i,j}(\Gamma_n)^\circ$ we have $E([M])=0$ iff $i=0$.

(c) We have
$F(F_{i,j}(\Gamma_n))\subset F_{i+1,j}(\Gamma_{n+1})$
and
$F(F_{i,j}(\Gamma_n)^\circ)\subset F_{i+1,j}(\Gamma_{n+1})^\circ$.
\end{cor}

\begin{proof}
First, let $L\in\Irr(\Oc(\Gamma_n))$ with
$[L]\in F_{i,j}(\Gamma_n)$. Thus
$L\in\Irr(\Oc(\Gamma))_{i',j'}$ with $(i',j')\leqslant (i,j)$.
Proposition \ref{prop:voiciunlemme} yields
$$\Supp(F(L))=X_{l',j',\CC^{n+1}},\quad
\Supp\bigl(E(L)\bigr)=X_{l',j',\CC^{n-1}}\ \text{if}\ i'\neq 0.$$
Hence we have
$F([L])\in F_{i+1,j}(\Gamma_{n+1})$
and
$E([L])\in F_{i-1,j}(\Gamma_{n-1}).$
Part $(b)$ follows from Proposition \ref{prop:voiciunlemme} and
Remarks \ref{rk:support0j}, \ref{rk:support3}.
Part $(c)$ follows from Proposition \ref{prop:voiciunlemme}
and Remark \ref{rk:support3}.
The second part of $(a)$ follows from Proposition \ref{prop:voiciunlemme} and Remark \ref{rk:support3}.
\end{proof}

\vskip3mm

\begin{cor}\label{cor:keretilde}
Let $L\in\Irr(\Oc(\Gamma_n))_{i,j}$.

(a)
If $\tilde e_q(L)\neq 0$ then $\tilde e_q(L)\in\Irr(\Oc(\Gamma_{n-1}))_{i-1,j}$.

(b)
If $\tilde f_q(L)\neq 0$ then $\tilde f_q(L)\in\Irr(\Oc(\Gamma_{n+1}))_{i+1,j}$.
\end{cor}

\begin{proof}
Set $L'=\tilde e_q(L)$. Assume that $L'\neq 0$. By Proposition \ref{prop:EF}
we have
$$L'\in\Irr(\Oc(\Gamma_{n-1})),\quad
\tilde f_q(L')=L.$$ Next, since $L\in\Irr(\Oc(\Gamma))_{i,j}$ and
since $\tilde e_q(L)$ is a constituent of $E(L)$, we have
$[L']\in F_{i-1,j}(\Gamma_{n-1})$
by Corollary \ref{cor:kerE}. We must prove that
$[L']\in F_{i-1,j}(\Gamma_{n-1})^\circ$. If this is false then we have
$[L']\in F_{i',j'}(\Gamma_{n-1})$ with
$$(i',j')=(i-2,j),\,(i-m-1,j+1),\,(i-1,j-1).$$
Thus, since $\tilde f_q(L')$ is a constituent of $F(L')$,
by Corollary \ref{cor:kerE} we have
\begin{equation}\label{number}[L]\in \gr_{i,j}(\Gamma_{n})\cap F_{i'+1,j'}(\Gamma_{n}).
\end{equation}
Therefore (\ref{bizarre}) yields
$i'+1\geqslant i$, so $i'=i-1$ and $j'=j-1$.
So, applying (\ref{bizarre}) once again we get a contradiction with (\ref{number}).
This proves $(a)$. The proof of $(b)$ is similar.
\end{proof}

\begin{cor}\label{cor:kere}
(a) For $x\in[\Oc(\Gamma)]$ we have
$$\bigl(e_q(x)=0,\ \forall q=0,1,\dots, m-1\bigr) \iff
x\in F_{0,\bullet}(\Gamma).$$

(b) For $M\in\Oc(\Gamma)$ we have
$$
E(M)=0
\iff E([M])=0\iff [M] \in F_{0,\bullet}(\Gamma).
$$

(c)
The space $F_{0,\bullet}(\Gamma)$
is spanned by the set
$$\aligned
\{[L]\,;\,L\in\Irr(\Oc(\Gamma))_{0,\bullet}\}
&=\{[L]\,;\,L\in\Irr(\Oc(\Gamma)),\,E(L)=0\}\cr
&=\{[L]\,;\,L\in\Irr(\Oc(\Gamma)),\,\tilde e_q(L)=0,\ \forall q=0,1,\dots, m-1\}.
\endaligned$$
\end{cor}

\begin{proof}
For $x\in[\Oc(\Gamma)]$ we write $x=\sum_Lx_L[L]$
where $L$ runs over the set $\Irr(\Oc(\Gamma))$. By
\cite[lem.~6.1, prop.~6.2]{S}, for each $q$ we have
$$e_q(x)=0 \iff x_L=0\ \text{if}\ e_q([L])\neq 0.$$
Thus the $\CC$-vector space
$$\bigl\{x\in[\Oc(\Gamma)]\,;\,e_q(x)=0,\ \forall q=0,1,\dots, m-1\bigr\}$$
is spanned by the classes of the simples modules $L$ such that
$e_q([L])=0$ for all $q=0,1,\dots, m-1$.
Then, apply Corollary \ref{cor:kerE}.
This proves $(a)$. Parts $(b)$, $(c)$ are obvious.
Note that 
$$\tilde e_q(L)=0,\ \forall q \iff e_q(L)=0,\ \forall q,$$
because a non zero finitely generated module has a non zero top.
\end{proof}

\vskip3mm

\section{The Fock space}


From now on we'll abbreviate
$$R(\frak S)=\bigoplus_{n\geqslant 0}[\Rep(\CC \frak S_n)],\quad
R(\Gamma)=\bigoplus_{n\geqslant 0}[\Rep(\CC \Gamma_n)].$$

\vskip3mm

\subsection{The Hopf $\CC$-algebra $\Sym$}
\label{sec:4.1}
This section and the following one are reminders on symmetric functions and the Heisenberg algebra.
First, recall that the $\CC$-vector space $R(\frak S)$
is identified with the $\CC$-vector space of symmetric functions
$$\Sym=\CC[x_1,x_2,\dots]^{\frak S_\infty}$$
via the {\it characteristic map}
\cite[chap.~I]{M}
$$\ch:R(\frak S)\to \Sym.$$
The map $\ch$ intertwines the induction/restriction in $R(\frak S)$
with the multiplication/comultiplication in $\Sym$.
It takes the class of the simple module $\bar L_\l$
to the Schur function $S_\l$
for each $\l\in\Pc$. The power sum polynomials are given by
$$P_\l=P_{\l_1}P_{\l_2}\dots,\quad
P_r=\sum_ix_i^r,\quad P_0=1,\quad\l\in\Pc,\quad r>0.$$
We equip the $\CC$-vector space $\Sym$ with the level $1$ action of
$\widehat{\sen\len}_m$ given by
\begin{equation}\label{ef1}e_q(S_\l)=\sum_\nu S_\nu,\quad
f_q(S_\l)=\sum_\mu S_\mu,
\quad q=0,\dots,m-1,
\end{equation}
where $\nu$ (resp. $\mu$) runs through all partitions obtained from
$\l\in\Pc$ by removing (resp. adding) a node of content $q$ mod $m$.
We equip $\Sym$ with the symmetric bilinear form
such that the Schur functions form an orthonormal basis.
The operators $e_q$, $f_q$ are adjoint to each other for this pairing.

\vskip3mm

\subsection{The Heisenberg algebra}
\label{sec:heisenberg}
The {\it Heisenberg algebra} is the Lie algebra
$\Hen$ spanned by the elements $\oneb$ and
$b_r$, $b'_r$, $r>0$, satisfying the following relations
$$[b'_r,b'_{s}]=[b_r,b_{s}]=0,\quad
[b'_r,b_{s}]=r\oneb\delta_{r,s},\quad r,s>0.$$
Let $U(\Hen$) be the enveloping algebra of $\Hen$, and
let $U^-(\Hen)\subset U(\Hen)$ be the
subalgebra generated by the elements $b_r$ with $r>0$.
Write $U^-(\Hen)_r$ for the subspace of
$U^-(\Hen)$ spanned by the monomials
$b_{r_1}b_{r_2}\cdots$ with $\sum_ir_i=r$.
For $\l\in\Pc$ and $f\in \Sym$ we consider the following elements in $U(\Hen)$
$$\gathered
b_\l=b_{\l_1}b_{\l_2}\cdots,\quad
b'_{\l}=b'_{\l_1}b'_{\l_2}\cdots,\cr
b_f=\sum_{\l\in\Pc}z_\l^{-1}\langle P_\l,f\rangle\,b_{\l},\quad
b'_f=\sum_{\l\in\Pc}z_\l^{-1}\langle P_\l,f\rangle\,b'_{\l},
\endgathered$$
where $z_\lambda$ is as in (\ref{zlambda}).
For any integer $\ell$ we can equip $\Sym$ with the level $\ell$ action of $\Hen$ such that
$b_{r}$ acts by multiplication by $\ell P_{r}$ and $b'_r$ acts by
$r\partial/\partial_{P_{r}}$ for $r>0$.
The operators $b_r$, $b'_{r}$ are adjoint to each other for the pairing
on $\Sym$ introduced in Section \ref{sec:4.1}.
Further, 
they commute with the $\widehat{\sen\len}_m$-action in (\ref{ef1}), see e.g.,
\cite{U}.
We write $V^\Hen_\ell=\Sym$ regarded as a level $\ell$ module of $\Hen$.
Consider the Casimir operator
\begin{equation}
\label{casimir0}
\partial={1\over\ell}\sum_{r\geqslant 1}b_{r}b'_{r}.
\end{equation}
To avoid any confusion we may call it also the 
{\it level $\ell$ Casimir operator}.
This formal sum defines a
diagonalisable $\CC$-linear operator on $V^\Hen_\ell$ such that
$$[\partial,b_r]=rb_r,\quad [\partial,b'_r]=-rb'_r.$$
Below, we'll equip $\Sym$ with the $\Hen$-action of level 1, 
i.e., we'll identify $\Sym=V^\Hen_1$,
unless mentioning explicitly the contrary.

\vskip3mm

\subsection{The Lie algebras $\widehat{\gen\len}_m$ and $\widetilde{\gen\len}_m$}
\label{sec:extremal}
We define the Lie algebra $\widehat{\gen\len}_m$ in the same way as $\widehat{\sen\len}_m$,
with $\gen\len_m$ instead of $\sen\len_m$.
We'll also use the extended affine Lie algebra
$\widetilde{\gen\len}_m,$
obtained by adding to
$\widehat{\gen\len}_m$
the 1-dimensional vector space spanned by the scaling element
$D$ such that
$[D,x\otimes\varpi^r]=r\,x\otimes\varpi^r$
and $[D,\oneb]=0$.
The Lie algebra
\begin{equation}\label{prodlie}
\bigl(\widehat{\sen\len}_m\times\Hen\bigr)/
\bigl(m(\oneb,0)-(0,\oneb)\bigr).
\end{equation}
is isomorphic to $\widehat{\gen\len}_m$ via the obvious map, which
takes the element $b'_r$ to $\sum_{p=1}^m e_{pp}\otimes\varpi^r$
and the element $b_r$ to $\sum_{p=1}^m e_{pp}\otimes\varpi^{-r}$
for each $r>0$. Unless specified otherwise,
by a $\widehat{\gen\len}_m$-module we'll always mean
a module over the Lie algebra (\ref{prodlie}), i.e., a
$\widehat{\sen\len}_m$-module
with a compatible $\Hen$-action.
Similarly, a $\widetilde{\gen\len}_m$-module we'll always mean a $\widehat{\gen\len}_m$-module
with a scaling operator $D$ such that
$$[D,x\otimes\varpi^r]=r\,x\otimes\varpi^r,\quad
[D,b_r]=-rb_r,\quad [D,b'_r]=rb'_r.$$
By a dominant integral weight
of  $\widehat{\gen\len}_m$, $\widetilde{\gen\len}_m$  we'll always mean
a dominant integral weight of  $\widehat{\sen\len}_m$, $\widetilde{\sen\len}_m$.
We denote the sets of such weights by
$P^{\widehat{\gen\len}_m}_+$, $P^{\widetilde{\gen\len}_m}_+$ or by
$P^{\widehat{\sen\len}_m}_+$, $P^{\widetilde{\sen\len}_m}_+$.
For $\l\in P^{\widetilde{\sen\len}_m}_+$ let $V_{\l}^{\widetilde{\sen\len}_m}$
and $V_{\l}^{\widetilde{\gen\len}_m}$
be the irreducible integrable modules over $\widetilde{\sen\len}_m$, $\widetilde{\gen\len}_m$
with the highest weight $\l$.
As a $\widetilde{\gen\len}_m$-module we have
$$V_{\omega_0}^{\widetilde{\gen\len}_m}=
V_{\omega_0}^{\widetilde{\sen\len}_m}\otimes V^\Hen_m.$$
Let $Q^{\sen\len_m}$, $P^{\sen\len_m}$ be the root lattice and weight lattice
of $\sen\len_m$. The weights of the module $V_{\o_0}^{\tilde{\sen\len}_m}$
are  all the weights of the form
$$\mu=\o_0+\beta-{1\over 2}\langle\beta,\beta\rangle\delta-i\delta,\quad
\beta\in Q^{\sen\len_m},\quad i\geqslant 0.$$
Among those, the {\it extremal weights} are the weights for which $i=0$.
The set of the extremal weights  coincide with the set of the
{\it maximal weights}, i.e., with the set of the weights
$\mu$ such that $\mu+\delta$ is not a weight of $V_{\o_0}^{\tilde{\sen\len}_m}$.
A weight $\mu$ of $V_{\o_0}^{\tilde{\sen\len}_m}$ is extremal if and only if
$$\langle\mu,\mu\rangle=0.$$
Note also that we have $\langle\mu,\mu\rangle=-2i$
if and only if $\mu+i\delta$ is an extremal weight.
See e.g., \cite[sec.~20.3, 20.5]{C} for details.
Now, let $T_m$ be the standard maximal torus in
$\SL_m$, and let $\ten_m$ be its Lie algebra. Let  $\widehat\Sen_m$ be the affine symmetric group.
It is the semidirect product
$\Sen_m\ltimes Q^{\sen\len_m}$. Note that $Q^{\sen\len_m}$ is the group of cocharacters of $T_m$.
We'll regard it as a lattice in $\ten_m$ in the usual way, and we'll
identify $\ten_m$ with $\ten_m^*$ via the standard invariant pairing
on $\ten_m$. The $\widehat\Sen_m$-action on
$\ten^*_m\oplus\CC\o_0\oplus\CC\delta$, see
e.g., \cite[sec.~13.1]{Ku}, is such that
the element $\beta$ in $Q^{\sen\len_m}$ acts via the operator
\begin{equation}\label{formule78997}
\xi_\beta:\ \mu\mapsto\mu+\mu(\oneb)\beta-\bigl(\langle\mu,\beta\rangle+{1\over 2}\langle\beta,\beta\rangle
\mu(\oneb)\bigr)\delta.
\end{equation}
In particular, we have
$$\xi_\beta(\o_0)=\o_0+\beta-{1\over 2}\langle\beta,\beta\rangle\delta.$$
We'll use the same notation for the $\widehat\Sen_m$-action on
$\ten^*_m\oplus\CC\o_0\oplus\CC\delta$
and on
$\ten^*_m\oplus\CC\o_0$, hoping it will not create any confusion.
Therefore, for $\l\in\ten^*_m\oplus\CC\o_0$ the symbol
$\xi_\beta(\l)$ will denote both the weight (\ref{formule78997}) and the weight
$\mu+\mu(\oneb)\beta$.
We can view the cocharacter $\beta\in Q^{\sen\len_m}$ as a group-scheme
homomorphism $\GG_m\to T_m$.
Thus the image $\beta(\varpi)$ of the element $\varpi\in K$
lies in $T_m(K)$.
For any $\widetilde{\sen\len}_m$-module $V$ let
$V[\mu]$, $\mu\in P^{\widetilde{\sen\len}_m},$
be the corresponding weight subspace in $V$.
Since the coadjoint action of
$\beta(\varpi)$ on
$\ten^*_m\oplus\CC\o_0\oplus\CC\delta$ is given by $\xi_\beta^{-1}$,
see e.g., \cite{PS},
we have also
\begin{equation}\label{formule78998}
\beta(V[\mu])=V[\xi_{\beta}^{-1}(\mu)]
\end{equation}
if $V$ is integrable.

\subsection{The Hopf $\CC$-algebra $\Sym_\Gamma$}
Now, let us consider the Hopf $\CC$-algebras $R(\Gamma)$.
Once again, the multiplication/comultiplication on
$R(\Gamma)$ is given by the induction/restriction.
We equip $R(\Gamma)$ with the symmetric $\CC$-bilinear form
given by
$$\langle f,g\rangle={|\Gamma_n|}^{-1}\sum_{x\in\Gamma_n}f(x)g(x^{-1}),
\quad f,g\in[\Rep(\CC \Gamma_n)].$$
Here we regard $f,g$ as characters of $\CC\Gamma_n.$
This bilinear form is a Hopf pairing.
Next, we consider the Hopf $\CC$-algebra
$\Sym_\Gamma=\Sym^{\otimes\Gamma}$.
We'll use the following elements in $\Sym_\Gamma$
$$f^\gamma=1\otimes\cdots \otimes 1\otimes f\otimes 1\otimes\dots\otimes 1,\quad
f\in \Sym,\quad\g\in\Gamma,$$
with $f$ at the $\gamma$-th place.
We abbreviate
$$P_\mu^\gamma=(P_\mu)^{\gamma},\quad
P_\l=\prod_{\gamma\in\Gamma}P_{\l(\gamma)}^\gamma,\quad
\mu\in\Pc,\quad\l\in\Pc^\Gamma.$$
The comultiplication in $\Sym_\Gamma$ is characterized by
$$\Delta(P_r^\g)=
P_r^\g\otimes 1+
1\otimes P_r^\g,\quad r>0,\quad \g\in\Gamma.$$
Following
\cite[chap.~I, app.~B, (7.1)]{M} we write
$$P_{r,p}=\ell^{-1}\sum_{\gamma\in\Gamma}\g^{p}P_r^\gamma,
\quad r\geqslant 0,\quad p\in\ZZ_\ell.$$
We equip $\Sym_\Gamma$ with the Hopf pairing such that
$$\langle P_{r,p},P_{s,q}\rangle=r\delta_{p,q}\delta_{r,s},\quad
r,s>0,\quad p,q\in\ZZ_\ell.$$
We may regard $P_{r,p}$, $r>0$, as the $r$-th
power sum of a new sequence of variables
$x_{i,p}$, $i>0$. We define the
following elements in $\Sym_\Gamma$
\begin{equation}\label{SG}S_{\mu,p}=S_\mu(x_{i,p}),\quad
S_\l=\prod_{p\in\ZZ_\ell}S_{\l(p),p},\quad\mu\in\Pc,\quad\l\in\Pc^\ell.
\end{equation}
The Hopf $\CC$-algebras
$R(\Gamma)$
and $\Sym_\Gamma$
are identified via the characteristic map
\cite[chap.~I, app.~B, (6.2)]{M}
$$\ch:R(\Gamma)\to\Sym_\Gamma.$$
This map intertwines the induction in
$R(\Gamma)$ with the multiplication in $\Sym_\Gamma$
by \cite[chap.~I, app.~B, (6.3)]{M}.
By \cite[chap.~I, app.~B, (9.4)]{M} and (\ref{labelsimple})
we have
\begin{equation}\label{ch}
\ch(\bar L_\l)=S_{\tau\l},\quad \l\in\Pc^\ell,\end{equation}
where $\tau$ is the permutation of $\Pc^\ell$ such that
$(\tau\l)(p)=\l(p+1)$ for each $p\in\ZZ_\ell$.
For $\l\in\Pc^\Gamma$ we write
$$z_\l=\prod_{\gamma\in\Gamma}z_{\l(\gamma)}\ell^{\,l(\l(\gamma))},$$
where $z_{\lambda(\gamma)}$ is as in (\ref{zlambda}),
and we define $\bar\l\in\Pc^\Gamma$
by $\bar\l(\gamma)=\l(\gamma^{-1})$.
Then we have
\begin{equation}\label{pairing}
\begin{split}
\langle S_\l,S_{\mu}\rangle=\delta_{\l,\mu},\quad \l,\mu\in\Pc^\ell,\cr
\langle P_\l,P_{\bar\mu}\rangle=\delta_{\l,\mu}z_\l,\quad \l,\mu\in\Pc^\Gamma.
\end{split}
\end{equation}
The first equality is proved as in
\cite[chap.~I, app.~B, (7.4)]{M}, while the second one is
\cite[chap.~I, app.~B, (5.3')]{M}.
By (\ref{ch}), (\ref{pairing})
the map $\ch$ is an isometry.
Thus it intertwines the restriction in
$R(\Gamma)$ with the comultiplication in $\Sym_\Gamma$.

\begin{prop} \label{GammaS}
(a)
The restriction $\Rep(\CC\Gamma_n)\to\Rep(\CC \frak S_n)$
yields the $\CC$-algebra homomorphism
$\Res_{\frak S}^\Gamma:\Sym_\Gamma\to\Sym$
such that $S_\l\mapsto\prod_pS_{\l(p)}$,
$P_{r,p}\mapsto P_r$.

(b)
The induction $\Rep(\CC \frak S_n)\to\Rep(\CC\Gamma_n)$
yields the $\CC$-algebra homomorphism
$\Ind_{\frak S}^\Gamma:\Sym\to \Sym_\Gamma$
such that $P_r\mapsto P_r^1=\sum_{p\in\ZZ_\ell}P_{r,p}.$
\end{prop}

\begin{proof}
The first part of $(a)$ is easy by 
Section \ref{sec:rep}, and it is left to the reader. 
For the second one, observe that
$$
\ch(\sigma_{r,p})=P_{r,p},\quad
r>0,$$
where $\sigma_{r,p}$ is the class function on $\Gamma_r$ wich takes the
value $r(\gamma_1\gamma_2\cdots\gamma_r)^p$ on pairs
$(w,(\gamma_1,\gamma_2,\dots,\gamma_r))$ such that
$w$ is a $r$-cycle,
and 0 elsewhere,
see \cite[lem.~5.1]{FJW}.
Now we concentrate on $(b)$. Note that
$$\Res^\Gamma_{\frak S}(P_0^\gamma)=1,\quad
\Res^\Gamma_{\frak S}(P_r^\gamma)=\ell\delta_{\gamma,1}P_r,\quad r>0.$$
Therefore, for $\l\in\Pc^\Gamma$ we have
$$\Res^\Gamma_{\frak S}(P_{\l})
=\prod_{\gamma\in\Gamma}\Res^\Gamma_{\frak S}(P_{\l(\gamma)}^\gamma)
=\begin{cases}
\ell^{\,l(\l(1))}P_{\l(1)}&\text{if}\ \lambda(\gamma)=\emptyset\
\text{for}\ \g\neq 1,\cr
0&\text{else}.
\end{cases}$$
If $f, g\in[\Rep(\CC\Gamma_n)]$ are the characters of finite dimensional $\Gamma_n$-modules $V$, $W$,
then $\langle f,g\rangle$ is the dimension of the space of $\CC\Gamma_n$-linear maps $V\to W$.
Hence, by Frobenius reciprocity the operator $\Ind_{\frak S}^\Gamma$ is adjoint to the operator
$\Res_{\frak S}^\Gamma$. Thus,
$$\langle\Ind_{\frak S}^\Gamma(P_r),P_\l\rangle
=\begin{cases}
r\ell^{\,l(\l(1))}\delta_{\l(1),(r)}&\text{if}\ \lambda(\gamma)=\emptyset\ \text{for}\ \g\neq 1,\cr
0&\text{else}.
\end{cases}$$
This implies that
$\Ind_{\frak S}^\Gamma(P_r)=aP_r^1$ for some $a$.
To determine $a$ let $\l$ be such that
$\lambda(\gamma)=\emptyset$ if $\gamma\neq 1$ and $\lambda(1)=(r)$.
Then we have
$$P_\l=P_r^1,\quad\langle P_\l,P_\l\rangle=r\ell.$$
This implies that $a=1.$
\end{proof}

\vskip3mm

\begin{rk}
Let $f\mapsto\bar f$ be the $\CC$-antilinear involution of
$\Sym_\Gamma$ which fixes the
$P_\l$'s with $\l\in\Pc^\Gamma$,
see \cite[chap.~I, app.~B, (5.2)]{M}.
For $\l\in\Pc^\ell$ let $\bar\l$ be the $\ell$-partition
given by $\bar\l(p)=\l(-p)$.
We have
$$\bar P_{r,p}=P_{r,-p},\quad \bar S_\l=S_{\bar\l},\quad
r>0,\quad p\in\ZZ_\ell,\quad\l\in\Pc^\ell.$$
\end{rk}

\vskip3mm

\begin{rk}
Setting $\ell=1$ in $\Sym_\Gamma$ we get the standard Hopf algebra structure and
the Hopf pairing of $\Sym$.
\end{rk}

\vskip3mm

\begin{rk}
We have
\cite[chap.~I, app.~B, (7.1')]{M}
$$P_r^\gamma=\sum_{p\in\ZZ_\ell}\gamma^{-p}P_{r,p},\quad r\geqslant 0,\quad
P_0^\gamma=1,\quad P_{0,p}=\delta_{0,p}.$$
\end{rk}

\vspace{3mm}

\subsection{The level 1 Fock space}\label{sec:fock1}
Fix once for all a basis $(\eps_1,\dots,\eps_m)$ of $\CC^m$.
The {\it level 1 Fock space of $\widehat{\sen\len}_m$} is the space
$\Fc_{m}$
of semi-infinite wedges of the $\CC$-vector space
$V_m=\CC^m\otimes\CC[t,t^{-1}]$.
More precisely, we have
$${\textstyle \Fc_m=
\bigoplus_{d\in\ZZ}\Fc^{(d)}_m},$$
where $\Fc^{(d)}_m$
is the subspace spanned by the semi-infinite wedges of 
{\it charge $d$}, i.e., the semi-infinite wedges of
the form
\begin{equation}
\label{wedge}
u_{i_1}\wedge u_{i_2}\wedge\cdots,
\quad i_1>i_2>\dots,
\quad u_{i-jm}=\eps_i\otimes t^j,
\end{equation}
where $i_k=d-k+1$ if $k\gg 0$.
We write
\begin{equation}\label{standardbasis}|\l,d\rangle=u_{i_1}\wedge u_{i_2}\wedge\cdots,\quad
\l\in\Pc,\quad i_k=\l_k+d-k+1,
\quad k>0.
\end{equation}
The elements $|\l,d\rangle$ with $\l\in\Pc$
form a basis of $\Fc^{(d)}_m.$
We equip $\Fc^{(d)}_m$ with the $\CC$-bilinear symmetric 
form such that this basis is orthonormal.

The Fock space $\Fc_m^{(d)}$ is equipped with a level one representation of
$\widehat{\sen\len}_m$ in the following way. 
First, the $\CC$-vector space $V_m$
is given the level 0 action of $\widehat{\sen\len}_m$ induced by the
homomorphism
\begin{equation}\label{eval}
\widehat{\sen\len}_m\to\sen\len_m\otimes\CC[t,t^{-1}],\quad \oneb\mapsto 0,
\quad x\otimes\varpi\mapsto x\otimes t
\end{equation}
and the obvious actions of
$\sen\len_m$ and $\CC[t,t^{-1}]$ on $V_m$.
Then, taking semi-infinite wedges, this action yields a level 1 action of 
$\widehat{\sen\len}_m$ on $\Fc^{(d)}_m$, 
see e.g., \cite{U}.

Next, observe that the multiplication by $t^r$, $r>0$, 
yields an endomorphism of $V_m$.
Taking semi-infinite wedges it yields a linear operator $b_{r}$ 
on $\Fc^{(d)}_m$.
Let $b'_r$ be the adjoint of $b_r$.
Then $b'_r$, $b_{r}$ define a level $m$ action of  $\Hen$ on  $\Fc^{(d)}_m$.
The $\widehat{\sen\len}_m$-action and the $\Hen$-action on
$\Fc^{(d)}_m$ glue together, yielding a level 1  representation of
$\widehat{\gen\len}_m$ on $\Fc^{(d)}_m$, see \cite{U} again.

As a $\widehat{\gen\len}_m$-module we have a canonical isomorphism
$$\Fc^{(d)}_m=V_{\omega_{d\,\text{mod}\,m}}^{\widehat{\gen\len}_m}.$$
It identifies the symmetric bilinear form of
$\Fc^{(d)}_m$ with the Shapovalov form on 
$V_{\omega_{d\,\text{mod}\,m}}^{\widehat{\gen\len}_m},$ i.e., 
with the unique (up to a scalar) symmetric bilinear form 
such that the adjoint of 
$b_{r}$, $e_{q}$ are $b'_r$, $f_q$ respectively.

\begin{rk}
The $\CC$-linear isomorphism
\begin{equation}\label{focksym}
{\textstyle \Fc^{(d)}_m\to \Sym,\quad
|\l,d\rangle\mapsto S_\l,\quad\l\in\Pc}
\end{equation}
takes the operators
$b'_r$, $b_{r}$, $e_{q}$, $f_{q}$ on the left hand side to the operators
$b'_{mr}$, $b_{mr}$, $e_{q-d}$, $f_{q-d}$ on the right hand side.
\end{rk}

\subsection{The level $\ell$ Fock space}\label{sec:fockl}
Fix a basis $(\eps_1,\dots,\eps_m)$ of $\CC^m$ and a basis
$(\dot\eps_1,\dots,\dot\eps_\ell)$ of $\CC^\ell$.
The {\it level $\ell$ Fock space of $\widehat{\sen\len}_m$} is the $\CC$-vector space
$${\textstyle\Fc_{m,\ell}=
\bigoplus_{d\in\ZZ}\Fc^{(d)}_{m,\ell}}$$
of semi-infinite wedges of the $\CC$-vector space $V_{m,\ell}=\CC^m\otimes\CC^\ell\otimes\CC[z,z^{-1}]$.
The latter are defined as in
(\ref{wedge})
with
\begin{equation}
\label{u}
u_{i+(j-1)m-km\ell}=\eps_i\otimes\dot\eps_j\otimes z^k.
\end{equation}
Here $i=1,\dots,m$, $j=1,\dots,\ell$, and $k\in\ZZ$.
We define basis elements $|\l,d\rangle$, with $\l\in\Pc$,
of $\Fc^{(d)}_{m,\ell}$ as in (\ref{standardbasis}), using the semi-infinite wedges above.
We equip $\Fc^{(d)}_{m,\ell}$ with the $\CC$-bilinear symmetric 
form such that the basis elements $|\l,d\rangle$ are orthonormal.
This yields a $\CC$-linear isomorphism
\begin{equation}\label{focksymell}
{\textstyle \Fc^{(d)}_{m,\ell}\to \Sym,\quad
|\l,d\rangle\mapsto S_\l},\quad\l\in\Pc.
\end{equation}
We equip the $\CC$-vector space $\Fc^{(d)}_{m,\ell}$
with the following actions, see \cite{U} for details :
\begin{itemize}

\item The level $m\ell$ action of $\Hen$ such that
$b'_r$, $b_r$ is taken to the operator
$b'_{m\ell r}$, $b_{m\ell r}$ on $\Sym$ under the
isomorphism (\ref{focksymell}) for $r>0$.

\item The level $\ell$ action of $\widehat{\sen\len}_m$
defined as follows : equip the $\CC[z,z^{-1}]$-module $V_{m,\ell}$
with the level 0 action of $\widehat{\sen\len}_m$ given by the
evaluation homomorphism (\ref{eval}) and the obvious actions of
$\sen\len_m$ and $\CC[z,z^{-1}]$ on $V_{m,\ell}$.
Taking semi-infinite wedges we get a level $\ell$ action of $\widehat{\sen\len}_m$
on $\Fc^{(d)}_{m,\ell}$.

\item The level $m$ action of $\widehat{\sen\len}_\ell$ which is defined
as above by exchanging the role of $m$ and $\ell$.
\end{itemize}

\noindent The actions of $\Hen$, $\widehat{\sen\len}_m$ and
$\widehat{\sen\len}_\ell$ commute with each other.
We call {\it $\ell$-charge of weight $d$} an $\ell$-tuple of integers
$s=(s_p)$ such that $d=\sum_ps_p$.
Set
\begin{equation}\label{weight}\hat\g(s,m)=(m-s_1+s_\ell)\,\omega_0+
\sum_{p=1}^{\ell-1}(s_p-s_{p+1})\omega_p.
\end{equation}
The Fock space associated with the $\ell$-charge $s$ is the subspace
\begin{equation}\label{coco}\Fc_{m,\ell}^{(s)}=\Fc^{(d)}_{m,\ell}[\hat\g(s,m)]
\end{equation}
consisting of the elements of weight $\hat\g(s,m)$
with respect to the $\widehat{\sen\len}_\ell$-action.
It is an $\widehat{\sen\len}_m\times\Hen$-submodule of $\Fc^{(d)}_{m,\ell}$.
Consider the basis elements $|\l,s\rangle$, $\l\in\Pc^\ell$, of
$\Fc^{(s)}_{m,\ell}$ defined in \cite[sec.~4.1]{U}.
The representation of $\widehat{\sen\len}_m$
on $\Fc^{(s)}_{m,\ell}$ can be characterized in the following way, see e.g.,
\cite{JMMO}, \cite{U},
\begin{equation}\label{ef}e_q|\l,s\rangle=\sum_\nu |\nu,s\rangle,\quad
f_q|\l,s\rangle=\sum_\mu |\mu,s\rangle,
\end{equation}
where $\nu$ (resp.~$\mu$) runs through all $\ell$-partitions obtained
by removing (resp. adding) a node
of coordinate $(i,j)$ in the
$p$-th partition of $\l$ such that $q=s_p+j-i$ modulo $m$.
Consider the $\CC$-vector space isomorphism
\begin{equation}\label{focksyml}
\Sym_\Gamma\to\Fc_{m,\ell}^{(s)},\quad
S_{\tau\l}\mapsto |\l,s\rangle,\quad
\l\in\Pc^\ell.
\end{equation}
By \cite[sec.~4.1]{U} we have an equality of sets
\begin{equation}\label{rk:4.7}
\{|\l,s\rangle\,;\,\l\in\Pc^\ell,\,s=(s_p)\in\ZZ^\ell,\,\sum_ps_p=d\}=
\{|\l,d\rangle\,;\,\l\in\Pc\}.
\end{equation}
Thus the elements $|\l,s\rangle$ form an orthonormal basis
of $\Fc^{(d)}_{m,\ell}$ and the map (\ref{focksyml}) preserves the pairings
by (\ref{pairing}).
The representation of $\Hen$
on $\Fc^{(s)}_{m,\ell}$ can be characterized in the following way.

\begin{prop}
\label{prop:fockl}
The operators $b'_r$, $b_{r}$, $r>0$,
on $\Fc^{(s)}_{m,\ell}$
are adjoint to each other.
Further $b_{r}$ acts  as the multiplication by the element
$P_{mr}^1=\sum_{p}P_{mr,p}$ of $\Sym_\Gamma$ under the isomorphism
(\ref{focksyml}).
\end{prop}

\begin{proof}
The first claim is \cite[prop.~5.8]{U}. To prove the second one,
observe that the formulas in \cite[sec.~4.1, 4.3 and (25)]{U} imply that
the $\CC$-linear map
$$\Fc^{(s)}_{m,\ell}\to\bigotimes_{p\in\ZZ_\ell}\Fc_m^{(s_p)},\quad
|\l,s\rangle\mapsto\bigotimes_{p\in\ZZ_\ell}|\l(p),s_p\rangle,$$
intertwines the operator $b_r$ on the left hand side and the operator
$$b_r\otimes 1\otimes\cdots\otimes 1+
1\otimes b_r\otimes 1\otimes\cdots\otimes 1+\cdots+
1\otimes\cdots\otimes 1\otimes b_r$$
on the right hand side.
Thus the proposition follows from the definition of
the $\Hen$-action on $\Fc_m$ in Section \ref{sec:fock1}
and from the definition of
the $\Hen$-action on $\Sym$ in Section \ref{sec:heisenberg}.
\end{proof}

\vskip3mm

\begin{rk}\label{rk:scaling}
The $\widehat{\sen\len}_m$-action on $\Fc^{(s)}_{m,\ell}$ can be extended to
an $\widetilde{\sen\len}_m$-action such that
the weight of $|\l,s\rangle$ is
$$-\Delta(s,m)\delta+\sum_{p=1}^\ell\o_{s_p}-\sum_{q=0}^{m-1}n_q(\l)\a_q,$$
see \cite[sec.~4.2]{U}.
Here $n_q(\l)$ is the number of $q$-nodes in $\l$, i.e., it is the sum over all $p$'s of the number of
nodes of coordinate $(i,j)$ in
the $p$-th partition of $\l$ such that $s_p+j-i=q\,\text{mod}\, m$.
We have also used the notation
$$\Delta(s,m)={1\over 2}\sum_{p=1}^\ell\langle\o_{s_p\,\text{mod}\,m},\,\o_{s_p\,\text{mod}\,m}\rangle
+{1\over 2}\sum_{p=1}^\ell s_p(s_p/m-1).$$In particular, we have
$$D(|\l,s\rangle)= -(\Delta(s,m)+n_0(\l))\,|\l,s\rangle.$$
\end{rk}

\vskip3mm

\subsection{Comparison of the $\widetilde{\gen\len}_\ell$-modules
$\Fc_{m,\ell}^{(0)}$ and $V_{\omega_0}^{\widetilde{\gen\len}_\ell}$}
\label{sec:fock}
The Fock space
$\Fc_{m,\ell}$
can be equipped with a level 1 representation of
$\widetilde{\gen\len}_\ell$
in the following way. The assignment
$$z\mapsto t^m,\quad\epsilon_i\mapsto t^{1-i},\quad
i=1,2,\dots,m,$$
yields a $\CC$-linear isomorphism
\begin{equation}\label{isomspace}\gathered
V_{m,\ell}=
\CC^m\otimes\CC^\ell\otimes\CC[z,z^{-1}]\to
\CC^\ell\otimes\CC[t,t^{-1}]=V_\ell,\cr
u_{i+(j-1)m-km\ell}\mapsto u_{j+(i-1)\ell-km\ell},
\endgathered
\end{equation}
see (\ref{wedge}), (\ref{u}).
Taking semi-infinite wedges, it yields a $\CC$-linear isomorphism
\begin{equation}\label{isomfock}\Fc_{m,\ell}\to\Fc_\ell.\end{equation}
Pulling back the representation of $\widetilde{\gen\len}_\ell$ on
$\Fc_{\ell}$ in Section \ref{sec:fock1}
and Remark \ref{rk:scaling}
by (\ref{isomfock}) we get a level 1 action of $\widetilde{\gen\len}_\ell$
on $\Fc_{m,\ell}$ such that :

\begin{itemize}

\item For $d\in\ZZ$ the level 1 representation of
$\widetilde{\gen\len}_\ell$ on $\Fc_{m,\ell}$ yields an isomorphism
\begin{equation}\label{coucou}\Fc_{m,\ell}^{(d)}=V_{\omega_{d\,\text{mod}\,\ell}}^{\widetilde{\gen\len}_\ell}.
\end{equation}

\item The level $m$-action of $\widehat{\gen\len}_\ell$ in $\Fc_{m,\ell}$ given in Section \ref{sec:fockl}
can be recovered from the level 1 action by composing it with the Lie algebra
homomorphism
\begin{equation}\label{multm}
\widehat{\gen\len}_\ell\to\widehat{\gen\len}_\ell,\quad
x\otimes\varpi^{r}\mapsto x\otimes\varpi^{mr},\,
\quad\oneb\mapsto m\oneb.
\end{equation}

\item
Pulling back the level $\ell$ representation of $\Hen$ on
$\Fc_{\ell}$ in Section \ref{sec:fock1}
by (\ref{isomfock}) we get a level $\ell$ action of $\Hen$
on $\Fc_{m,\ell}$.
The level $m\ell$-action of $\Hen$ in $\Fc_{m,\ell}$ given in Section
\ref{sec:fockl} can be recovered from the latter by composing it
with the Lie algebra homomorphism
\begin{equation}\label{twoheisenberg}
b_r\mapsto b_{mr},\quad
b'_r\mapsto b'_{mr},\quad
\oneb\mapsto m\oneb.
\end{equation}
Hence, the action of the level $m\ell$
Casimir operator, i.e., the operator
obtained by replacing $\ell$ by $m\ell$ in (\ref{casimir0}),
associated with the representation of $\Hen$ on $\Fc_{m,\ell}$
is the same as the action of the $m$-th Casimir operator 
\begin{equation}\label{mcasimir}
\partial_m={1\over m\ell}\sum_{r\geqslant 1}b_{mr}b'_{mr}
\end{equation}
associated with the
level $\ell$ representation of $\Hen$ on $\Fc_{m,\ell}$.

\item  To a partition $\l$ we associate
an {\it $\ell$-quotient $\l^*$}, an {\it $\ell$-core $\l^c$} and a
{\it content polynomial $c_\l(X)$} as in \cite[chap.~I]{M}.
In \cite[sec.~2.1]{LMi} a bijection $\tau$
is given from the set of $\ell$-cores to the set of
$\ell$-charges of weight 0.
By \cite[rem.~4.2$(i)$]{U} the inverse of the map (\ref{isomfock})
is such that
\begin{equation*}
\Fc_\ell^{(0)}\to\Fc_{m,\ell}^{(0)},\quad
|\l,0\rangle\mapsto |\l^*,\tau(\l^c)\rangle.
\end{equation*}
Now, the same argument as in \cite[ex.~I.11]{M} shows that
\begin{equation}\label{macdo0}
c_\l(X)=c_{\l^c}(X)\prod_{p=0}^{\ell-1}(X+p)^{|\l^*|}\
\text{mod}\ \ell.
\end{equation}
Further, by Remark \ref{rk:scaling} the scaling element $D$
of the level 1 representation of $\widetilde{\gen\len}_\ell$
on $\Fc_{m,\ell}^{(0)}$ is given by
\begin{equation}\label{macdo00}
D(|\l,0\rangle)=-n_0(\l)\,|\l,0\rangle,
\end{equation}
where $n_0(\l)$ is the number of 0-nodes in $\l$.
Thus we have the following relation
\begin{equation}\label{macdo1}[D,f_p]=-f_p,\quad\forall f_p\in\widehat{\sen\len}_m.
\end{equation}
Further we have
\begin{equation}\label{macdo2}
D(|\l,0\rangle)=-\bigl(n_0(\l^c)+|\l^*|\bigr)\,|\l,0\rangle.
\end{equation}
\end{itemize}

\vskip3mm

\begin{rk} We finish this section by several remarks concerning the Fock space
that we'll not use in the rest of the paper.
First, there is a tautological $\CC$-linear isomorphism
$\CC^m\otimes\CC^\ell=\CC^{m\ell}$. It yields $\CC$-linear isomorphisms
$V_{m,\ell}\to V_{m\ell}$ and
$\Fc_{m,\ell}\to\Fc_{m\ell}.$ 
Recall that $\Fc_{m\ell}$ 
is equipped with a level 1 action of
$\widehat{\sen\len}_{m\ell}$, and that 
$\Fc_{m,\ell}$ is equipped with a 
level $(\ell,m)$-action of 
$\widehat{\sen\len}_{m}\times\widehat{\sen\len}_{\ell}$.
Now, there is a well-known Lie algebra inclusion
$$(\widehat{\sen\len}_{m}\times\widehat{\sen\len}_{\ell})/(m(\oneb,0)-\ell(0,\oneb))
\subset\widehat{\sen\len}_{m\ell},\quad (\oneb,0)\mapsto \ell\oneb
,\quad (0,\oneb)\mapsto m\oneb.$$
This inclusion intertwines the 
$\widehat{\sen\len}_{m}\times\widehat{\sen\len}_{\ell}$-action on
$\Fc_{m,\ell}$ and the $\widehat{\sen\len}_{m\ell}$-action on 
$\Fc_{m,\ell}=\Fc_{m\ell}$.
Further, we want to compare 
the $\widehat{\sen\len}_{m\ell}$-action on $\Fc_{m,\ell}$ with the
level one $\widehat{\sen\len}_{\ell}$-action on $\Fc_{m,\ell}$
given in the begining of this section.
The $\CC$-linear isomorphisms
(\ref{isomspace}) and (\ref{isomfock})
yield a $\CC$-linear isomorphism
\begin{equation}\label{psi}
\Fc_{\ell}\to\Fc_{m,\ell}=\Fc_{m\ell}.
\end{equation}
The right hand side is equipped with a level 1 action of
$\widehat{\sen\len}_{m\ell}$, and the left hand side with a level 1 action of
$\widehat{\sen\len}_{\ell}$. Consider the following elements in
${\sen\len}_m\otimes\CC[\varpi,\varpi^{-1}]$
$$\gathered
x(i+km)=
\sum_{j=1}^{m-i}e_{j,i+j}\otimes \varpi^k+
\sum_{j=m-i+1}^me_{j,i+j-m}\otimes \varpi^{k+1},\cr
1\leqslant i\leqslant m,\quad k\in\ZZ.\endgathered$$
For $x\in{\sen\len}_m\otimes\CC[\varpi,\varpi^{-1}]$ and
$p,q=1,2,\dots,\ell$ we define the element
$x^{(p,q)}\in{\sen\len}_{m\ell}\otimes\CC[\varpi,\varpi^{-1}]$
by
$$x^{(p,q)}=\sum_{i,j=1}^me_{i+(p-1)m,\,j+(q-1)m}\otimes a_{i,j}
\quad\text{for}\quad
x=\sum_{i,j=1}^me_{i,j}\otimes a_{i,j}.$$
The following claim 
is proved by a direct computation which is left to the reader.

\begin{prop}\label{lem:lemmegrecque}
(a) There is a  Lie algebra inclusion
$\widehat{\sen\len}_\ell\subset\widehat{\sen\len}_{m\ell}$ given by
$$\oneb\mapsto\oneb,\quad
e_{p,q}\otimes\varpi^{r}\mapsto
x(r)^{(p,q)},\quad
p,q=1,2,\dots,\ell,\quad r\in\ZZ.$$


(b) The map (\ref{psi}) intertwines
the $\widehat{\sen\len}_\ell$-action on $\Fc_\ell$
and the $\widehat{\sen\len}_{m\ell}$-action on $\Fc_{m\ell}$.

\end{prop}
\end{rk}

\section{The categorification of the Heisenberg algebra}

We'll abbreviate
$$[\Oc(\Gamma)]=\bigoplus_{n\geqslant 0}[\Oc(\Gamma_n)].$$
Assume that $h$, $h_p$ are rational numbers as in (\ref{h}).
Thus $\Lambda$ is a rational weight
of $\widehat{\sen\len}_\ell$ of level 1.
Let $m$ be the denominator of $h$.
We'll assume that $m>2$.

\vskip3mm

\subsection{The functors $A_{\l,!}, A_\l^*, A_{\l,*}$ on $D^b(\Oc(\Gamma))$}
To simplify the exposition, from now on we'll assume that $\ell>1$.
All the statements below have an analoguous version for $\ell=1$,
by replacing everywhere $\CC^n$ by $\CC^n_0$.
Let $n,r$ be non-negative integers.
Consider the point
$$b_{n,r}=(0,\dots,0,1,\dots,1)\in\hen=\CC^{n+r},$$
with $x_i=0$ for $1\leqslant i\leqslant n$, and
$x_i=1$ for $n<i\leqslant n+r$.
The centralizer of $b_{n,r}$ in $\Gamma_{n+r}$ is the parabolic subgroup
$\Gamma_{n,r}$. We have
$$\hen/\hen^{\Gamma_{n,r}}=\CC^{n}\times\CC^r_0.$$
Here $\CC^{n}$ is the reflection representation of $\Gamma_{n}$
and $\CC^r_0$ is the reflection representation of $\Sen_{r}$.
Note that
$$\Oc(\Gamma_{n,r})=
\Oc(\Gamma_{n,r},\CC^{n}\times\CC^{r}_0),\quad
\Oc(\frak  S_{r})=\Oc(\frak S_{r},\CC^{r}_0).$$
In particular we have a canonical equivalence of categories
$$\Oc(\Gamma_{n,r})=\Oc(\Gamma_{n})\otimes\Oc(\frak S_{r}).$$
Thus the induction and restriction relative to $b_{n,r}$ yield functors
\begin{equation*}
\begin{split}
{}^\Oc\!\Ind_{n,r}:
\Oc(\Gamma_{n})\otimes\Oc(\frak S_{r})\to\Oc(\Gamma_{n+r}),
\cr
{}^\Oc\!\Res_{n,r}:
\Oc(\Gamma_{n+r})\to\Oc(\Gamma_{n})\otimes\Oc(\frak S_{r}).
\end{split}
\end{equation*}
Now consider the functors
${}^\Oc\!\Ind_{n,mr}$, ${}^\Oc\!\Res_{n,mr}$. The parameters of
$H(\Gamma_{n+mr})$ and $H(\Gamma_{n})$ are $h$, $\Lambda$.
The parameter of $H(\frak S_{mr})$ is $h$.
Fix a partition $\l\in\Pc_r$.
We define the functors
$$\gathered
\Oc(\Gamma_{n})\otimes\Oc(\frak S_{mr})\to \Oc(\Gamma_{n}),\cr
M\mapsto\Hom_{\Oc(\frak S_{mr})}(M,L_{m\l})^*,\quad
M\mapsto\Hom_{\Oc(\frak S_{mr})}(L_{m\l},M),
\endgathered$$
as the tensor product of the identity of
$\Oc(\Gamma_{n})$
and of the functors
$$\gathered\Oc(\frak S_{mr})\to\Rep(\CC),\cr
M\mapsto\Hom_{\Oc(\frak S_{mr})}(M,L_{m\l})^*,\quad
M\mapsto\Hom_{\Oc(\frak S_{mr})}(L_{m\l},M).\endgathered$$
Here the upperscript $*$ denotes the dual $\CC$-vector space.
We denote the corresponding derived functors in the following way
$$M\mapsto\RHom_{\Oc(\frak S_{mr})}(M,L_{m\l})^*,\quad
M\mapsto\RHom_{\Oc(\frak S_{mr})}(L_{m\l},M).$$

\begin{df}
For $\l\in\Pc_r$ with $r\geqslant 0$ we define the functors
$$\begin{aligned}
&A_{\l,!}&:D^b(\Oc(\Gamma_{n+mr}))\to D^b(\Oc(\Gamma_{n})),\quad
&M\mapsto\RHom_{D^b(\Oc(\frak S_{mr}))}({}^\Oc\!\Res_{n,mr}(M),L_{m\l})^*,\cr
&A_\l^*&:D^b(\Oc(\Gamma_{n}))\to D^b(\Oc(\Gamma_{n+mr})),\quad
&M\mapsto{}^\Oc\!\Ind_{n,mr}(M\otimes L_{m\l}),\cr
&A_{\l,*}&:D^b(\Oc(\Gamma_{n+mr}))\to D^b(\Oc(\Gamma_{n})),\quad
&M\mapsto\RHom_{D^b(\Oc(\frak S_{mr}))}(L_{m\l},{}^\Oc\!\Res_{n,mr}(M)).
\end{aligned}$$
\end{df}

\begin{prop} \label{prop:adjointtriple}
We have a triple of exact adjoint endofunctors
$(A_{\l,!},\,A_\l^*,\,A_{\l,*})$ of the triangulated category
$D^b(\Oc(\Gamma))$. For $M,N\in D^b(\Oc(\Gamma))$ we have
$$\begin{gathered}
\RHom_{D^b(\Oc(\Gamma))}(A_\l^*(M),N)=\RHom_{D^b(\Oc(\Gamma))}(M,
A_{\l,*}(N)),\cr
\RHom_{D^b(\Oc(\Gamma))}(A_{\l,!}(M),N)=\RHom_{D^b(\Oc(\Gamma))}(M,
A_\l^*(N)).
\end{gathered}$$
\end{prop}

\begin{proof} Obvious because the functors
${}^\Oc\!\Ind_{n,mr}$ and ${}^\Oc\!\Res_{n,mr}$ are exact and biadjoint,
see \cite{BE}, \cite{S}.
\end{proof}

\vskip3mm

\subsection{The $\Sen_r$-action on
$(A_!)^r$, $(A^*)^r$ and $(A_*)^r$.}
\label{sec:5.2}
For $\flat=!,*$
we write
$A^*=A^*_{(1)}$ and $A_\flat=A_{(1),\flat}$.
For $r\geqslant 1$,
the transitivity of the induction and restriction functors \cite[cor.~2.5]{S}
yield functor isomorphisms
\begin{equation}
\label{5.2}
\aligned
(A_!)^r&=\RHom_{D^b(\Oc(\Sen_{(m^r)}))}\bigl({}^\Oc\!\Res_{n,(m^r)}(\bullet),
L\bigr)^*,\cr
(A^*)^r&=
{}^\Oc\!\Ind_{n,(m^r)}(\bullet\otimes L)=
{}^\Oc\!\Ind_{n,mr}\bigl(\bullet\otimes
{}^\Oc\!\Ind_{(m^r)}(L)\bigr),\cr
(A_*)^r&=\RHom_{D^b(\Oc(\Sen_{(m^r)}))}\bigl(L,
{}^\Oc\!\Res_{n,(m^r)}(\bullet)\bigr).
\endaligned
\end{equation}
Here, to unburden the notation we abbreviate
$L=L_{(m)}^{\otimes r}$.
The goal of this section is to construct a $\Sen_r$-action on
$(A_!)^r$, $(A^*)^r$ and $(A_*)^r$, and to decompose these functors
using this action. To do this,
let $\Hb(\Gamma_{n,(m^r)})$, $\Hb(\Gamma_{n})$, $\Hb(\Sen_m)$
be as in Appendix \ref{app:A}, with the parameters $\zeta$ and $v_p$ as in
Section \ref{sec:KZ}.
There is an obvious isomorphism
$$\Hb(\Gamma_{n,(m^r)})=\Hb(\Gamma_{n})\otimes \Hb(\Sen_m)^{\otimes r}.$$
Let $\tau_i\in\Sen_{n+mr}$ be the unique permutation such that
\begin{itemize}
\item $\tau_i$ is minimal in the coset
$\Sen_{(n,m^r)}\tau_i\Sen_{(n,m^r)}$,
\item
$\tau_i(vw_1w_2\dots w_r)\tau_i^{-1}=vw_1\dots w_{i+1}w_i\dots w_r$
for $v\in\Sen_n$, $w_1,\dots,w_r\in \Sen_m$.
\end{itemize}
Let
$\tau_i$
denote also the algebra isomorphism 
$\Hb(\Gamma_{n,(m^r)})\to\Hb(\Gamma_{n,(m^r)})$
given by
$$\gathered
x\otimes y_1\otimes\dots \otimes y_r\to x\otimes y_1\otimes\dots
\otimes y_{i+1}\otimes y_i\otimes\dots \otimes y_r.
\endgathered$$
We have the following relation
in $\Hb(\Gamma_{n+mr})$
\begin{equation}
\label{taui}
T_{\tau_i}z=\tau_i(z)T_{\tau_i},
\quad z\in\Hb(\Gamma_{n,(m^r)}).
\end{equation}
Therefore, the element $T_{\tau_i}$ belongs to the normalizer
of $\Hb(\Gamma_{n,(m^r)})$ in $\Hb(\Gamma_{n+mr})$.
The twist of a module by $\tau_i$ yields the functor
$$\gathered
\tau_i:\Rep(\Hb(\Gamma_{n,(m^r)}))\to\Rep(\Hb(\Gamma_{n,(m^r)})),
\cr
M\otimes N_1\otimes\dots \otimes N_r\to M\otimes N_1\otimes\dots
\otimes N_{i+1}\otimes N_i\otimes\dots \otimes N_r.
\endgathered$$
We define the morphism of functors
$$\gathered
{}^\Hb\tau_i:
{}^\Hb\!\Ind_{n,(m^r)}\to{}^\Hb\!\Ind_{n,(m^r)}\circ\tau_i,\quad
{}^\Hb\tau_i(M)(h\otimes v)=hT_{\tau_i}\otimes \tau_i(v),\cr
h\in\Hb(\Gamma_{n+mr}),\quad
v\in M,\quad
M\in\Rep(\Hb(\Gamma_{n,(m^r)})).
\endgathered$$
It is well-defined by (\ref{taui}).
Next, the permutation $\tau_i$ yields also a functor
$$\gathered
\tau_i:\Oc(\Gamma_{n,(m^r)})\to\Oc(\Gamma_{n,(m^r)}),\cr
M\otimes N_1\otimes\dots \otimes N_r\to M\otimes N_1\otimes\dots
\otimes N_{i+1}\otimes N_i\otimes\dots \otimes N_r.
\endgathered$$
The functor $\KZ $
yields a $\CC$-algebra isomorphism \cite[lem.~2.4]{S}
\begin{equation}\label{kzh}\KZ:\End\bigl({}^\Oc\!\Ind_{n,(m^r)}\bigr)\to
\End\bigl(\KZ\circ{}^\Oc\!\Ind_{n,(m^r)}\bigr)=
\End\bigl({}^\Hb\!\Ind_{n,(m^r)}\circ\KZ\bigr).
\end{equation}
For the same reason we have also an isomorphism
$$\KZ :\Hom\bigl({}^\Oc\!\Ind_{n,(m^r)},
{}^\Oc\!\Ind_{n,(m^r)}\circ\tau_i\bigr)\to
\Hom\bigl({}^\Hb\!\Ind_{n,(m^r)}\circ\KZ,
{}^\Hb\!\Ind_{n,(m^r)}\circ\tau_i\circ\KZ\bigr).
$$
So there is a unique morphism of functors
$${}^\Oc\!\tau_i:{}^\Oc\!\Ind_{n,(m^r)}
\to{}^\Oc\!\Ind_{n,(m^r)}\circ\tau_i$$
which satisfies the following identity
\begin{equation}
\label{5.18}
\KZ({}^\Oc\!\tau_i(M))={}^\Hb\tau_i(\KZ(M)),
\quad
M\in\Oc(\Gamma_{n,(m^r)}).
\end{equation}
The functor
$\bullet\otimes L$
yields a map
\begin{equation}\label{map'}
\Hom\bigl({}^\Oc\!\Ind_{n,(m^r)},{}^\Oc\!\Ind_{n,(m^r)}
\circ\tau_i\bigr)\to\End((A^*)^r).
\end{equation}
Let  $\bar\tau_i$ denote
the image of ${}^\Oc\tau_i$
by this map.

\begin{lemma}
\label{lem:L5}
The following relations hold in $\End\bigl((A^*)^r\bigr)$
\begin{itemize}
\item $\bar\tau_i^2=1,$
\item
$\bar\tau_i\bar\tau_j=\bar\tau_j\bar\tau_i\ \text{if}\ j\neq i-1,i+1,$
\item
$\bar\tau_i\bar\tau_{i+1}\bar\tau_i=\bar\tau_{i+1}\bar\tau_{i}\bar\tau_{i+1}.$
\end{itemize}
\end{lemma}

\begin{proof}
We'll write $L^S=(L_{(m)}^S)^{\dot\otimes r}.$
Consider the  morphism of functors
\begin{equation}
\label{EQ1}
\gathered
{}^\Hb\tau_i^0:{}^\Hb\!\Ind_{(m^r)}\to{}^\Hb\!\Ind_{(m^r)}\circ\tau_i,
\quad
{}^\Hb\tau_i^0(M)(h\otimes v)=hT_{\tau_i}\otimes \tau_i(v),\cr
h\in\Hb(\Sen_{mr}),\quad
v\in M,\quad M\in\Rep(\Hb(\Sen_m)^{\otimes r}).
\endgathered
\end{equation}
It is well-defined by (\ref{taui}).
By (\ref{kzh}) there is a unique morphism of functors
$${}^\Oc\tau_i^0:{}^\Oc\!\Ind_{(m^r)}\to{}^\Oc\!\Ind_{(m^r)}\circ\tau_i$$
such that
\begin{equation}
\label{5.8}
\KZ({}^\Oc\tau_i^0(M))={}^\Hb\tau_i^0(\KZ(M)).
\end{equation}
We define the endomorphism $\bar\tau_i^0$ of the module
${}^\Oc\!\Ind_{(m^r)}(L)$ by
\begin{equation}
\label{EQ2}
\bar\tau_i^0={}^\Oc\tau_i^0(L).
\end{equation}
The transitivity of the induction functor
\cite[cor.~2.5]{S} yields
\begin{equation}
\label{5.29}
\gathered
(A^*)^r(M)={}^\Oc\!\Ind_{n,mr}
\bigl(M\otimes{}^\Oc\!\Ind_{(m^r)}(L)\bigr),\cr
\bar\tau_i(M)={}^\Oc\!\Ind_{n,mr}(\oneb\otimes\bar\tau_i^0).
\endgathered
\end{equation}
Therefore, we are reduced to check the following relations
\begin{itemize}
\item $(\bar\tau_i^0)^2=1,$
\item
$\bar\tau_i^0\bar\tau_j^0=\bar\tau_j^0\bar\tau_i^0\ \text{if}\ j\neq i-1,i+1,$
\item
$\bar\tau_i^0\bar\tau_{i+1}^0\bar\tau_i^0=\bar\tau_{i+1}^0\bar\tau_{i}^0\bar\tau_{i+1}^0.$
\end{itemize}
To prove this, recall that
Rouquier's functor $R$
yields an equivalence
\begin{equation}\label{R}
\Oc(\Sen_{mr})\to\Rep(\Sb_\zeta(mr)).\end{equation}
Here $\zeta$ is a primitive $m$-th root of 1.
We have
\begin{equation}\label{5.11bis}R(L_{m\l})=L^S_{m\l}.\end{equation}
By Proposition \ref{prop:tensorproduct} we have also
$$R\bigl({}^\Oc\!\Ind_{(m^r)}(L)\bigr)=L^S.$$
Thus the functor $R$ yields a $\CC$-algebra isomorphism
$$\End_{\Oc(\Sen_{mr})}\bigl({}^\Oc\!\Ind_{(m^r)}(L)\bigr)=
\End_{\Sb_\zeta(mr)}(L^S).$$
Therefore, we are reduced to check the following relations
in $\End_{\Sb_\zeta(mr)}(L^S)$
\begin{itemize}
\item $R(\bar\tau_i^0)^2=1,$
\item
$R(\bar\tau_i^0)R(\bar\tau_j^0)=
R(\bar\tau_j^0)R(\bar\tau_i^0)\ \text{if}\ j\neq i-1,i+1,$
\item
$R(\bar\tau_i^0)R(\bar\tau_{i+1}^0)R(\bar\tau_i^0)=
R(\bar\tau_{i+1}^0)R(\bar\tau_{i}^0)R(\bar\tau_{i+1}^0).$
\end{itemize}
By Proposition \ref{prop:tensorproduct} there is an isomorphism of functors
$\Oc(\Sen_{(m^r)})\to\Rep(\Sb_\zeta(mr))$
$$(\bullet)^{\dot\otimes r}\circ R=
R\circ{}^\Oc\!\Ind_{(m^r)}\circ(\bullet)^{\otimes r}.$$
Since $\oneb_R\,{}^\Oc\!\tau_i^0\,\oneb_{(\bullet)^{\otimes r}}$
is an endomorphism of the right hand side and since $R$ is an equivalence,
there is an unique endomorphism
${}^\Sb\tau_i^0$
of the functor
$$(\bullet)^{\dot\otimes r}:\Rep(\Sb_\zeta(m))\to\Rep(\Sb_\zeta(mr))$$
such that
\begin{equation}
\label{5.12}
{}^\Sb\tau_i^0\,\oneb_R=
\oneb_R\,{}^\Oc\!\tau_i^0\,\oneb_{(\bullet)^{\otimes r}}.
\end{equation}
Consider the diagram
$$\xymatrix{
\End\bigl({}^\Oc\!\Ind_{(m^r)}\circ(\bullet)^{\otimes r}\bigr)
\ar[r]^-\KZ\ar[d]_R&
\End
\bigl({}^\Hb\!\Ind_{(m^r)}\circ\KZ\circ(\bullet)^{\otimes r}\bigr)
\cr
\End\bigl((\bullet)^{\dot\otimes r}\circ R\bigr).
\ar[ur]_{\Phi^*}
}$$
The upper map is invertible by (\ref{kzh}), the vertical one by
Proposition \ref{prop:tensorproduct}, and the lower one by
Corollary \ref{cor:tensor*}. The diagram is commutative because
$\Phi^*\circ R=\KZ$. By (\ref{5.8}) and
(\ref{5.12}) the image of
${}^\Oc\!\tau_i^0\,\oneb_{(\bullet)^{\otimes r}}$ is given by
\begin{equation}\label{5.12bis}\begin{split}\xymatrix{
{}^\Oc\!\tau_i^0\oneb_{(\bullet)^{\otimes r}}
\ar@{|->}[r]\ar@{|->}[d]&
{}^\Hb\tau_i^0\,\oneb_{\KZ\circ(\bullet)^{\otimes r}}
\cr
{}^\Sb\tau_i^0\,\oneb_{R(\bullet)}.
\ar@{|->}[ur]
}\end{split}\end{equation}
Now, recall the endomorphisms of functors
$\Rc_{\bullet,i}$, $\Sc_{\bullet,i}$ defined in
(\ref{Rc}), (\ref{Sc}).
By Corollary \ref{cor:B7} the functor
$\Phi^*$ yields a map
$$\gathered
\End\bigl((\bullet)^{\dot\otimes r}\bigr)
\to
\End\bigl({}^\Hb\!\Ind_{(m^r)}\circ(\bullet)^{\otimes r}\circ\Phi^*\bigr)
,\quad
\Rc_{\bullet,i}\mapsto\Sc_{\Phi^*(\bullet),i}.
\endgathered
$$
By (\ref{EQ1}) we have
$$\Sc_{M,i}={}^\Hb\tau_i^0(M^{\otimes r}),\quad
M\in\Rep(\Hb(\Sen_m)).
$$
Therefore, by (\ref{5.12bis}) we have also
\begin{equation}
\label{EQ7}\Rc_{M,i}={}^\Sb\tau_i^0(M),\quad
M\in\Rep(\Sb_\zeta(m)).
\end{equation}
Now, by  (\ref{EQ2}), (\ref{5.11bis}) and (\ref{5.12}) we have
\begin{equation*}
\label{EQ6}
R(\bar\tau_i^0)={}^\Sb\tau_i^0\bigl(L_{(m)}^S\bigr).\end{equation*}
Thus, by (\ref{EQ7}) we must check that the operators
$\Rc_{L_{(m)}^S,i}$
satisfies the same relations as above.
The quantum Frobenius homomorphism yields a functor
$$\Fr^*:\Rep(\Sb_1(r))\to\Rep(\Sb_\zeta(mr))$$
such that
$L_{(m)}^S=\Fr^*(\bar L_{(1)}^S),$ see Section \ref{sec:B7}.
It is a braided tensor functor by Proposition \ref{prop:B8}.
Thus the claim follows from
Proposition \ref{prop:B8bis}.
\end{proof}

\vskip3mm

We can now prove the following, which is the main result of this section.

\begin{prop}
\label{prop:decomp}
Let $r\geqslant 1$.

(a) The group $\frak S_r$ acts on the functors $(A^*)^r$, $(A_*)^r$.

(b) We have the following $\Sen_r$-equivariant isomorphisms of functors
\begin{equation*}
(A^*)^r=\bigoplus_{\l\in\Pc_r}\bar L_\l\otimes A^*_{\l},\quad
(A_*)^r=\bigoplus_{\l\in\Pc_r}\bar L_\l\otimes A_{\l,*}.
\end{equation*}
\end{prop}

\begin{proof}
First, we concentrate on part $(a)$.
To unburden the notation we abbreviate
$$L=L_{(m)}^{\otimes r},\quad
L^S=(L_{(m)}^S)^{\dot\otimes r}.$$
By Lemma \ref{lem:L5} the assignment
$s_i\mapsto\bar\tau_i$ yields
a $\Sen_r$-action on
$(A^*)^r$.
Under the adjunction
$\bigl({}^\Oc\!\Ind_{n,(m^r)},{}^\Oc\!\Res_{n,(m^r)}\bigr)$
the isomorphism
${}^\Oc\!\tau_i$
yields a (right transposed) isomorphism of
${}^\Oc\!\Res_{n,(m^r)}$. We'll denote it by
${}^\Oc\!\tau_i$ again.
By definition of the right transposition, the following square is
commutative for $M\in\Oc(\Gamma_{n+mr})$
\begin{equation*}
\xymatrix{
\Hom_{\Oc(\Gamma_{n+mr})}
\bigl({}^\Oc\!\Ind_{n,(m^r)}(L),M\bigr)
\ar[r]^{\circ{}^\Oc\!\tau_i(L)}\ar@{=}[d]&
\Hom_{\Oc(\Gamma_{n+mr})}
\bigl({}^\Oc\!\Ind_{n,(m^r)}(L),M\bigr)
\ar@{=}[d]
\cr
\Hom_{\Oc(\Gamma_{n,(m^r)})}
\bigl(L,{}^\Oc\!\Res_{n,(m^r)}(M)\bigr)
\ar[r]^{{}^\Oc\!\tau_i(M)\circ}&
\Hom_{\Oc(\Gamma_{n,(m^r)})}
\bigl(L,{}^\Oc\!\Res_{n,(m^r)}(M)\bigr).
}
\end{equation*}
Here and in the rest of the proof, we use the canonical isomorphisms
$$\gathered
\Hom_{\Oc(\Gamma_{n+mr})}
\bigl({}^\Oc\!\Ind_{n,(m^r)}(L),M\bigr)=
\Hom_{\Oc(\Gamma_{n+mr})}
\bigl({}^\Oc\!\Ind_{n,(m^r)}(\tau_i(L)),M\bigr),\cr
\Hom_{\Oc(\Gamma_{n,(m^r)})}\bigl(L,{}^\Oc\!\Res_{n,(m^r)}(M)\bigr)=
\Hom_{\Oc(\Gamma_{n,(m^r)})}\bigl(L,\tau_i
({}^\Oc\!\Res_{n,(m^r)}(M))\bigr)
\endgathered$$
given by $\tau_i(L)=L$ without mentionning them explicitly.
Let $\langle\bullet,\bullet\rangle$ denote the canonical pairing
$$\RHom_{D^b(\Oc(\Sen_{(m^r)}))}(\bullet,L)^*
\times\RHom_{D^b(\Oc(\Sen_{(m^r)}))}(\bullet,L)\to\CC.$$
We define the $\Sen_r$-action on $(A_*)^r$ by
\begin{equation}
\label{5.9}
\gathered
s_i(f)={}^\Oc\!\tau_i(M)\circ f,\cr
f\in(A_*)^r(M)=\RHom_{D^b(\Oc(\Sen_{(m^r)}))}\bigl(L,
{}^\Oc\!\Res_{n,(m^r)}(M)\bigr).
\endgathered
\end{equation}
Note that the formulas (\ref{5.9}) do define an action of the group $\Sen_r$
by Lemma \ref{lem:L5}, because the square above is commutative.

Now, we prove part $(b)$.
It is convenient to rewrite the $\Sen_r$-action on $(A^*)^r$ in a
slightly different way.
Setting $n=0$ in the construction above we get a $\Sen_r$-action on
${}^\Oc\!\Ind_{(m^r)}(L)$
such that $s_i$ acts through the operator $\bar\tau_i^0$ in (\ref{EQ2}),
and by (\ref{5.29})
the reflection $s_i$ acts on $(A^*)^r$ through the automorphism
$${}^\Oc\!\Ind_{n,mr}(\oneb\otimes{}^\Oc\!\bar\tau_i^0).$$
We claim that the following
identity holds in $\Rep(\CC\Sen_r)\otimes\Oc(\frak S_{mr})$
\begin{equation}\label{formind}{}^\Oc\!\Ind_{(m^r)}(L)=
\bigoplus_{\l\in\Pc_r}\bar L_\l\otimes L_{m\l}.\end{equation}
To prove (\ref{formind}) we use
Rouquier's functor $R$
as in the proof of Lemma \ref{lem:L5}.
It is enough to check the following identity in
$\Rep(\CC\Sen_r)\otimes\Rep(\Sb_\zeta(mr))$
\begin{equation*}L^S=
\bigoplus_{\l\in\Pc_r}\bar L_\l\otimes L^S_{m\l}.\end{equation*}
To do that, note that by Proposition \ref{prop:B8}
the functor in Section \ref{sec:B7}
$$\Fr^*:\Rep(\Sb_1(r))=\Rep(\Sb_{(-1)^m}(r))\to\Rep(\Sb_\zeta(mr))$$
given by the quantum Frobenius homomorphism
is a braided tensor functor. Further we have
$$\Fr^*(\bar L^S_\l)=L^S_{m\l},\quad
\Fr^*((\bar L^S_{(1)})^{\dot\otimes r})=L^S,$$
where $\bar L^S_\l$ is the simple
$\Sb_1(r)$-module with the highest weight $\l$.
Therefore, to prove (\ref{formind}) we are reduced to check
the following identity in
$\Rep(\CC\frak S_r)\otimes\Rep(\Sb_1(r))$
\begin{equation*}(\bar L^S_{(1)})^{\dot\otimes r}=
\bigoplus_{\l\in\Pc_r}\bar L_\l\otimes\bar L^S_{\l}.\end{equation*}
This is a trivial consequence of the Schur duality.
The decomposition
\begin{equation}\label{p1r}
(A^*)^r=\bigoplus_{\l\in\Pc_r}\bar L_\l\otimes A_{\l}^*
\end{equation}
is a direct consequence of (\ref{formind}).
The decomposition of the functor
$(A_*)^r$ follows from (\ref{formind})
and the commutativity of the diagram above,
because it implies that the canonical isomorphism
$$(A_*)^r(M)=\RHom_{D^b(\Oc(\Sen_{(m^r)}))}\bigl(
{}^\Oc\!\Ind_{(m^r)}(L),{}^\Oc\!\Res_{n,mr}(M)\bigr)$$
is $\Sen_r$-equivariant.
\end{proof}

\vskip3mm

\begin{rk}
Using an adjunction 
$\bigl({}^\Oc\!\Res_{n,(m^r)},{}^\Oc\!\Ind_{n,(m^r)}\bigr)$
for each $r$,
we can construct in a similar way a $\Sen_r$-action on
the functor $(A_!)^r$ such that we have the decomposition
$$(A_!)^r=\bigoplus_{\l\in\Pc_r}\bar L_\l\otimes A_{\l,!}.$$
Then, by Propositions
\ref{prop:adjointtriple}
and
\ref{prop:decomp}
we have the triple 
$\bigl((A_!)^r,(A^*)^r,(A_*)^r\bigr)$
of adjoint $\Sen_r$-equivariant functors.
\end{rk}

\vskip3mm

\begin{rk}
We have used the hypothesis $m>2$ in the proof of
Proposition \ref{prop:decomp} when using Rouquier's functor $R$.
Probably this is not necessary.
\end{rk}

\vskip3mm

\begin{prop} \label{prop:shift} For $r\geqslant 1$
we have an isomorphism of functors
$$(A_{!})^r[2r(1-m)]=(A_{*})^r.$$
\end{prop}

\begin{proof}
Once again we'll abbreviate $L=L_{(m)}^{\otimes r}$.
Let $\Pc erv(\PP^{m-1})$ be the category of perverse sheaves on
$\PP^{m-1}$ which are constructible with respect to the standard stratification
$\PP^{m-1}=\CC^0\cup\CC^1\cup\cdots\cup\CC^{m-1}$.
By \cite[thm.~1.3]{BEG} the category $\Oc(\frak S_m)$ decomposes
as the direct sum of $\Pc erv(\PP^{m-1})$ and semisimple blocks.
Under this equivalence the module $L_{(m)}$ is taken to the perverse sheaf
$\CC_{\PP^{m-1}}[m-1]$. So, by Verdier duality \cite[(3.1.8)]{KS} we have an isomorphism of functors
$D^b(\Oc(\Sen_{m}))\to D^b(\CC)$
\begin{equation}\label{verdier}\RHom_{D^b(\Oc(\frak S_m))}(L_{(m)},\bullet)\to
\RHom_{D^b(\Oc(\frak S_m))}(\bullet,L_{(m)})^*[2(1-m)].
\end{equation}
The tensor power of (\ref{verdier}) is an isomorphism of functors
$D^b(\Oc(\Sen_{(m^r)}))\to D^b(\CC)$
$$\gathered
\theta^0:\RHom_{D^b(\Oc(\Sen_{(m^r)}))}(L,\bullet)\to
\RHom_{D^b(\Oc(\Sen_{(m^r)}))}(\bullet,L)^*[2r(1-m)].\cr
\endgathered$$
The group $\Sen_r$ acts on
$D^b(\Oc(\Sen_{(m^r)}))$
in such a way that the simple reflection $s_i$ acts
via the permutation functor
$$\gathered
\tau_i:\Oc(\Sen_{(m^r)})\to\Oc(\Sen_{(m^r)}),\quad
M_1\otimes\dots \otimes M_r\to M_1\otimes\dots
\otimes M_{i+1}\otimes M_i\otimes\dots \otimes M_r.
\endgathered$$
The isomorphism $\theta^0$ is $\Sen_r$-equivariant, i.e., we have
$$\gathered
\theta^0(\tau_i(M))(\tau_i(f))=
\tau_i(\theta^0(M)(f)),
\cr
M\in \Oc(\Sen_{(m^r)}), \quad
f\in\RHom_{D^b(\Oc(\Sen_{(m^r)}))}(L,M).
\endgathered$$
It yields an isomorphism of functors
$D^b(\Oc(\Gamma_{n,(m^r)}))\to D^b(\Oc(\Gamma_n))$
$$\gathered
\theta:\RHom_{D^b(\Oc(\Sen_{(m^r)}))}
(L,\bullet)\to
\RHom_{D^b(\Oc(\Sen_{(m^r)}))}(\bullet,L)^*
[2r(1-m)]\cr
\endgathered$$
such that
\begin{equation}
\label{5.5}\gathered
\theta(\tau_i(M))(\tau_i(f))=
\tau_i(\theta(M)(f)),
\cr
M\in \Oc(\Gamma_{n,(m^r)}), \quad
f\in\RHom_{D^b(\Oc(\Sen_{(m^r)}))}(L,M).
\endgathered
\end{equation}
We define an isomorphism of functors
$D^b(\Oc(\Gamma_{n+mr}))\to D^b(\Oc(\Gamma_n))$
by
$$\theta'=\theta\,\oneb_{{}^\Oc\!\Res_{n,(m^r)}}.$$
More precisely, we have
$$\gathered
\theta':
\RHom_{D^b(\Oc(\Sen_{(m^r)}))}\bigl(L,
{}^\Oc\!\Res_{n,(m^r)}(\bullet)\bigr)\to\cr
\to\RHom_{D^b(\Oc(\Sen_{(m^r)}))}\bigl({}^\Oc\!\Res_{n,(m^r)}(\bullet),
L\bigr)^*[2r(1-m)].
\endgathered$$
By (\ref{5.2}) we may view $\theta'$ as an isomorphism 
$(A_*)^r\to(A_!)^r[2r(1-m)].$
\end{proof}

\vskip3mm

\begin{rk}
Probably we can choose the $\Sen_r$-action on $(A_!)^r$ in such a way that
the isomorphism $(A_*)^r\to(A_!)^r[2r(1-m)]$ is $\Sen_r$-equivariant.
This would imply that for $\l\in\Pc_r$ we have
$A_{\l,!}[2r(1-m)]=A_{\l,*}.$
We'll not use this.
\end{rk}

\vskip3mm

\begin{rk}
The transitivity of the induction functor \cite[cor.~2.5]{S}
yields an isomorphism of functors $A_\l^*\,A^*_\mu=A^*_\mu\,A^*_\l$
for $\l,\mu\in\Pc.$
Taking the adjoint functors we get also the isomorphisms
$A_{\l,!}\,A_{\mu,!}=A_{\mu,!}\,A_{\l,!}$ and
$A_{\l,*}\,A_{\mu,*}=A_{\mu,*}\,A_{\l,*}.$
\end{rk}

\vskip3mm

\begin{rk}
The functors $A_{\l,!}$, $A_\l^*$, $A_{\l,*}$ yield linear endomorphisms
of the $\CC$-vector space
$[\Oc(\Gamma)].$
Let us denote them $A_{\l,!}$, $A_\l^*$, $A_{\l,*}$ again.
\end{rk}

\vskip3mm

\begin{rk}
\label{rk:triangle} 
Recall that $\langle m\rangle=\bigoplus_{i=0}^{m-1}\CC[-2i].$
For any object
$M$ of $D^b(\Oc(\Gamma))$ there should be a distinguished triangle
\begin{equation*}\label{triangle}
\xymatrix{
\ell\langle m\rangle M\ar[r]&
A_*A^*(M)\ar[r]&A^*A_*(M)\ar[r]^-{+1}&.
}
\end{equation*}
\end{rk}

\vskip3mm

\vskip3mm

\subsection{The functors $a^*_\l$, $a_{\l,*}$
on $\Oc(\Gamma)$ and the $\Hen$-action on the Fock space}

For $i\in\ZZ$ and $\flat=!,*$ we consider the endofunctor
$H^i(A_{\l,\flat})$ of $\Oc(\Gamma)$ given by
$$
H^i(A_{\l,\flat})(M)=H^i(A_{\l,\flat}(M)),\quad M\in\Oc(\Gamma).$$
From now on we'll write $Ra_{\l,\flat}=A_{\l,\flat}$ and
$R^ia_{\l,\flat}=H^i(A_{\l,\flat})$.

\begin{df}
Let $a_\l^*$ be the restriction of $A_\l^*$ to the Abelian category
$\Oc(\Gamma)$.
Since $a_\l^*$ is an exact endofunctor of $\Oc(\Gamma)$,
we may write $a_\l^*$ for $A_\l^*$ if it does not create any confusion.
We abbreviate $a_{\l,\flat}=R^0a_{\l,\flat}$.
The functor $a_{\l,*}$
is a left exact endofunctor of $\Oc(\Gamma)$,
while $a_{\l,!}$
is right exact.
\end{df}

\noindent
Consider the chain of $\CC$-linear isomorphisms
which is the composition of (\ref{groth}), of the characteristic map $\ch$, and of
(\ref{focksyml}),
\begin{equation}\label{chain}
\begin{matrix}
[\Oc(\Gamma)]&\to&R(\Gamma)&\to&\Sym_\Gamma&\to&
\Fc_{m,\ell}^{(s)},\cr
\Delta_\l&\mapsto&\bar L_\l&\mapsto& S_{\tau\l}&\mapsto&|\l,s\rangle.
\end{matrix}
\end{equation}
Recall that  symmetric bilinear form on $\Fc^{(s)}_{m,\ell}$ defined in
Section \ref{sec:fockl}.

\begin{prop}\label{prop:comparison}
(a)
The map (\ref{chain}) identifies the symmetric $\CC$-bilinear form on
$\Fc^{(s)}_{m,\ell}$ with the $\CC$-bilinear form
$$[\Oc(\Gamma)]\times[\Oc(\Gamma)]\to\CC,
\quad(M,N)\mapsto\sum_i(-1)^i\dim\Ext^i_{\Oc(\Gamma)}(M,N).$$

(b)
The map (\ref{chain}) identifies the operators $b_{S_\l}$, $b'_{S_\l}$ on
$\Fc_{m,\ell}^{(s)}$ with the operators $a_\l^*$, $Ra_{\l,*}$ on $[\Oc(\Gamma)]$.
\end{prop}

\begin{proof}
Part $(a)$ is obvious because we have
$$\dim\Ext^i_{\Oc(\Gamma_n)}(\Delta_\l,\nabla_\mu)=\delta_{i,0}\delta_{\l,\mu},\quad
[\Delta_\mu]=[\nabla_\mu],\quad \forall\l,\mu\in\Pc^\ell_n,$$
because $\Oc(\Gamma_n)$ is a quasi-hereditary category, see e.g.,
\cite[prop.~A.2.2]{D}.
Now we concentrate on $(b)$.
By $(a)$ and Proposition \ref{prop:adjointtriple},
the pairs $(b_{S_\l}, b'_{S_\l})$ and
$(a_\l^*,Ra_{\l,*})$ consist of adjoint linear operators on $\Fc^{(s)}_{m,\ell}$.
So it is enough to check that under (\ref{chain}) we have
the following equality
\begin{equation}\label{equality}b_{S_\l}=a_\l^*.\end{equation}
To do that, observe first that,
by Proposition \ref{GammaS},
for $r>0$ the map
$\ch:R(\Gamma)\to\Sym_\Gamma$
intertwines
the operator
$$R(\Gamma)\to R(\Gamma),\quad
M\mapsto
\Ind_{\Gamma\times\frak  S}^\Gamma\bigl(M\otimes\ch^{-1}(P_{mr})\bigr)$$
and the multiplication by $\sum_{p\in\ZZ_\ell}P_{mr,p}$.
Here we have abbreviated
$$\Ind_{\Gamma\times\frak  S}^{\Gamma}=
\bigoplus_{n,r\geqslant 0}
\Ind_{n,mr}.$$
Next,
by Proposition \ref{prop:fockl}, the map
$\Sym_\Gamma\to\Fc^{(s)}_{m,\ell}$ above intertwines
the multiplication by $\sum_{p\in\ZZ_\ell}P_{mr,p}$
and the operator $b_r$.
By definition, the plethysm with the power sum $P_m$ is
the $\CC$-algebra endomorphism
$$\psi^m:\ \Sym\to\Sym,\quad
f\mapsto\sum_{\lambda\in\Pc}
z_{\lambda}^{-1}\langle f,P_{\lambda}\rangle\,P_{m\lambda}.$$
The discussion above implies that the map
$R(\Gamma)\to\Fc^{(s)}_{m,\ell}$ above
identifies the action of $b_{S_\l}$ on $\Fc_{m,\ell}^{(s)}$
with the operator
$$R(\Gamma)\to R(\Gamma),\quad
M\mapsto\Ind_{\Gamma\times\frak S}^\Gamma(M\otimes\ch^{-1}\psi^m(S_{\l})).$$
Now, recall the maps
$$\spe:[\Rep(\CC \Gamma_n)]\to [\Oc(\Gamma_n)],\quad
\spe:[\Rep(\CC \frak S_{mr})]\to [\Oc(\frak S_{mr})].$$
By Lemma \ref{lem:indresgroth}, they  commute with the induction and
restriction.
We claim that
$$\spe\circ\ch^{-1}\circ\psi^m(S_\l)=L_{m\l}.$$
Thus (\ref{equality}) follows from (\ref{chain}).
To prove the claim,
set $\zeta$ equal to a primitive $m$-th root of 1.
Then Rouquier's functor yields an isomorphism, see (\ref{R}),
\begin{equation*}[\Oc(\frak S_{mr})]=[\Rep(\Sb_\zeta(mr))].
\end{equation*}
Next, the quantum Frobenius homomorphism
yields a commutative diagram
\begin{equation}\label{diag1}
\begin{split}
\xymatrix{
[\Rep(\Sb_1(r))]\ar[r]^{\Fr^*}&[\Rep(\Sb_\zeta(mr))]\cr
\Sym\ar@{=}[u]^\chi
\ar[ur]^-\psi\ar[r]^{\psi^m}
&\Sym\ar@{=}[u]_\chi
}
\end{split}
\end{equation}
where $\chi$ is the formal character, see e.g., \cite[sec.~II.H.9]{J}.
Consider the chain of maps
$$\theta\;:\ \xymatrix{
[\Rep(\CC \frak S_{mr})]\ar@{=}[r]^-{(\ref{groth})}
&[\Oc(\frak S_{mr})]\;
\ar@{=}[r]^-{(\ref{R})}
&[\Rep(\Sb_\zeta(mr))]}.
$$
We have
$$\psi(S_\l)=L_{m\l}^S,\quad
\theta\, \ch^{-1}(S_\mu)=\Delta_\mu^S,\quad
\l\in\Pc_r,\quad\mu\in\Pc_{mr}.$$
Thus we have
$$\chi(\theta\,\ch^{-1}(S_\mu))=\chi(\Delta_\mu^S)=S_\mu,\quad
\mu\in\Pc_{mr}.$$
Therefore we have also
$$\chi(\theta\circ\ch^{-1}\circ\psi^m(S_\l))=
\psi^m(S_\l)=
\chi(\psi(S_\l))=
\chi(L_{m\l}^S).$$
This implies that
$\theta\circ\ch^{-1}\circ\psi^m(S_\l)=
L_{m\l}^S,$ proving the claim and the proposition.
\end{proof}

\vskip3mm

\begin{rk}
It has been conjectured in \cite[sec.~6.6]{E} that the Shapovalov form on
$V_{\omega_{d\,\text{mod}\,\ell}}^{\widehat{\gen\len}_\ell}$  
should be related to the bilinear form on $[\Oc(\Gamma)]$ in
Proposition \ref{prop:comparison}. Recall that
$$\Fc^{(d)}_\ell=V_{\omega_{d\,\text{mod}\,\ell}}^{\widehat{\gen\len}_\ell}$$ 
and that the Shapovalov form on the right hand side is identified with the
symmetric bilinear form on the left hand side considered in Section
\ref{sec:fock1}. By (\ref{rk:4.7}), under the isomorphism
(\ref{isomfock}), the latter is identified with the bilinear form on
$\Fc_{m,\ell}^{(d)}$ in Section \ref{sec:fockl}. 
Thus Proposition \ref{prop:comparison} implies Etingof's conjecture.
\end{rk}

\vskip3mm

\begin{prop} \label{prop:biadjoint} Let $\l\in\Pc_r$ with $r\geqslant 0$.

(a) We have a triple of adjoint functors $(a_{\l,!},\,a_\l^*,\,a_{\l,*})$.


(b) For $\flat=*,!$, $q=0,1,\dots, m-1,$ and $i\geqslant 0$ there are isomorphisms of functors
$$e_q\,R^ia_{\l,\flat}=R^ia_{\l,\flat}\,e_q,\quad e_q\,a_\l^*=a_\l^*\,e_q,\quad
f_q\,R^ia_{\l,\flat}=R^ia_{\l,\flat}\,f_q,\quad f_q\,a_\l^*=a_\l^*\,f_q.$$
\end{prop}

\begin{proof}
By definition of the functors $A_{\l,*}$, $A_{\l,!}$  we have
$$A_{\l,*}(\Oc(\Gamma))\subset D^{\geqslant 0}(\Oc(\Gamma)),\quad
A_{\l,!}(\Oc(\Gamma))\subset D^{\leqslant 0}(\Oc(\Gamma)).$$
Thus, by Proposition \ref{prop:adjointtriple}
we have the triple of adjoint endofunctors
of $\Oc(\Gamma)$
$$(a_{\l,!},\,a_\l^*,\,a_{\l,*})=(H^{0}(A_{\l,!}),\,A_\l^*,\,H^0(A_{\l,*})).$$
This proves $(a)$. 
Next, let us prove part $(b)$. It is enough to give isomorphisms of functors
\begin{equation}\label{eg1}
e_q\,a_\l^*=a_\l^*\,e_q,\quad
f_q\,a_\l^*=a_\l^*\,f_q.
\end{equation}
First, observe that we have an isomorphism of functors
\begin{equation}\label{eg2}
F\,a_\l^*=a_\l^*\,F.
\end{equation}
Indeed, for $M\in\Oc(\Gamma_n)$
the transitivity of the induction functor \cite[cor.~2.5]{S} yields
$$\aligned
F\,a_\l^*(M)
&={}^\Oc\!\Ind_{n+mr}{}^\Oc\!\Ind_{n,mr}(M\otimes L_{m\l})\cr
&={}^\Oc\!\Ind_{\Gamma_{n+mr}}^{\Gamma_{n+mr+1}}
{}^\Oc\!\Ind_{\Gamma_{n,mr}}^{\Gamma_{n+mr}}
(M\otimes L_{m\l})\cr
&={}^\Oc\!\Ind^{\Gamma_{n+mr+1}}_{\Gamma_{n,mr}}
(M\otimes L_{m\l}),
\endaligned$$
$$\aligned
a_\l^*F(M)
&={}^\Oc\!\Ind_{n+1,mr}({}^\Oc\!\Ind_{n}(M)\otimes L_{m\l})\cr
&={}^\Oc\!\Ind^{\Gamma_{n+mr+1}}_{\Gamma_{n+1,mr}}
\bigl({}^\Oc\!\Ind_{\Gamma_n}^{\Gamma_{n+1}}(M)\otimes L_{m\l}\bigr)\cr
&={}^\Oc\!\Ind^{\Gamma_{n+mr+1}}_{\Gamma_{n+1,mr}}
{}^\Oc\!\Ind_{\Gamma_{n,mr}}^{\Gamma_{n+1,mr}}
(M\otimes L_{m\l})\cr
&={}^\Oc\!\Ind^{\Gamma_{n+mr+1}}_{\Gamma_{n,mr}}(M\otimes L_{m\l}).
\endaligned$$
By (\ref{eg2}) for each $M\in\Oc(\Gamma_n)$ we have
\begin{equation}
\label{mod111}
\bigoplus_qf_q\,a_\l^*(M)=\bigoplus_qa_\l^*\,f_q(M).
\end{equation}
We must prove that we have also an isomorphism
$f_q\,a_\l^*(M)=a_\l^*\,f_q(M)$.
Let $\Oc(\Gamma)_\nu\subset\Oc(\Gamma)$ be the full subcategory consisting
of the modules whose class is a weight vector of weight
$\nu$ of $[\Oc(\Gamma)]$.
Here $\nu$ is any weight of the $\widetilde{\sen\len}_m$-module $[\Oc(\Gamma)]$.
Recall that

\begin{lemma} We have the block decomposition
$\Oc(\Gamma)=\bigoplus_\nu\Oc(\Gamma)_\nu$,
where $\nu$ runs over the set of all weights of
the $\tilde{\sen\len}_m$-module $[\Oc(\Gamma)]$.
\end{lemma}

\begin{proof}
By \cite{S} the image by $\KZ:[\Oc(\Gamma)]\to[\Rep(\Hb(\Gamma))]$
of the class of a standard module is the class of a Specht module.
By \cite[thm.~2.11]{LM} we have a block decomposition
$$\Rep(\Hb(\Gamma))=\bigoplus_\nu\Rep(\Hb(\Gamma))_\nu,$$
where $\nu$ runs over a set of weights of $\widetilde{\sen\len}_m$
and the block $\Rep(\Hb(\Gamma))_\nu$ is generated by the constituents
of the Specht modules whose classes are the images by KZ
of the class of a standard module in $\Oc(\Gamma)_\nu$.
In particular, each Specht module belongs to a single block of
$\Rep(\Hb(\Gamma))$.
Now, since the standard modules in $\Oc(\Gamma)$ are indecomposable
(they have a simple top), each of them belong to a single block and
any block is generated by the constituents of the standard modules
in this block.
Finally, by \cite{GGOR} the functor KZ induces a bijection from the blocks of $\Oc(\Gamma)$ to the blocks of $\Rep(\Hb(\Gamma)$. Hence 
two standard modules belong to the same block of $\Oc(\Gamma)$
if and only if their images by KZ belong to the same block
of $\Rep(\Hb(\Gamma))$.
Therefore $\Oc(\Gamma)_\nu$ is a block of $\Oc(\Gamma)$.
This proves the lemma.
\end{proof}

\noindent
Therefore, to prove the isomorphism $f_q\,a_\l^*(M)=a_\l^*\,f_q(M)$
we may assume that $M$ lies in $\Oc(\Gamma)_\nu$.
Then $f_q\,a_\l^*(M)$ and $a_\l^*\,f_q(M)$ belong to
$\Oc(\Gamma)_{\nu+\alpha_q}$
by Proposition \ref{prop:comparison}.
Thus the isomorphism above follows from (\ref{mod111}).
The second isomorphism in (\ref{eg1}) is proved.
Next, let us prove that we have an isomorphism of functors
\begin{equation}\label{eg3}
E\,a_\l^*=a_\l^*\,E.
\end{equation}
The first isomorphism in (\ref{eg1}) follows from (\ref{eg3})
by a similar argument to the one above.
For $M\in\Oc(\Gamma_n)$ we have
$$\aligned
E\,a_\l^*(M)
&={}^\Oc\!\Res_{n+mr}{}^\Oc\!\Ind_{n,mr}(M\otimes L_{m\l}),\cr
a_\l^*\,E(M)
&={}^\Oc\!\Ind_{n-1,mr}({}^\Oc\!\Res_n(M)\otimes L_{m\l}),\cr
\endaligned$$
As above, we abbreviate $L=L_{(m)}^{\otimes r}$.
By Proposition \ref{prop:decomp} it is enough to prove 
that we have a natural isomorphism
$${}^\Oc\!\Res_{n+mr}{}^\Oc\!\Ind_{n,mr}(M\otimes 
{}^\Oc\!\Ind_{(m^r)}(L))\to
{}^\Oc\!\Ind_{n-1,mr}({}^\Oc\!\Res_n(M)\otimes 
{}^\Oc\!\Ind_{(m^r)}(L))$$
that is equivariant with respect to the $\Sen_r$-action induced by the 
$\Sen_r$-action on ${}^\Oc\!\Ind_{(m^r)}(L)$ 
given in (\ref{formind}). To see this, note that 
Proposition \ref{prop:cycloheckeindres} yields
the following decomposition of functors
$$\gathered
{}^\Hb\!\Res_{n+mr}\circ\,{}^\Hb\!\Ind_{n,mr}=
\bigl({}^\Hb\!\Ind_{n-1,mr}\circ\,({}^\Hb\!\Res_n\otimes\,\oneb)\bigr)
\oplus\cr
\oplus
\bigl({}^\Hb\!\Ind_{n,mr-1}\circ\,
(\oneb\otimes{}^\Hb\!\Res_{mr})\bigr)^{\oplus\ell}.\endgathered$$
Therefore we have also the following decomposition of functors
$$\gathered
\KZ\circ{}^\Oc\!\Res_{n+mr}\circ\,{}^\Oc\!\Ind_{n,mr}=\cr
=\bigl(\KZ\circ\,{}^\Oc\!\Ind_{n-1,mr}\circ\,({}^\Oc\!\Res_n\otimes\,\oneb)\bigr)
\oplus
\bigl(\KZ\circ\,{}^\Oc\!\Ind_{n,mr-1}\circ\,
(\oneb\otimes{}^\Oc\!\Res_{mr})\bigr)^{\oplus\ell}.
\endgathered$$
The induction and restriction functors on $\Oc(\Gamma)$ take projective modules
to projective ones, because they are exact and biadjoint.
Thus, by (\ref{isomproj}) we have a natural isomorphism
$$\gathered
{}^\Oc\!\Res_{n+mr}\,{}^\Oc\!\Ind_{n,mr}(P)=\cr
={}^\Oc\!\Ind_{n-1,mr}\,({}^\Oc\!\Res_n\otimes\,\oneb)(P)
\oplus
\bigl({}^\Oc\!\Ind_{n,mr-1}\,
(\oneb\otimes{}^\Oc\!\Res_{mr})(P)\bigr)^{\oplus\ell}
\endgathered$$
for any projective module $P\in\Oc(\Gamma)$.
Since $\Oc(\Gamma)$ has enough projective objects, this yields an isomorphism of functors
$$
{}^\Oc\!\Res_{n+mr}\,{}^\Oc\!\Ind_{n,mr}
={}^\Oc\!\Ind_{n-1,mr}\,({}^\Oc\!\Res_n\otimes\,\oneb)
\oplus
\bigl({}^\Oc\!\Ind_{n,mr-1}\,
(\oneb\otimes{}^\Oc\!\Res_{mr})\bigr)^{\oplus\ell}.$$
In particular, the projection yields a morphism of functors
$${}^\Oc\!\Res_{n+mr}\,{}^\Oc\!\Ind_{n,mr}
\to {}^\Oc\!\Ind_{n-1,mr}\,({}^\Oc\!\Res_n\otimes\,\oneb).$$
Applying this to the module $M\otimes {}^\Oc\!\Ind_{(m^r)}(L)$ 
yields an $\Sen_r$-equivariant surjective morphism
$$\Psi(M): {}^\Oc\!\Res_{n+mr}{}^\Oc\!\Ind_{n,mr}(M\otimes {}^\Oc\!\Ind_{(m^r)}(L))\to
{}^\Oc\!\Ind_{n-1,mr}({}^\Oc\!\Res_n(M)\otimes {}^\Oc\!\Ind_{(m^r)}(L)).$$
Now, by (\ref{5.2}) 
the left hand side is equal to $E\circ (a^\ast)^r(M)$ and 
the right hand side is equal to $(a^\ast)^r\circ E(M)$. 
So by Proposition \ref{prop:comparison} and the fact that the actions 
of $\Hen$ and $\widehat{\sen\len}_m$ on $\Fc_{m,\ell}^{(s)}$ commute 
with each other, we have 
$$[E\circ (a^\ast)^r(M)]=[(a^\ast)^r\circ E(M)].$$
Thus $\Psi(M)$ is indeed an isomorphism. 
So (\ref{eg3}) is proved.
\end{proof}

\vskip3mm

\subsection{Primitive modules}

\begin{df}
A module $M\in\Oc(\Gamma)$ is {\it primitive}
if $Ra_*(M)=0$ and $E(M)=0$ (or, equivalently, if $R^ia_*(M)=e_q(M)=0$ for all $q,i$).
Let  $\CIrr(\Oc(\Gamma))$
be the set
of isomorphism classes of primitive simple modules.
\end{df}

\begin{prop}\label{prop:critical1}
For $L\in\Irr(\Oc(\Gamma_n))$ the following are equivalent

(a) $L\in\CIrr(\Oc(\Gamma_n))$,

(b) $L\in\Irr(\Oc(\Gamma_n))_{0,0}$,

(c) $\dim(L)<\infty$.
\end{prop}

\begin{proof}
Assume that $L\in\Irr(\Oc(\Gamma_n))$.
The equivalence of $(b)$ and $(c)$ is Remark \ref{rk:fdim}.
Let us prove that $(a)\Rightarrow (b)$.
Fix $l,j\geqslant 0$ such that $\Supp(L)=X_{l,j}$.
Set $i=n-l-mj$.
We first prove that $j=0$. Assume that $j>0$.
Then we have
$$\Gamma_{l,(m^j)}=\Gamma_{l,(m^{j-1})}\times\frak S_m,\quad
\Gamma_{l,(m^{j-1})}\subset\Gamma_{n-m}. $$
There are modules $M_\mu\in\Oc(\Gamma_{n-m})$, $\mu\in\Pc_m$, such that
in $[\Oc(\Gamma_{n,m})]$ we have
$$[\Res_{n,m}(L)]=\sum_{\mu\in\Pc_m} [M_\mu\otimes L_\mu].$$
The transitivity of the restriction functor \cite[cor.~2.5]{S}
yields the following formula
\begin{equation*}
[\Res_1(L)]=\sum_\mu [\Res_2(M_\mu)\otimes L_\mu],\quad
\Res_1={}^\Oc\!\Res_{\Gamma_{l,(m^j)}}^{\Gamma_n},\quad
\Res_2={}^\Oc\!\Res^{\Gamma_{n-m}}_{\Gamma_{l,(m^{j-1})}}.
\end{equation*}
The $H(\Gamma_{l,(m^j)})$-module
$\Res_1(L)$ is finite dimensional,
because $\Supp(L)=X_{l,j}$.
Thus we have  $\Res_2(M_\mu)=0$ unless $\mu=(m)$, and
\begin{equation}\label{eq0}
[\Res_1(L)]=[\Res_2(M_{(m)})\otimes L_{(m)}].
\end{equation}
Next, since $R a_*([L])=0$ we have
$$\aligned
0&=[\Res_2Ra_*(L)]\cr
&=\sum_{\mu\in\Pc_m}[\Res_2(M_\mu)\otimes \RHom_{\Oc(\frak S_m)}(L_{(m)},L_{\mu})],\cr
&=[\Res_2(M_{(m)})\otimes \REnd_{\Oc(\frak S_m)}(L_{(m)})].
\endaligned$$
Thus, using  \cite[thm.~1.3]{BEG} we get
$\Res_2(M_{(m)})=0$. This yields a contradiction with (\ref{eq0}) because
$\Res_1(L)\neq 0$. So we have $j=0$.
Next, since $E(L)=0$, by Corollary \ref{cor:kere} and
Remark \ref{rk:support0j}
we have $i=0$.

Finally, we prove that $(c)\Rightarrow (a)$.
We must prove that if $L$ is finite dimensional then it is primitive.
This is obvious, because ${}^\Oc\!\Res_{n,m}(L)={}^\Oc\!\Res_{n}(L)=0$.
\end{proof}

\vskip3mm

\begin{rk}
By Proposition \ref{prop:critical1} the elements of $\CIrr(\Oc(\Gamma_n))$
form a basis of $F_{0,0}(\Gamma_n)$.
\end{rk}

\subsection{Endomorphisms of induced modules}

\noindent For $r\geqslant 1$ we consider the algebras
$$B_{r}=\Sen_r\ltimes
\CC[x_1,x_2,\dots,x_r],\quad
B_{r,\ell}=B_r/(x_1^\ell, x_2^\ell,\dots,x_r^\ell).$$
The following proposition is the main result of this subsection.

\begin{prop}
\label{prop:endo}
Let $r\geqslant 1$.

(a) The $\CC$-algebra homomorphism $\CC\frak S_r\to\End_{\Oc(\Gamma)}((a^*)^r)$
in Proposition \ref{prop:decomp} extends to a $\CC$-algebra homomorphism
$B_r\to\End_{\Oc(\Gamma)}((a^*)^r)$ such that
$x_1,x_2,\dots,x_r$
map to nilpotent operators in
$\End_{\Oc(\Gamma)}((a^*)^r(L))$ for each $L\in\Oc(\Gamma)$.

(b)
The $\CC$-algebra homomorphism $B_r\to\End_{\Oc(\Gamma)}((a^*)^r)$
factors to an isomorphism
$B_{r,\ell}=\End_{\Oc(\Gamma)}((a^*)^r(L))$ for $L\in\CIrr(\Oc(\Gamma))$.
\end{prop}

\begin{proof}
The proof of this proposition is done in several steps.
Let $\Hb(\Gamma_{n,(m^r)})$, $\Hb(\Gamma_{n})$ and
$X_i$ be as in Appendix \ref{app:A}. Consider the elements
$$
\xi_{i}=X_{n+m(i-1)+1}X_{n+m(i-1)+2}\cdots X_{n+mi},
\quad
i=1,2,\dots,r.$$
They belong to the centralizer
of $\Hb(\Gamma_{n,(m^r)})$ in $\Hb(\Gamma_{n+mr})$.
Thus the right multiplication by
$\xi_i,$ $i=1,2,\dots,r$, defines an automorphism ${}^\Hb\xi_i$ of the functor
${}^\Hb\!\Ind_{n,(m^r)}$. More precisely, for a $\Hb(\Gamma_{n,(m^r)})$-module $M$
we set
$${}^\Hb\xi_i(h\otimes v)=h\xi_i\otimes v,\quad
h\in\Hb(\Gamma_{n+mr}),\quad
v\in M.$$
The functor $\KZ $
yields a $\CC$-algebra isomorphism (\ref{kzh})
\begin{equation*}\KZ:\End\bigl({}^\Oc\!\Ind_{n,(m^r)}\bigr)\to
\End\bigl({}^\Hb\!\Ind_{n,(m^r)}\circ\KZ\bigr).
\end{equation*}
Thus there is a unique endomorphism ${}^\Oc\xi_i$
of the functor
${}^\Oc\!\Ind_{n,(m^r)}$
such that
\begin{equation}
\label{5.15}
\KZ({}^\Oc\xi_i(M))=
{}^\Hb\xi_i(\KZ(M)),\quad \forall M\in\Oc(\Gamma_{n,(m^r)}).
\end{equation}
The functor
$\bullet\otimes L:\Oc(\Gamma_{n})\to\Oc(\Gamma_{n,(m^r)})$
yields a $\CC$-algebra homomorphism
\begin{equation}\label{map0}
\End\bigl({}^\Oc\!\Ind_{n,(m^r)}\bigr)\to\End((a^*)^r).
\end{equation}
Let  $\bar\xi_i$ denote
the image of ${}^\Oc\xi_i$
by the map (\ref{map0}).
Next, recall the operators
$$\bar\tau_i\in\End((a^*)^r)=\End((A^*)^r),\quad i=1,2,\dots,r-1$$
defined in Proposition
\ref{prop:decomp}, see also the proof of Lemma \ref{lem:L5}.

\begin{lemma}
\label{lem:L1}
The following relations hold in $\End((a^*)^r)$ for $j\neq i,i+1$
$$\bar\tau_i\circ\bar\xi_i\circ\bar\tau_i=\bar\xi_{i+1},\quad
\bar\tau_i\circ\bar\xi_j\circ\bar\tau_i=\bar\xi_{j}.$$
\end{lemma}

\begin{proof}
By (\ref{5.15}) and (\ref{5.18}) it is enough to prove that
$$\gathered
({}^\Hb\tau_i\,\oneb_{\tau_i})\circ
({}^\Hb\xi_i\,\oneb_{\tau_i})\circ{}^\Hb\tau_i={}^\Hb\xi_{i+1},\quad
\quad
({}^\Hb\tau_i\,\oneb_{\tau_i})\circ({}^\Hb\xi_j\,\oneb_{\tau_i})
\circ{}^\Hb\tau_i={}^\Hb\xi_{j}.
\endgathered$$
To do so, we are reduced to check the following relations in $\Hb(\Gamma_{n+mr})$
$$T_{\tau_i}\xi_i T_{\tau_i}=\xi_{i+1},
\quad T_{\tau_i}\xi_j T_{\tau_i}=\xi_{j}.$$
Recall that $\zeta$ is a $m$-th root of 1. Let $a_i=n+(i-1)m+1$, $b_i=n+im$, and
$$K_l=T_{b_i-l}T_{b_i-l+2}\cdots T_{b_i+l-2}T_{b_i+l}.$$
A direct computation yields that
$$T_{\tau_i}=K_0K_1\cdots K_{m-2}K_{m-1}K_{m-2}\cdots K_1K_0.$$
Further, for $0\leqslant l\leqslant m-1$ we have
\begin{eqnarray*}
K_l X_{a_i}X_{a_i+1}\cdots X_{b_i-l-2}X_{b_i-l-1}(X_{b_i-l}X_{b_i-l+2}\cdots X_{b_i+l}) K_l=\\ =\zeta^{l+1}X_{a_i}X_{a_i+1}\cdots X_{b_i-l-2} (X_{b_i-l-1}X_{b_i-l+1}X_{b_i-l+3}\cdots X_{b_i+l+1}),
\end{eqnarray*}
and for $0\leqslant l\leqslant m-2$ we have
\begin{eqnarray*}
K_l (X_{b_i-l}X_{b_i-l+2}\cdots X_{b_i+l})X_{b_i+l+2}X_{b_i+l+3}\cdots X_{b_i+m}K_l=\\ =\zeta^{l+1}(X_{b_i-l+1}X_{b_i-l+3}\cdots X_{b_i+l-1} )X_{b_i+l+1}X_{b_i+l+2}X_{b_i+l+3}\cdots X_{b_i+m}.
\end{eqnarray*}
We deduce that
$$\begin{aligned}
T_{\tau_i}\xi_i T_{\tau_i}
&=T_{\tau_i}X_{a_i}X_{a_i+1}\cdots X_{b_i}T_{\tau_i}\cr
&=\zeta^{1+2+\cdots+m}K_0\cdots K_{m-2} X_{a_i+1}X_{a_i+3}\cdots X_{b_i+m-2}X_{b_i+m}K_{m-2}\cdots K_{0}\cr
&=\zeta^{1+2+\cdots+m}\zeta^{1+2+\cdots+m-1}X_{a_i+m}X_{a_i+m+1}\cdots X_{b_i+m}\cr
&=\zeta^{m^2}\xi_{i+1}\cr
&=\xi_{i+1}.
\end{aligned}$$
The relation $T_{\tau_i}\xi_jT_{\tau_i}=\xi_{j}$ for $j\neq i, i+1$ is obvious.
\end{proof}

\vskip3mm

For any element $w\in\Sen_r$ we set
$$\bar\tau_w=\bar\tau_{i_1}\bar \tau_{i_2}\cdots \bar\tau_{i_k}
\in\End((a^*)^r)$$
for $w=s_{i_1}s_{i_2}\cdots s_{i_k}$.
This definition does not depend on the choice of the decomposition
of $w$ by
Lemma \ref{lem:L5}.
Next, for a tuple $p=(p_1,p_2,\dots,p_r)\in\ZZ^r$ such that
$0\leqslant p_i<\ell$ we set
$$\xi^p=\xi_1^{p_1}\xi_2^{p_2}\cdots\xi_r^{p_r},
\quad
\bar\xi^p=\bar\xi_1^{p_1}\bar\xi_2^{p_2}\cdots\bar\xi_r^{p_r}.$$

\begin{lemma}
\label{lem:L2}
For any $L\in\Irr(\Oc(\Gamma))$
the elements $\bar\xi^p\,\bar\tau_w(L)$ of $\End_{\Oc(\Gamma)}((a^*)^r(L))$,
with
$w\in\Sen_r$ and $p\in[0,\ell)^r$, are linearly independent.
\end{lemma}

\begin{proof}
For  $w, i_1,\dots,i_k,p$ as above the expression
$\tau_{i_1}\tau_{i_2}\cdots\tau_{i_k}$ is reduced.
Let us define the following elements in
$\Hb(\Gamma_{n+mr})$
$$t_w=T_{\tau_{i_1}}T_{\tau_{i_2}}\cdots T_{\tau_{i_k}},
\quad \xi^p=\xi_1^{p_1}\xi_2^{p_2}\cdots\xi_r^{p_r}.$$
Recall that the elements
$$X_1^{p_1}X_2^{p_2}\cdots X_{n+mr}^{p_{n+mr}}T_w,\quad
p_i\in[0,\ell),\quad w\in\Sen_{n+mr},$$
form a $\CC$-basis of $\Hb(\Gamma_{n+mr})$.
Further $\xi^p$ centralizes $\Hb(\Gamma_{n,(m^r)})$ and
the element $\tau_{i_1}\tau_{i_2}\cdots\tau_{i_k}$ above
is minimal in its right $\Sen_{(n,m^r)}$-coset.
Therefore the left $\Hb(\Gamma_{n,(m^r)})$-submodule
of $\Hb(\Gamma_{n+mr})$ spanned by
$$\{\xi^pt_w\,;\, w\in\Sen_r,\, p\in[0,\ell)^r\},$$
is indeed the direct sum
$$\bigoplus_{p,w}\Hb(\Gamma_{n,(m^r)})\,\xi^pt_w,$$
where $p$ runs over $[0,\ell)^r$ and $w$ over $\Sen_r$.
In other words, there is an injective
$\Hb(\Gamma_{n,(m^r)})$-module homomorphism
\begin{equation}\label{injection}
{}^\Hb\psi:\Hb(\Gamma_{n,(m^r)})^{\oplus\ell^rr!}\to\Hb(\Gamma_{n+mr}),\quad
(h_{p,w})\mapsto\sum_{p,w}h_{p,w}\,\xi^p\,t_w,
\end{equation}
where $w,p$ run over $\Sen_r$, $[0,\ell)^r$ respectively.
Further, since $\xi^p$ centralizes $\Hb(\Gamma_{n,(m^r)})$, the relation
(\ref{taui}) yields
$$z\xi^pt_w=
\xi^pzt_w=
\xi^pt_ww^{-1}(z),
\quad
z\in\Hb(\Gamma_{n,(m^r)}),$$
where $w^{-1}(z)=\tau_{i_k}\cdots\tau_{i_2}\tau_{i_1}(z)$.
Therefore
${}^\Hb\psi$ is a $(\Hb(\Gamma_{n,(m^r)}),\Hb(\Gamma_{n,(m^r)}))$-bimodule
homomorphism, where the right
$\Hb(\Gamma_{n,(m^r)})$-action on
$\Hb(\Gamma_{n,(m^r)})^{\oplus\ell^rr!}$ is twisted in the obvious way.
Since ${}^\Hb\psi$ is injective, and both sides are free
$\Hb(\Gamma_{n,(m^r)})$-modules, for each
$M\in\Oc(\Gamma_{n,(m^r)})$
we have an injective homomorphism
$$\gathered
{}^\Hb\psi(\KZ(M)) :
\bigoplus_{p,w}w\KZ(M)\to
{}^\Hb\!\Res_{n,(m^r)}\circ\,{}^\Hb\!\Ind_{n,(m^r)}\KZ(M)=\cr
=\KZ\circ{}^\Oc\!\Res_{n,(m^r)}\circ\,{}^\Oc\!\Ind_{n,(m^r)}(M),
\endgathered$$
where
$$w=\tau_{i_1}\tau_{i_2}\dots\tau_{i_k}:
\Rep(\Hb(\Gamma_{n,(m^r)}))
\to\Rep(\Hb(\Gamma_{n,(m^r)})).$$
Further, we have
$$w\KZ(M)=\KZ(wM),$$
where
$$w:
\Oc(\Gamma_{n,(m^r)})
\to\Oc(\Gamma_{n,(m^r)})$$
is the twist by the permutation
$$w:H(\Gamma_{n,(m^r)})=H(\Gamma_n)\otimes H(\Sen_m)^{\otimes r}\to
H(\Gamma_n)\otimes H(\Sen_m)^{\otimes r}
=H(\Gamma_{n,(m^r)}).$$
The canonical adjunction morphism
$$P\to S(\KZ(P))$$ is an isomorphism for each
projective module $P\in\Oc(\Gamma)$.
Here
$S:\Rep(\Hb(\Gamma))\to\Oc(\Gamma)$ is the functor from Section \ref{sec:KZ}.
Further, the functors ${}^\Oc\!\Res_{n,(m^r)}$ and
${}^\Oc\!\Ind_{n,(m^r)}$ preserve the projective
objects, because they are bi-adjoint and exact.
Therefore, applying the left exact functor $S$ to the map
${}^\Hb\psi(\KZ(P))$, with $P$ projective in $\Oc(\Gamma_{n,(m^r)})$,
we get an injection
$${}^\Oc\psi(P):\bigoplus_{p,w}wP\to
{}^\Oc\Res_{n,(m^r)}\circ\,{}^\Oc\Ind_{n,(m^r)}(P).$$
Since the category $\Oc(\Gamma_{n,(m^r)})$ has enough projective objects and since
the functor ${}^\Oc\Res_{n,(m^r)}\circ\,{}^\Oc\Ind_{n,(m^r)}$ is exact, the  five lemma implies that
there is a functorial injective morphism
$${}^\Oc\psi(M):\bigoplus_{p,w}wM\to
{}^\Oc\Res_{n,(m^r)}\circ\,{}^\Oc\Ind_{n,(m^r)}(M),\quad
M\in\Oc(\Gamma_{n,(m^r)}).$$
Now, set $M=L\otimes L_{(m)}^{\otimes r}$ with $L\in\Irr(\Oc(\Gamma))$.
Then we have $wM=M$ for all $w$ as above.
Therefore we get an injective linear map
$$\gathered
\CC^{\ell^rr!}=\Hom_{\Oc(\Gamma)}(L\otimes L_{(m)}^{\otimes r},
L\otimes L_{(m)}^{\otimes r})^{\oplus\ell^rr!}\to\cr
\to\Hom_{\Oc(\Gamma)}(L\otimes L_{(m)}^{\otimes r},{}^\Oc\Res_{n,(m^r)}{}^\Oc\Ind_{n,(m^r)}(L\otimes L_{(m)}^{\otimes r}))
=\End_{\Oc(\Gamma)}((a^*)^r(L)).
\endgathered$$
It maps the canonical basis elements to the
elements $\bar\xi^p\,\bar\tau_w(L)$  with
$w\in\Sen_r$ and $p\in[0,\ell)^r$.
\end{proof}

\begin{lemma}
\label{lem:L3}
For $L\in\CIrr(\Oc(\Gamma_n))$
the following identity holds in
$[\Oc(\Gamma_{n,(m^r)})]$
$$[{}^\Oc\!\Res_{n,(m^r)}(a^*)^r(L)]=\ell^rr!\,[L\otimes L_{(m)}^{\otimes r}].$$
\end{lemma}

\begin{proof}
By Lemma \ref{lem:mackey} the left hand side is equal to
\begin{equation*}
\sum_x
{}^\Oc\!\Ind^{\Gamma_{n,(m^r)}}_{x^{-1}W_xx}\circ\,
{}^x\bigl({}^\Oc\!\Res^{\Gamma_{n,(m^r)}}_{W_x}
([L\otimes L_{(m)}^{\otimes r}])\bigr),
\end{equation*}
where $W_x=x\Gamma_{n,(m^r)}x^{-1}\cap \Gamma_{n,(m^r)}$ and
$x$ runs over a set of representatives of the double cosets
in $\Gamma_{n,(m^r)}\setminus \Gamma_{n+mr} /\Gamma_{n,(m^r)}$.
Since $W_x$ is a parabolic subgroup of $\Gamma_{n,(m^r)}$, it is generated by reflections. Hence we can decompose the group $W_x$ in the following way
\begin{equation}\label{decompW}
W_x=W'_x\times W''_x,
\quad W'_x\subset\Gamma_{n},\quad W''_x\subset\frak S_m^r.
\end{equation}
Here $W'_x,$ $W''_x$ are parabolic subgroups.
We have
$${}^\Oc\!\Res^{\Gamma_{n,(m^r)}}_{W_x}(L\otimes L_{(m)}^{\otimes r})=
{}^\Oc\!\Res^{\Gamma_{n}}_{W'_x}(L)\otimes
{}^\Oc\!\Res^{\frak S_m^r}_{W''_x}(L_{(m)}^{\otimes r}),$$
and a similar decomposition holds for the induction functor.
Further, since $L\in\CIrr(\Oc(\Gamma_{n}))$ we have
${}^\Oc\!\Res^{\Gamma_{n}}_{W'_x}(L)=0$
if $W'_x$ is proper by Proposition \ref{prop:critical1}.
Thus we can assume that $W'_x=\Gamma_n$, i.e.,
we can assume that $x$ belongs to the subgroup
$\{1\}\times\Gamma_{mr}\subset\Gamma_{n+mr}$.
We'll abbreviate
$$\Sen_m^r=\{1\}\times\Sen_m^r,
\quad\Gamma_{mr}=\{1\}\times\Gamma_{mr}.$$
Then we have $W''_x=x\Sen_m^rx^{-1}\cap\Sen_m^r$,
and we are reduced to check that
\begin{equation*}
\sum_x
{}^\Oc\!\Ind^{\Sen_m^r}_{x^{-1}W_xx}\circ\,
{}^x\bigl({}^\Oc\!\Res^{\Sen_m^r}_{W_x}
([L_{(m)}^{\otimes r}])\bigr)=\ell^rr!\,[L_{(m)}^{\otimes r}],
\end{equation*}
where $W_x=x\Sen_m^rx^{-1}\cap \Sen_m^r$ and
$x$ runs over a set of representatives of the double cosets
in $\Sen_m^r\setminus \Gamma_{mr} /\Sen_m^r$. Now, observe that
$${}^\Oc\!\Res^{\Sen_m^r}_{W_x}
(L_{(m)}^{\otimes r})=0$$
unless $x\Sen_m^rx^{-1}=\Sen_m^r$, and that
$x\Sen_m^rx^{-1}=\Sen_m^r$ if and only if $x$ belongs to
$N_{\Gamma_{mr}}(\Sen_m^r)$, the normalizer of $\Sen_m^r$ in $\Gamma_{mr}$.
Further, we have a group isomorphism
$$N_{\Gamma_{mr}}(\Sen_m^r)/\Sen_m^r=\Gamma_{r}.$$
This proves the lemma.
\end{proof}

\begin{lemma}
\label{lem:L4}
For $L\in\CIrr(\Oc(\Gamma))$
the elements $\bar\xi^p\,\bar\tau_w(L)$ with
$w\in\Sen_r$ and $p\in[0,\ell)^r$ form a basis of $\End_{\Oc(\Gamma)}((a^*)^r(L))$.
\end{lemma}

\begin{proof}
By Lemma \ref{lem:L2} it is enough to check that
$$\dim\,\End_{\Oc(\Gamma)}((a^*)^r(L))\leqslant \ell^rr!.$$
For $L\in\CIrr(\Oc(\Gamma_n))$ Lemma \ref{lem:L3} yields
$$\aligned
\dim\,\End_{\Oc(\Gamma)}((a^*)^r(L))
=
\dim\,\Hom_{\Oc(\Gamma)}(L\otimes L_{(m)}^{\otimes r},{}^\Oc\Res_{n,(m^r)}(a^*)^r(L))\leqslant\ell^rr!.
\endaligned$$
\end{proof}

\begin{lemma}
\label{lem:L6}
For $i=1,2,\dots,r$ and $L\in\Oc(\Gamma)$ the
operator $\bar\xi_i(L)+1$ on $(a^*)^r(L)$ is nilpotent.
Further, if $L\in\CIrr(\Oc(\Gamma))$ we have
$(\bar\xi_i(L)+1)^\ell=0$.
\end{lemma}

\begin{proof}
The $\CC$-vector space $[\Oc(\Gamma)]$ is equipped with an $\widetilde{\sen\len}_m$-action
via the isomorphism (\ref{chain}), see also Remark \ref{rk:scaling}.
For a weight $\mu$ of $\widetilde{\sen\len}_m$
let $\Oc(\Gamma)_\mu\subset\Oc(\Gamma)$
be the Serre subcategory generated by the simple modules $L$ whose class in $[\Oc(\Gamma)]$
has the weight $\mu$. 
Set $\Oc(\Gamma_n)_\mu=\Oc(\Gamma)_\mu\cap\Oc(\Gamma_n)$.
Although we'll not need this formula, note that if $\Delta_\l\in\Oc(\Gamma_n)_\mu$ then we have
$$\mu=\mu_0-\sum_{q=0}^{m-1}n_q(\l)\a_q$$
where $\mu_0$ is a weight which does not depend on $n$, $\l$, and
$n_q(\l)$ is the number of $q$-nodes in the $\ell$-partition $\l$.
The element
$$z_n=X_1X_2\cdots X_n$$
belongs to the center of $\Hb(\Gamma_n)$.
Thus it yields an element ${}^\Hb z_n$ in the center of $\Rep(\Hb(\Gamma_n))$.
Since $\KZ$ identifies the centers of $\Oc(\Gamma_n)$ and $\Rep(\Hb(\Gamma_n))$,
it yields also an element ${}^\Oc z_n$ in the center of $\Oc(\Gamma_n)$. Let
$L\in\Irr(\Oc(\Gamma_n)_\mu)$. Then
${}^\Oc z_n$ acts on $L$ by multiplication by the scalar $\zeta^{\nu(\mu)}$,
where $\nu$ is a linear form such that $\nu(\a_i)=i$ for $i=0,1,\dots,m-1$, see e.g.,
\cite[sec.~4.1]{S}. Now the operator $a^*$ maps $\Oc(\Gamma_n)_\mu$ to
$\Oc(\Gamma_{n+m})_{\mu+\delta}$ by Proposition \ref{prop:comparison}.
Thus ${}^\Oc z_{n+m}$ acts on $a^*(L)$ by multiplication by the scalar
$\zeta^{\nu(\mu+\delta)}$. Therefore $\bar\xi_1$ acts on $a^*(L)$ by multiplication by the scalar
$$\zeta^{\nu(\delta)}=\zeta^{m(m-1)/2}=-1.$$
By Lemmas \ref{lem:L1}, \ref{lem:L5} this implies that
for any $L\in\Oc(\Gamma)$ we have $(\bar\xi_i(L)+1)^N=0$ in
$\End_{\Oc(\Gamma)}((a^*)^r(L))$
for $i=1,2,\dots,r$ and $N$ large enough.

Now, assume that $L\in\CIrr(\Oc(\Gamma))$.
Let $N_i$ be the minimal integer
such that $(\bar\xi_i(L)+1)^{N_i}=0$.
By Lemmas \ref{lem:L1}, \ref{lem:L5} we have
$N_1=N_2=\dots=N_r$.
Hence, by Lemma \ref{lem:L2} we have also $\ell=N_1=N_2=\dots=N_r$.
\end{proof}

\vskip3mm

\noindent
Now we complete the  proof of Proposition \ref{prop:endo}.
The previous lemmas imply that the assignment
\begin{equation}
\label{map}
x_i\mapsto\bar\xi_i+1,\quad s_{j}\mapsto\bar\tau_j,\quad
i=1,2,\dots,r,\quad j=1,2,\dots,r-1,
\end{equation}
yields
a $\CC$-algebra morphism
$B_{r}\to\End_{\Oc(\Gamma)}((a^*)^r)$
such that
$x_i$ maps to a nilpotent operator in
$\End_{\Oc(\Gamma)}((a^*)^r(L))$ for each $L\in\Oc(\Gamma)$.
The action of $s_j$ on $(a^*)^r$ given above is the same as 
the action of $s_j$ on $(A^*)^r$ in Proposition
\ref{prop:decomp}.
This proves part $(a)$.
Part $(b)$ follows  from Lemmas \ref{lem:L4}, \ref{lem:L6}.
\end{proof}

\vskip3mm

For a module $M$ in $\Oc(\Gamma)$ the adjunction yields a morphism
$$\eta(M):M\otimes L_{(m)}^{\otimes r}\to {}^\Oc\!\Res_{n,(m^r)}(a^*)^r(M).$$

\begin{cor}
\label{cor:cor1}
For $r\geqslant 1$ and $L\in\CIrr(\Oc(\Gamma_n))$ the
$\CC$-algebra isomorphism
(\ref{map})
$$B_{r,\ell}=\End_{\Oc(\Gamma)}((a^*)^r(L))$$
yields an isomorphism
of $B_{r,\ell}\times H(\Gamma_{n,(m^r)})$-modules
$$\gathered
B_{r,\ell}\otimes(L\otimes L_{(m)}^{\otimes r})\to
{}^\Oc\!\Res_{n,(m^r)}(a^*)^r(L),
\quad w\otimes v\mapsto {}^\Oc\!\Res_{n,(m^r)}(w)\cdot\eta(L)(v).
\endgathered$$
\end{cor}

\begin{proof} The corollary follows from
Proposition \ref{prop:endo} and Lemma \ref{lem:L3}, because
$$\aligned
\End_{\Oc(\Gamma)}((a^*)^r(L))
=
\Hom_{\Oc(\Gamma)}(L\otimes L_{(m)}^{\otimes r},{}^\Oc\!\Res_{n,(m^r)}(a^*)^r(L))
\endaligned$$
is a free $B_{r,\ell}$-module of rank one
and, in $[\Oc(\Gamma_{n,(m^r)})]$, we have
$$[{}^\Oc\!\Res_{n,(m^r)}(a^*)^r(L)]=
\dim(B_{r,\ell})\,[L\otimes L_{(m)}^{\otimes r}].$$
\end{proof}

\vskip3mm

\begin{df}
For $\l\in\Pc_r$, $r\geqslant 1$, we can regard the $\Sen_r$-module
$\bar L_\l$ as a $B_{r,\ell}$-module
such that $x_1,x_2,\dots,x_r$ act by zero.
For $L\in\CIrr(\Oc(\Gamma_n))$ we define
$$\bar a^*_\l(L)=\bar L_\l\otimes_{B_{r,\ell}}(a^*)^r(L)\in\Oc(\Gamma_{n+mr}).$$
\end{df}

\begin{df}
For $r\geqslant 1$ we define a functor
$\Oc(\Gamma_{n+mr})\to\Rep(\Sen_r)\otimes\Oc(\Gamma_n)$ by
$$\aligned
\Psi(M)
&=\Hom_{\Oc(\Sen_m^r)}(L_{(m)}^{\otimes r},{}^\Oc\!\Res_{n,(m^r)}(M))
\cr
&=\Hom_{\Oc(\Sen_{mr})}({}^\Oc\!\Ind_{(m^r)}(L_{(m)}^{\otimes r}),
{}^\Oc\!\Res_{n,mr}(M)).
\endaligned$$
The  $\Sen_r$-action on $\Psi(M)$
is the $\Sen_r$-action on ${}^\Oc\!\Ind_{(m^r)}(L_{(m)}^{\otimes r})$
in the proof of Proposition \ref{prop:decomp}.
In other words, we have $\Psi=(a_*)^r$, viewed as a $\Sen_r$-equivariant 
functor as in the proof of Proposition \ref{prop:decomp}.
\end{df}

\begin{cor}
\label{cor:cor2}
For $r\geqslant 1$ and $L\in\CIrr(\Oc(\Gamma_n))$
we have an isomorphism
$$(L\otimes L_{(m)}^{\otimes r})^{\oplus \dim(\bar L_\l)}=
{}^\Oc\!\Res_{n,(m^r)}(\bar a_\l^*(L))$$
as $H(\Gamma_{n,(m^r)})$-modules, and we have an isomorphism
of $\Sen_r\times H(\Gamma_n)$-modules
$$\bar L_\l\otimes L=\Psi(\bar a_\l^*(L)).$$
\end{cor}

\begin{proof}
Corollary \ref{cor:cor1}
yields an isomorphism
$$B_{r,\ell}\otimes(L\otimes L_{(m)}^{\otimes r})=
{}^\Oc\!\Res_{n,(m^r)}\bigl((a^*)^r(L)\bigr)$$
which factors to an isomorphism
\begin{equation}
\label{tutu}
\gathered
\CC\Sen_r\otimes(L\otimes L_{(m)}^{\otimes r})=
{}^\Oc\!\Res_{n,(m^r)}\bigl(\overline{(a^*)^r}(L)\bigr),
\endgathered
\end{equation}
with
$$\overline{(a^*)^r}(L)=
(a^*)^r(L)\big/\sum_ix_i\,(a^*)^r(L).$$
Further, taking the isotypic components
the isomorphism (\ref{tutu}) factors to an isomorphism
$$
(L\otimes L_{(m)}^{\otimes r})^{\oplus\dim (\bar L_\l)}=
{}^\Oc\!\Res_{n,(m^r)}\bigl(\bar a^*_\l(L)\bigr).
$$
This proves the first claim.
To prove the second claim, observe that
Corollary \ref{cor:cor1}
and (\ref{tutu})
yield compatible $\Sen_r\times\Sen_r\times H(\Gamma_n)$-module isomorphism
\begin{equation}
\label{tvtv}
B_{r,\ell}\otimes L=\Psi\bigl((a^*)^r(L)\bigr),
\quad\CC\Sen_r\otimes L=\Psi\bigl(\overline{(a^*)^r}(L)\bigr).
\end{equation}
The first $\Sen_r$-action on
$\Psi\bigl(\overline{(a^*)^r}(L)\bigr)$
is the $\Sen_r$-action in the definition of $\Psi$,
and the first $\Sen_r$-action on $\CC\Sen_r\otimes L$ is the dual of the right $\Sen_r$-action on
$\CC\Sen_r$.
The second $\Sen_r$-action on
$\Psi\bigl(\overline{(a^*)^r}(L)\bigr)$ is the $\Sen_r$-action on $\overline{(a^*)^r}(L)$
in Corollary \ref{cor:cor1},
and the second $\Sen_r$-action on
$\CC\Sen_r\otimes L$ is the left $\Sen_r$-action on
$\CC\Sen_r$.
To identify  the actions as above, it is enough to note that
the isomorphism
\begin{equation}
\label{twtw}
\gathered
B_{r,\ell}=\Hom_{\Oc(\Gamma_n)}(L,B_{r,\ell}\otimes L)=
\Hom_{\Oc(\Gamma_n)}\bigl(L,\Psi(a^*)^r(L)\bigr)=\cr
=\End_{\Oc(\Gamma)}((a^*)^r(L))\endgathered
\end{equation}
given by (\ref{tvtv})
is equal to the isomorphism (\ref{map}), and that the $\Sen_r$-actions
on $(a^*)^r(L)$ are taken to the left and to the dual
right $\Sen_r$-action on $B_{r,\ell}$
by the map (\ref{twtw}).
Next, write
$$\CC\Sen_r=\bigoplus_\l\bar L_\l\otimes\bar L_\l$$
as an $\Sen_r\times\Sen_r$-module, and take the isotypic component.
\end{proof}

\subsection{Definition of  the map $\tilde a_\l$}

\begin{prop}\label{prop:eheh}
For $\l\in\Pc_r$ with $r\geqslant 1$ we have
$$a_\l^*(F_{i,j}(\Gamma_n))\subset F_{i,j+r}(\Gamma_{n+mr}),
\quad a_\l^*(F_{i,j}(\Gamma_n)^\circ)\subset F_{i,j+r}(\Gamma_{n+mr})^\circ.$$
\end{prop}

\begin{proof}
By Remark \ref{rk:support4}
we have
$$\Supp(L_{m\l})=X_{\frak S_m^r,\CC_0^{mr}}.$$
Let $L\in\Irr(\Oc(\Gamma_n))$.
First, assume that
$L\in\Irr(\Oc(\Gamma_n))_{i,j}$, i.e., that
$$\Supp(L)=X_{l,j,\CC^n}$$ by Remark \ref{rk:support3}.
Hence the module $L\otimes L_{m\l}$
has the following support
$$\Supp(L\otimes L_{m\l})=X_{l,j,\CC^n}\times X_{\frak S_m^r,\CC_0^{mr}}.$$
So by Proposition \ref{prop:indsup} we have
$$\Supp(a^*_\l(L))=X_{l,j+r,\CC^{n+mr}}.$$
Thus the class of $a^*_\l(L)$ belongs to $F_{i,j+r}(\Gamma_{n+mr})^\circ$
by Remark \ref{rk:support3}. Next, assume that
$[L]\in F_{i,j}(\Gamma_n)$, i.e.,
$$\Supp(L)=X_{l',j',\CC^n},\quad
X_{l',j',\CC^n}\subset X_{l,j,\CC^n}.$$
Thus we have
$$\Supp(a^*_\l(L'))=X_{l',j'+r,\CC^{n+mr}}.$$
So (\ref{bizarre}) yields
$$X_{l',j'+r,\CC^{n+mr}}\subset X_{l,j+r,\CC^{n+mr}},$$ i.e.,
the class of $a^*_\l(L)$ lies in $F_{i,j+r}(\Gamma_{n+mr}).$
\end{proof}

\begin{prop}\label{prop:tetesimple}
Let $\l\in\Pc_r$ with $r\geqslant 1$, and let $L\in\CIrr(\Oc(\Gamma_{n}))$.
The module $\top(\bar a_\l^*(L))$ has a unique constituent
in $\Irr(\Oc(\Gamma_{n+mr}))_{0,r}$.
\end{prop}

\begin{proof}
Since the module $L$ is primitive, it belongs to $\Irr(\Oc(\Gamma_n))_{0,0}$ by Proposition \ref{prop:critical1}.
Thus $[a^*_\l(L)]\in F_{0,r}(\Gamma_{n+mr})$ by Proposition \ref{prop:eheh}.
Thus the constituents of $\bar a^*_\l(L)$ belong to the set
$$\bigcup_{j\leqslant r}\Irr(\Oc(\Gamma_{n+mr}))_{0,j}$$ by Remark \ref{rk:support0j}.
Now, for $L'$ in $\Irr(\Oc(\Gamma_{n+mr}))_{0,j}$ we have
${}^\Oc\!\Res_{n,(m^r)}(L')=0$ if  $j<r$,
and $\dim {}^\Oc\!\Res_{n,(m^r)}(L')<\infty$ if $j=r$.
Further, the constituents of a finite dimensional module in $\Oc(\Sen_m^r)$
are all isomorphic to $L_{(m)}^{\otimes r}$, and, using
\cite[thm.~1.3]{BEG} as in the proof of Proposition \ref{prop:shift}, we get
$$\Ext^1_{\Oc(\Sen_{m}^r)}(L_{(m)}^{\otimes r},L_{(m)}^{\otimes r})=0.$$
Thus if $L'$ is a constituent of $\top(\bar a^*_\l(L))$ then we have
a surjective map
\begin{equation}\label{surjmap}\Psi(\bar a^*_\l(L))\to \Psi(L').
\end{equation}
We have also
$$\aligned
\Psi(L')=\bigoplus_{\mu\in\Pc_r}
\bar L_\mu\otimes\Hom_{\Oc(\Sen_{mr})}(L_{m\mu},{}^\Oc\!\Res_{n,mr}(L')).
\endaligned$$
Finally, Corollary \ref{cor:cor2} yields an isomorphism of $\Sen_r\otimes H(\Gamma_n)$-modules
$$\bar L_\l\otimes L=\Psi(\bar a^*_\l(L)).$$
Thus the surjectivity of (\ref{surjmap}) implies that
\begin{equation}
\label{formuleX}
\Hom_{\Oc(\Sen_{mr})}(L_{m\mu},{}^\Oc\!\Res_{n,mr}(L'))=0,\quad\forall\mu\neq\l.
\end{equation}
Since the $\Sen_r\otimes H(\Gamma_n)$-module $\bar L_\l\otimes L$ is simple,
the map (\ref{surjmap}) is invertible if it is nonzero. Assume further that $L'\in\Irr(\Oc(\Gamma_{n+mr}))_{0,r}$.
Then Proposition \ref{prop:ressup} yields
$${}^\Oc\!\Res_{n,(m^r)}(L')\neq 0.$$
Since $\dim {}^\Oc\!\Res_{n,(m^r)}(L')<\infty$
and the constituents of a finite dimensional module in $\Oc(\Sen_m^r)$
are all isomorphic to $L_{(m)}^{\otimes r}$,
 we have also
$\Psi(L')\neq 0$.
Therefore (\ref{surjmap}) is indeed invertible.
This implies that $\top(\bar a_\l^*(L))$ has a unique constituent
in $\Irr(\Oc(\Gamma_{n+mr}))_{0,r}$. Indeed, otherwise we would have a surjective map
$$\bar a^*_\l(L)\to L'\oplus L'',\quad L',L''\in\Irr(\Oc(\Gamma_{n+mr}))_{0,r},$$ yielding a surjective map
$$\bar L_\l\otimes L=\Psi(\bar a^*_\l(L))\to \Psi(L')\oplus\Psi(L'')=(\bar L_\l\otimes L)^{\oplus 2}.$$
This is absurd.
\end{proof}

\vskip3mm

\begin{df}
For $\l\in\Pc_r$ and $L\in\CIrr(\Oc(\Gamma))$
we define $\tilde a_\l(L)$ to be the unique constituent of
$\top(\bar a_\l^*(L))$ in $\Irr(\Oc(\Gamma))_{0,r}$.
\end{df}

\vskip3mm

\begin{prop}\label{prop:simple}
For $L\in\Irr(\Oc(\Gamma))_{0,r}$ there is
$L'\in\CIrr(\Oc(\Gamma))$, $\lambda\in\Pc_r$
such that $\tilde a_\lambda(L')\simeq L.$
In other words, there is a surjective map
\begin{equation}
\begin{gathered}
\CIrr(\Oc(\Gamma))\times\Pc_r\to
\Irr(\Oc(\Gamma))_{0,r},\quad
(L',\l)\mapsto\tilde a_\l(L').
\end{gathered}\end{equation}
\end{prop}

\begin{proof}
By Proposition \ref{prop:critical1}
the module $L$ is  primitive if and
only if $r=0$. Thus we can assume that $r>0$,
i.e., that $a_*(L)\neq 0$ by Corollary \ref{cor:kere},
else the claim is obvious.
Now, we first claim that there is a module $L_1\in\Irr(\Oc(\Gamma))_{0,r-1}$
with a surjective morphism
$\bar a^*(L_1)\to L.$
Indeed, since $a_*(L)\neq 0$, the adjunction map
$\eps:a^*(a_*(L))\to L$ is non-zero, hence it is surjective.
Hence, there is a constituent $L_1$ of
$a_*(L)$ such that $\eps$ yields a surjective morphism
$a^*(L_1)\to L.$

\begin{lemma}\label{lem:lemmepekinois}
If $L\in\Irr(\Oc(\Gamma))_{0,r}$ and
$L_1$ is a constituent of $a_*(L)$ such that
$a^*(L_1)$ maps onto $L$ then $L_1\in\Irr(\Oc(\Gamma))_{0,r-1}.$
\end{lemma}

\noindent
Fix the integer $n$ such that $L_1\in\Irr(\Oc(\Gamma_n))$.
Then $\bar\xi_1$ acts on $a^*(L_1)$ as the operator
$${}^\Oc\!z_{n+m}(a^*(L_1))\circ a^*({}^\Oc\!z_n(L_1))^{-1}.$$
The second factor is a scalar because $L_1$ is a simple module.
Hence $x_1$ acts on $a^*(L_1)$ as an element of the center of
$\Oc(\Gamma_{n+m})$, see (\ref{map}).
Therefore, since $L$ is simple and since the operator
$x_1$ on $a^*(L_1)$  is nilpotent by Proposition \ref{prop:endo},
the operator $x_1$ is 0 on $L$.
Thus the map $a^*(L_1)\to L$
factors to a surjective morphism
$$\eps_1:\bar a^*(L_1)\to L.$$
This proves the claim.

Now, assume that for $0<k<r$ there is a module
$L_k\in\Irr(\Oc(\Gamma))_{0,r-k}$ with a surjective homomorphism
$$\eps_k:\overline{(a^*)^k}(L_k)\to L,\quad
\quad\overline{(a^*)^k}(L_k)=
(a^*)^k(L_k)\bigl/\sum_ix_i(a^*)^k(L_k)
.$$
By the claim above, there is a module
$L_{k+1}\in\Irr(\Oc(\Gamma))_{0,r-k-1}$ with a surjective homomorphism
$$\bar a^*(L_{k+1})\to L_k.$$
Applying the functor $(a^*)^k$, which is exact, we get a surjective map
$$(a^*)^k\bar a^*(L_{k+1})\to (a^*)^k(L_k).$$
Taking the quotient by the action of $x_2,\dots,x_k,x_{k+1}$ it yields
a surjective map
$$(a^*)^k\bar a^*(L_{k+1})\bigl/
\sum_{i=2}^{k+1}x_i(a^*)^k\bar a^*(L_{k+1})\to
\overline{(a^*)^k}(L_k).$$
Now, since $a^*$ is exact, we have
$$(a^*)^k\bar a^*(L_{k+1})=
(a^*)^{k+1}(L_{k+1})\bigl/
x_1(a^*)^{k+1}(L_{k+1}).$$
Therefore we get a surjective map
$$\overline{(a^*)^{k+1}}(L_{k+1})=
(a^*)^{k+1}(L_{k+1})\bigl/
\sum_{i=1}^{k+1}x_i(a^*)^k\bar a^*(L_{k+1})\to
\overline{(a^*)^k}(L_k).$$
Composing it with $\eps_k$ we get a surjective homomorpism
$$\eps_{k+1}:\overline{(a^*)^{k+1}}(L_{k+1})\to L.$$
By induction, this yields
a module $L_r\in\Irr(\Oc(\Gamma))_{0,0}$ with a surjective homomorphism
$$\eps_r:\overline{(a^*)^r}(L_r)\to L.$$
Then we have $L_r\in\CIrr(\Oc(\Gamma))$ by Proposition \ref{prop:critical1},
and there is
$\l\in\Pc_r$ such that
$\bar a^*_\l(L_r)$ maps onto $L.$
The proposition follows from
Proposition \ref{prop:tetesimple}.
\end{proof}

\vskip3mm

\begin{proof}[Proof of Lemma \ref{lem:lemmepekinois}]
Fix $i,j\geqslant 0$ such that $L_1\in\Irr(\Oc(\Gamma))_{i,j}$.
By Proposition
\ref{prop:biadjoint},
since $E(L)=0$ we have $E\,a_*(L)=0$. Hence $E(L_1)=0$ by Proposition \ref{prop:EF}.
Thus $i=0$ by Corollary \ref{cor:kere}.
So, by Proposition \ref{prop:eheh} we have
$a^*(L_1)\in F_{0,j+1}(\Gamma)$. Since $a^* (L_1)$ maps onto $L$, we have
also $[L]\in  F_{0,j+1}(\Gamma)$.  Since $L\in\Irr(\Oc(\Gamma))_{0,r}$ this implies that
$r\leqslant j+1$ by Remark \ref{rk:support0j}.

Now, we prove that
$j+1\leqslant r$. Fix $n\geqslant 1$ such that $L\in\Oc(\Gamma_n)$. Recall that
$$a_*(L)=\Hom_{\Oc(\Sen_m)}\bigl(L_{(m)},{}^\Oc\!\Res_{n-m,m}(L)\bigr).$$
Thus there is an obvious inclusion
$$a_*(L)\otimes L_{(m)}\subset {}^\Oc\!\Res_{n-m,m}(L).$$
Hence, since $L_1$ is a constituent of $a_*(L)$, the module
$L_1\otimes L_{(m)}$ is a constituent of ${}^\Oc\!\Res_{n-m,m}(L)$.
Let us abbreviate
$$W'=\Gamma_{l,(m^j)},\quad l=n-(j+1)m,$$ regarded as a subgroup of $\Gamma_{n-m}$.
Then $W'\times \Sen_m\subset \Gamma_{n-m}\times\Sen_m$ in the obvious way.
Since $L_1\in\Irr(\Oc(\Gamma_{n-m}))_{0,j}$, we have
$$\Supp(L_1\otimes L_{(m)})=X_{W'\times\Sen_m,\,\CC^{n-m}\times\CC^m_0}.$$
By Proposition \ref{prop:ressup} applied to the module
$M=L$, we have also
$$\Supp(L_1\otimes L_{(m)})=X_{W_1,\,\CC^{n-m}\times\CC^m_0},$$
where $W_1$ is a parabolic subgroup of $\Gamma_{n-m,m}$ containing
a subgroup $\Gamma_n$-conjugate to $\Gamma_{n-mr,(m^r)}$. Hence we have
$F_{0,j+1}(\Gamma_n)\subset F_{0,r}(\Gamma_n)$.
Therefore we have
$j+1\leqslant r$ by Remark \ref{rk:support0j}.

\end{proof}

\vskip3mm

\vskip3mm

\section{The filtration of the Fock space and Etingof's conjecture}

Recall that $[\Oc(\Gamma)]$
is identified with the Fock space
$\Fc_{m,\ell}^{(s)}$ via the map (\ref{chain}).
The aim of this section is to identify the filtration on
$[\Oc(\Gamma)]$ defined in Section \ref{sec:stratum} in terms of supports of irreducible modules,
with a filtration on the Fock space given by representation theoretic tools.
We'll use the following notation :
$n,m,j,i$ are integers with $n>0$, $m>2$, $i,j\geqslant 0$ and $i=n-l-jm$.

\vskip3mm

\subsection{The representation theoretic interpretation of $F_{0,0}(\Gamma)$}
The goal of this section is to give a representation theoretic interpretation 
of $F_{0,0}(\Gamma)$ using the actions of
$\widehat{\sen\len}_m$ and $\Hen$ on $[\Oc(\Gamma)]$ defined
in the previous sections.
Note that the set $\Irr(\Oc(\Gamma))_{0,0}$ is a basis of the $\CC$-vector 
space $F_{0,0}(\Gamma)$. Further, we have proved that
$\Irr(\Oc(\Gamma))_{0,0}=\CIrr(\Oc(\Gamma))$.
in Proposition \ref{prop:critical1}.
Recall that the operators $b'_r$, $r\geqslant 1$, on $\Fc_{m,\ell}^{(s)}$ 
given in Section \ref{sec:fockl} act on $[\Oc(\Gamma)]$
via the map (\ref{chain}).

\begin{lemma}\label{lem:critical2} For $L\in\Irr(\Oc(\Gamma))$
we have
$L\in\CIrr(\Oc(\Gamma))$
if and only if  $E([L])=b'_r([L])=0$ in $[\Oc(\Gamma)]$ for all $r\geqslant 1$
\end{lemma}

\begin{proof}
It is enough to prove that for $L\in\CIrr(\Oc(\Gamma))$
we have $b'_r(L)=0$ for all $r\geqslant 1$.
A direct summand of the zero object  is zero in any additive category.
Further, for $L\in\CIrr(\Oc(\Gamma))$ we have
$(Ra_*)^r(L)=0$ for  $r\geqslant 1$.
Thus we have also $Ra_{\l,*}(L)=0$
for all $\l\in\Pc$ by Proposition \ref{prop:decomp}.
By Proposition \ref{prop:comparison}
the map (\ref{chain}) identifies
the $\CC$-linear operator $Ra_{\lambda,*}$ on
$[\Oc(\Gamma)]$ with the action of
$b'_{S_\lambda}$ on $\Fc_{m,\ell}^{(s)}$ given in Section \ref{sec:fockl}.
This proves the lemma.
\end{proof}

\vskip3mm

\noindent
In particular the lemma yields an inclusion
$$F_{0,0}(\Gamma)\subset
\{x\in[\Oc(\Gamma)]\,;\,e_q(x)=b'_r(x)=0,\,\forall q,r\}.$$
However it is not obvious that the right hand side 
is spanned by classes of irreducible objects of $\Oc(\Gamma)$.
This follows indeed from Proposition \ref{prop:noyau} below.
Before to prove it we need the following lemma.

\vskip3mm

\begin{prop}
\label{prop:noyau}
We have
$$\{x\in[\Oc(\Gamma)]\,;\,e_q(x)=b'_r(x)=0,\,\forall q,r\}=F_{0,0}(\Gamma).$$
\end{prop}

\begin{proof}
Consider the set $$F_{0,0}(\Gamma)'=\{x\in F_{0,\bullet}(\Gamma)\,;\,b'_r(x)=0,\,\forall r\geqslant 1\}.$$
By Corollary \ref{cor:kere} it is enough to prove that
$$F_{0,0}(\Gamma)=F_{0,0}(\Gamma)'.$$
We have
$$F_{0,0}(\Gamma)'=
\bigoplus_{n\geqslant 0} F_{0,0}(\Gamma_n)',\quad
F_{0,0}(\Gamma_n)'=F_{0,0}(\Gamma)'\cap
F_{0,\bullet}(\Gamma_n).$$
The actions of
$\widehat{\sen\len}_m$ and $\Hen$
on $\Fc_{m,\ell}^{(s)}$ commute to each other.
Thus, by Corollary \ref{cor:kere} the $\CC$-vector space
$F_{0,\bullet}(\Gamma)$ is identified with a $\Hen$-submodule of
$\Fc_{m,\ell}^{(s)}$ via the map (\ref{chain}), and we have
\begin{equation}\label{Eq2}
\sum_{n\geqslant 0}\dim(F_{0,\bullet}(\Gamma_n))\cdot t^n=
\sum_{n\geqslant 0}\sharp\Irr(\Oc(\Gamma_n))_{0,\bullet}\cdot t^n.
\end{equation}
The representation theory of $\Hen$ yields
the following formula in $\ZZ[[t]]$
\begin{equation}\label{Eq1}
\bigl(\sum_{k\geqslant 0}\dim(F_{0,0}(\Gamma_{k})')\cdot t^{k}\bigr)
\bigl(\sum_{r\geqslant 0}\sharp\Pc_r\cdot t^{mr}\bigr)
=\sum_{n\geqslant 0}\dim(F_{0,\bullet}(\Gamma_n))\cdot t^n.
\end{equation}
Finally,
Proposition \ref{prop:simple}
yields a surjective map
\begin{equation}\label{Eq4}
\begin{gathered}
\CIrr(\Oc(\Gamma_{k}))\times\Pc_r\to
\Irr(\Oc(\Gamma_n))_{0,r},\quad
(L,\l)\mapsto\tilde a_\l(L)
\end{gathered}\end{equation}
for $k,r\geqslant 0$
such that $n=k+mr$. From (\ref{Eq2}) and (\ref{Eq4}) we get
\begin{equation}\label{Eq3}
\bigl(\sum_{k\geqslant 0}\sharp\CIrr(\Oc(\Gamma_{k}))\cdot t^{k}\bigr)
\bigl(\sum_{r\geqslant 0}\sharp\Pc_r\cdot t^{mr}\bigr)-
\sum_{n\geqslant 0}\dim(F_{0,\bullet}(\Gamma_n))\cdot t^n
\in\NN[[t]].
\end{equation}
By Corollary \ref{cor:kere} and Lemma \ref{lem:critical2} we have
$\CIrr(\Oc(\Gamma_{k}))\subset F_{0,0}(\Gamma_{k})',$
hence we have
$$\sharp\CIrr(\Oc(\Gamma_{k}))\leqslant
\dim(F_{0,0}(\Gamma_{k})').$$
Therefore, comparing (\ref{Eq1}) and (\ref{Eq3}) we get the equality
\begin{equation}\label{Eq5}\sharp\CIrr(\Oc(\Gamma_{k}))=
\dim(F_{0,0}(\Gamma_{k})').
\end{equation}
In other words $\CIrr(\Oc(\Gamma_{k}))$
is a basis of $F_{0,0}(\Gamma_{k})'.$
Since $\CIrr(\Oc(\Gamma_{k}))$
is a basis of $F_{0,0}(\Gamma_{k})$
by Proposition \ref{prop:critical1},
we have also
$$F_{0,0}(\Gamma_k)=F_{0,0}(\Gamma_{k})'.$$
\end{proof}

\vskip3mm

\begin{rk}\label{rk:uneautre}
The proof
of Proposition \ref{prop:noyau}
and Corollary \ref{cor:kere}
imply that the map (\ref{Eq4})
yields a bijection
\begin{equation*}
\begin{gathered}
\CIrr(\Oc(\Gamma_{k}))\times\Pc_r\to
\Irr(\Oc(\Gamma_n))_{0,r},\quad
(L,\l)\mapsto\tilde a_\l(L)
\end{gathered}\end{equation*}
for $k,r\geqslant 0$ such that $n=k+mr$.
Note that Proposition \ref{prop:critical1} yields
$$\CIrr(\Oc(\Gamma_k))=\Irr(\Oc(\Gamma_k))_{0,0}.$$
\end{rk}

\vskip3mm

\subsection{The representation theoretic grading on $[\Oc(\Gamma)]$}
Using the actions of the Lie algebras $\Hen$ and $\widehat{\sen\len}_m$
we can now define a grading
$$[\Oc(\Gamma)]=\bigoplus_{i,j\geqslant 0}[\Oc(\Gamma)]_{i,j}.$$
Then, we'll compare it with the filtration by the support introduced
in Section \ref{sec:stratum}, i.e., we'll compare it with
the grading
$$[\Oc(\Gamma)]=\bigoplus_{i,j\geqslant 0}\gr_{i,j}(\Gamma).$$
To do so, let us consider the level $m\ell$ Casimir operator
\begin{equation*}
\partial={1\over m\ell}\sum_{r\geqslant 1}b_{r}b'_{r},
\end{equation*}
see (\ref{casimir0}).
Under the map (\ref{chain}) this formal sum defines a
diagonalisable $\CC$-linear operator on $[\Oc(\Gamma)]$.
For any integer $j$ let $[\Oc(\Gamma)]_{\bullet,j}$ be the eigenspace of $\partial$ with the eigenvalue $j$.
Note that $[\Oc(\Gamma)]_{\bullet,j}=0$ if $j<0$. Next, given an integer $i\geqslant 0$ we define
$[\Oc(\Gamma)]_{i,\bullet}$ to be the image of
$$\bigoplus_{\mu,\a} V_\mu^{\widehat{\sen\len}_m}[\mu-\a]\otimes
\Hom_{\widehat{\sen\len}_m}(V_\mu^{\widehat{\sen\len}_m},[\Oc(\Gamma)])$$
under the canonical maps
$$V_\mu^{\widehat{\sen\len}_m}\otimes
\Hom_{\widehat{\sen\len}_m}(V_\mu^{\widehat{\sen\len}_m},[\Oc(\Gamma)])\to
[\Oc(\Gamma)].$$
Here the sum runs over all sums $\a$
of $i$ affine simple roots of $\widehat{\sen\len}_m$, and over all dominant affine weight $\mu$
of $\widehat{\sen\len}_m$.  If $i<0$ we set
$[\Oc(\Gamma)]_{i,\bullet}=0.$

\begin{df} We define a grading on $[\Oc(\Gamma)]$ by the following formula
$$[\Oc(\Gamma)]_{i,j}=[\Oc(\Gamma)]_{i,\bullet}\cap[\Oc(\Gamma)]_{\bullet,j},
\quad
[\Oc(\Gamma_n)]_{i,j}=[\Oc(\Gamma)]_{i,j}\cap[\Oc(\Gamma_n)]$$
\end{df}

\begin{prop}\label{prop:propinteressante}
We have
$\dim[\Oc(\Gamma_n)]_{i,j}=
\dim\,\gr_{i,j}(\Gamma_n)$ for all $n,i,j\geqslant 0$.
\end{prop}

\begin{proof}
The vector space $[\Oc(\Gamma)]_{0,\bullet}$ is a $\Hen$-submodule
of $[\Oc(\Gamma)]$. Thus it is preserved by
the linear operator $\partial$ and
$[\Oc(\Gamma)]_{0,j}$ is the eigenspace with the eigenvalue $j$.
Since the $\Hen$-action on $[\Oc(\Gamma)]_{0,\bullet}$ has the level $m\ell$
we have $[\partial,b_j]=jb_j$ for all $j>0$.
Next, we have
$$[\Oc(\Gamma)]_{0,\bullet}=F_{0,\bullet}(\Gamma),\quad
[\Oc(\Gamma)]_{0,0}=F_{0,0}(\Gamma)$$
by Corollary \ref{cor:kere} and
Proposition \ref{prop:noyau}.
Further, the $\Hen$-action yields an isomorphism
\begin{equation}\label{isom1}
U^-(\Hen)_j\otimes [\Oc(\Gamma)]_{0,0}=[\Oc(\Gamma)]_{0,j}.
\end{equation}
By Remark \ref{rk:uneautre}, for $n=k+mj$
we have a bijection
\begin{equation}\label{isom0}
\Irr(\Oc(\Gamma_k))_{0,0}\times\Pc_j\to
\Irr(\Oc(\Gamma_n))_{0,j},\quad
(L,\l)\mapsto\tilde a_\l(L).
\end{equation}
Thus the isomorphism (\ref{isom1}) yields the following equality
\begin{equation}
\label{egalite}
\dim\,[\Oc(\Gamma_n)]_{0,j}=\sharp\Irr(\Oc(\Gamma_n))_{0,j}.
\end{equation}
Now, to compare
$\dim\,[\Oc(\Gamma_n)]_{i,j}$ and 
$\sharp\Irr(\Oc(\Gamma_n))_{i,j}$ for any $i\geqslant 0$,
we need some tools from canonical bases.
Since the integrable
$\widehat{\sen\len}_m$-module $[\Oc(\Gamma)]$
is not simple, the choice of a canonical basis of this module depends on a choice of a basis of
$[\Oc(\Gamma)]_{0,\bullet}.$
The general theory of canonical bases yields a bijection $\G$ between the
canonical basis of  $[\Oc(\Gamma)]$ and its crystal basis, the latter
being identified with $\Irr(\Oc(\Gamma))$ by Proposition \ref{prop:EF}.
The bijection $\G$
is such that a basis of $[\Oc(\Gamma)]_{0,\bullet}$ is given by
$$\{\G(L)\,;\,\tilde e_q(L)=0,\,\forall q\}.$$
By Corollary \ref{cor:kere} we have
$$\aligned
\{L\in\Irr(\Oc(\Gamma))\,;\,\tilde e_q(L)=0,\,\forall q\}
&=\Irr(\Oc(\Gamma))_{0,\bullet}\cr
&=\{\tilde a_\l(L)\,;\,
\forall\l\in\Pc,\,\forall L\in\Irr(\Oc(\Gamma))_{0,0}\}.
\endaligned$$ We'll choose the canonical basis of $[\Oc(\Gamma)]$ such that
$$\G(\tilde a_\l(L))=a^*_\l(L),\quad
\forall\l\in\Pc,\quad\forall L\in\Irr(\Oc(\Gamma))_{0,0}.$$
Then the set
$\{\G(L)\,;\, L\in\Irr(\Oc(\Gamma))_{0,j}\}$
is a basis of $[\Oc(\Gamma)]_{0,j}$ by (\ref{isom1}) and (\ref{isom0}).
The $\widehat{\sen\len}_m$-action on $[\Oc(\Gamma)]$ commutes
with the operator $\partial$.
Thus $[\Oc(\Gamma)]_{\bullet,j}$ is an $\widehat{\sen\len}_m$-module and the
$\widehat{\sen\len}_m$-action yields a surjective $\CC$-linear map
\begin{equation}\label{map1}
U^-(\widehat{\sen\len}_m)_i\otimes [\Oc(\Gamma)]_{0,j}\to
[\Oc(\Gamma)]_{i,j}.
\end{equation}
For weight reasons, the crystal of $[\Oc(\Gamma)]$
decomposes in the following way
$$\Irr(\Oc(\Gamma))=\bigsqcup_{i,j\geqslant 0}\Irr(\Oc(\Gamma))'_{i,j},
\quad \Irr(\Oc(\Gamma))'_{i,j}=\{L\in\Irr(\Oc(\Gamma))\,;\,\G(L)\in[\Oc(\Gamma)]_{i,j}\}.
$$
Since $\{\G(L)\,;\, L\in\Irr(\Oc(\Gamma))_{0,j}\}$
is a basis of $[\Oc(\Gamma)]_{0,j}$, we have
$$\Irr(\Oc(\Gamma))'_{0,j}=
\Irr(\Oc(\Gamma))_{0,j}.$$
Next $\Irr(\Oc(\Gamma))'_{\bullet,j}$
is the union of connected components of
$\Irr(\Oc(\Gamma))$ whose highest weight vector is in $\Irr(\Oc(\Gamma))'_{0,j}$, and
by Corollary \ref{cor:keretilde},
the set $\Irr(\Oc(\Gamma))_{\bullet,j}$
is the union of connected components of
$\Irr(\Oc(\Gamma))$ whose highest weight vector is in $\Irr(\Oc(\Gamma))_{0,j}$.
Thus, for all $n$ we have
$$\Irr(\Oc(\Gamma_n))'_{\bullet,j}=\Irr(\Oc(\Gamma_n))_{\bullet,j}.$$
By Corollary  \ref{cor:keretilde} and (\ref{map1}), for all $i$ we have also the inclusion
\begin{equation}\label{tyty}
\Irr(\Oc(\Gamma_n))'_{i,j}\subset\Irr(\Oc(\Gamma_n))_{i,j}.
\end{equation}
Thus (\ref{tyty}) is indeed an equality.
By definition, we have
$$\dim\,\gr_{i,j}(\Gamma_n)=\sharp\Irr(\Oc(\Gamma_n))_{i,j},
\quad
\dim\,[\Oc(\Gamma_n)]_{i,j}=\sharp\Irr(\Oc(\Gamma_n))'_{i,j}.$$
Thus the corollary is proved.
\end{proof}

\vskip3mm

\begin{rk}
Recall that $\gr_{i,j}(\Gamma)$ is identified with the subspace of
$[\Oc(\Gamma)]$ spanned by $\Irr(\Oc(\Gamma))_{i,j}$, see Section
\ref{sec:stratum}. Proposition \ref{prop:propinteressante} does not imply that
$[\Oc(\Gamma)]_{i,j}$ is also spanned by $\Irr(\Oc(\Gamma))_{i,j}$.
However, since
$$[\Oc(\Gamma)]_{0,0}=\{x\in[\Oc(\Gamma)]\,;\,e_q(x)=b'_r(x)=0,\,\forall q,r\},$$
the subspace $[\Oc(\Gamma)]_{0,0}$ is indeed spanned by $\Irr(\Oc(\Gamma))_{0,0}$
by Proposition \ref{prop:noyau}.
\end{rk}

\vskip3mm

\subsection{Etingof's conjecture}
Let $\a_{p,q}$ be the root of the elementary matrix $e_{p,q}$.
Recall that $\omega_0,\omega_1,\dots,\omega_{\ell-1}$
are the affine fundamental weights.
Fix a level 1 weight
$$\Lambda=\sum_ph_p\,\omega_p.$$

\begin{df}
Let $\tilde\aen_\Lambda$
be the Lie subalgebra of $\widetilde{\gen\len}_\ell$
spanned by $\oneb$, $D$ and
the elements $e_{p,q}\otimes\varpi^{r}$ with
$p,q=1,2,\dots,\ell$ and $r\in\ZZ$ such that
$\langle\Lambda,\a_{p,q}\rangle-hr\in\ZZ.$
We abbreviate $\tilde\aen=\tilde\aen_\Lambda$
and $\hat\aen=\tilde\aen\cap\widehat{\gen\len}_\ell$.
\end{df}

\vskip1mm

\noindent
We define the set of {\it positive real roots of $\tilde\aen$} to be
the set $\Delta^{\hat\aen}_+$ consisting of the real roots of
$\widehat{\gen\len}_\ell$ of the form
$\a+(r-\langle\Lambda,\a\rangle/h)\delta$ where $\a$ is a root of
$\gen\len_\ell$ and $\a+r\delta$ is a positive real root of
$\widehat{\gen\len}_\ell$.
Let $P^{\tilde\aen}_+$ be the set of {\it dominant integral weights for $\tilde\aen$},
i.e., the set of integral weights
of $\widetilde{\gen\len}_\ell$
which are $\geqslant 0$ on $\Delta^{\hat\aen}_+$.
For $\mu\in P^{\tilde\aen}_+$ let $V_\mu^{\tilde\aen}$ be the corresponding irreducible
integrable $\tilde\aen$-module.
We'll say that a non zero vector of an $\tilde\aen$-module is {\it primitive for $\tilde\aen$}
or {\it $\tilde\aen$-primitive} if
it is a weight vector whose weight belongs to $P^{\tilde\aen}_+$, and if it is killed by the action
of the weight vectors of $\tilde\aen$ whose weights are positive roots of $\tilde\aen$.
Now, let $h,h_p$ be the parameters of the $\CC$-algebra
$H(\Gamma_n)$ for each $n>0$.
Assume that $h$ is a rational number with the denominator $m>1$.
The elements of $\Hen$ can be regarded as elements
of $\widetilde{\gen\len}_\ell$ as in (\ref{prodlie}).
We have $b_{mr}, b'_{mr}\in\tilde\aen$ for each $r>0$.
The formal sum
\begin{equation*}
\partial_m={1\over m\ell}\sum_{r\geqslant 1}b_{mr}b'_{mr}
\end{equation*}
acts on every $\tilde\aen$-module $V_\mu^{\tilde\aen}$.
It is the $m$-th Casimir operator of $\widetilde{\gen\len}_\ell$
introduced in (\ref{mcasimir}).
For any weight $\l$ and any integer $j$ we denote by $V_\mu^{\tilde\aen}[\l,j]$
the subspace of weight $\l$ and eigenvalue $j$ of $\partial_m$.
We are interested in the following conjecture \cite[conj.~6.7]{E}.

\begin{conj}
\label{conj:etingof}
There exists an isomorphism of $\CC$-vector spaces
\begin{equation}
\label{conj}\gr_{i,j}(\Gamma_n)=\bigoplus_\mu
V^{\tilde\aen}_\mu[\omega_0-n\delta,j]\otimes
\Hom_{\tilde\aen}(V_\mu^{\tilde\aen},V_{\omega_0}^{\widetilde{\gen\len}_\ell}),
\end{equation}
where the sum is over all weights $\mu\in P^{\tilde\aen}_+$ such that
$\langle\mu,\mu\rangle=-2i$.
\end{conj}

\vskip3mm

\begin{rk}
\label{rk:omega0}
If $\Lambda=\omega_0$ then we have
$$\tilde\aen_{\o_0}=\bigl(\gen\len_\ell\otimes\CC[\varpi^m,\varpi^{-m}]\bigr)
\oplus\CC\oneb\oplus\CC D,$$
and the map (\ref{multm}) below yields a Lie algebra isomorphism
$\tilde\aen_{\o_0}=\widetilde{\gen\len}_\ell.$
\end{rk}

\vskip3mm

\begin{rk}
\label{rk:basic}
Assume that the $h_p$'s are rational numbers.
Let $\bar K$ be the algebraic closure of the field $K=\CC((\varpi))$.
Set $$\g=-\sum_{p=1}^{\ell-1}{h_p\over h}\,(\o_p-\o_0)\in P^{\sen\len_\ell}\otimes_\ZZ\QQ.$$ We have
$\a(\g)=-\langle\Lambda,\a\rangle/h$ for each root $\a$ of
$\sen\len_\ell$.
We may view $\g$ as the element $\g(\varpi)$ in $T_\ell(\bar K)$.
We have
$\tilde\aen=\ad(\g)^{-1}(\tilde\aen_{\omega_0}).$
Now, assume that $h$, $h_p$ are as in (\ref{h}).
Then we have $\g\in P^{\sen\len_\ell}$, because
$$\g
=\sum_{p=1}^{\ell-1}(s_{p+1}-s_p)\,(\o_p-\o_0)
.$$
A short computation
using the standard identification of
$\omega_p-\omega_0$ with the $\ell$-tuple
\begin{equation}\label{tuple}(1^p0^{\ell-p})-(p/\ell)\,(1^\ell)\end{equation}
shows that $\g$ belongs to $Q^{\sen\len_\ell}$ if and only if the $\ell$-charge $s$ has weight 0.
In this case $\g\in T_\ell(K)$, more precisely,
$\g$ is a cocharacter of $T_\ell$.
Thus the element $\xi_\g$ of the affine symmetric group is well-defined.
For a future use note that
\begin{equation}\label{gammaxi}
\hat\g(s,m)=\xi_\g^{-1}(\o_0)',
\end{equation}
where $\hat\g(s,m)\in P^{\widehat{\sen\len}_\ell}$
is as in (\ref{weight}),
and that
$\xi_\g(\mu)\in P^{\hat\aen}_+$ if and only if
$\mu'\in P^{\widehat{\gen\len}_\ell}_+$.
Here $\mu'$ is as in (\ref{mu'}) below.
\end{rk}

\vskip3mm

\subsection{Reminder on the level-rank duality}
For $\l\in\ZZ^\ell$ we consider the weights 
$$
\tilde\g(\l,m)=\hat\g(\l,m)-\Delta(\l,m)\delta\in
P^{\widetilde{\sen\len}_\ell},
$$
with $\hat\g(\l,m)\in P^{\widehat{\sen\len}_\ell},$
see Remark \ref{rk:scaling} and (\ref{weight}).
Note that $\tilde\g(\l,m)$ is dominant if and only if
$$\l\in A(\ell,m)=
\{(\l_1,\l_2,\dots,\l_\ell)\in\ZZ^\ell_+\,;\,\l_1-\l_\ell\leqslant m\}.$$
For $d\in\ZZ$ we write
$$A(\ell,m)_d=\{\l\in A(\ell,m)\,;\,\sum_p\l_p=d\}.$$
The {\it level-rank duality} yields a bijection
$A(\ell,m)_d\to A(m,\ell)_d$, $\l\mapsto\l^\dag$
such that

\begin{itemize}
\item we have the equality of weights
$$\hat\g(\l,m)=\sum_{p=1}^{m}\o_{\l_p^\dag\,\text{mod}\,\ell},$$
\item we have an $\widetilde{\sen\len}_m\times\Hen
\times\widetilde{\sen\len}_\ell$-module isomorphism
\begin{equation}\label{levelrankisom}\Fc_{m,\ell}^{(d)}
=\bigoplus_{\l\in A(\ell,m)_d}
V_{\tilde\g(\l^\dag,\ell)}^{\widetilde{\sen\len}_m}\otimes V^{\Hen}_{m\ell}
\otimes V_{\tilde\g(\l,m)}^{\widetilde{\sen\len}_\ell}
\end{equation}
and there are highest weight vectors
$v_{\tilde\g(\l^\dag,\ell)}$, $v_{m\ell}$,
$v_{\tilde\g(\l,m)}$ of
$V_{\tilde\g(\l^\dag,\ell)}^{\widetilde{\sen\len}_m}$, $V^{\Hen}_{m\ell}$,
$V_{\tilde\g(\l,m)}^{\widetilde{\sen\len}_\ell}$ such that
$|0,\l\rangle=v_{\tilde\g(\l^\dag,\ell)}\otimes v_{m\ell}\otimes v_{\tilde\g(\l,m)}$
for $\l\in A(\ell,m)_d$.
\end{itemize}
See e.g.,  \cite[(3.17)]{NT}, \cite[sec.~4.2, 4.3]{U}, for details.
Let $s=(s_p)$ be an $\ell$-charge of weight $d$.
Setting $d=0$, the formula (\ref{coco}) yields
\begin{equation}
\label{levelrank}
\Fc_{m,\ell}^{(s)}=
\bigoplus_{\l\in A(\ell,m)_0}
V_{\hat\g(\l^\dag,\ell)}^{\widehat{\sen\len}_m}\otimes V^{\Hen}_{m\ell}
\otimes\bigl( V_{\hat\g(\l,m)}^{\widehat{\sen\len}_\ell}[\hat\g(s,m)]\bigr).
\end{equation}
Here the bracket indicates the weight for the
$\widehat{\sen\len}_\ell$-action of level $m$.

\vskip3mm

\subsection{Proof of Etingof's conjecture for an integral $\ell$-charge}
In this subsection we prove Etingof's conjecture in the particular case where
the parameters $h$, $h_p$ are as in (\ref{h}). Note that our terminology differs from
\cite{E} because this case corresponds indeed to {\it rational (possibly non integral)}
values of the parameters.
From now on, unless specified otherwise
we'll assume that the parameters $h$, $h_p$ are as in (\ref{h}), and
we'll also assume that the
$\ell$-charge $s$ has weight zero.
To any level one weight $\mu$ of $\widehat{\gen\len}_\ell$ we associate the level $m$ weight
$\mu'$ given by
\begin{equation}
\label{mu'}\mu'=m\,\o_0+\sum_{p=1}^{\ell-1}\mu_p(\o_p-\o_0)\quad
\text{where}\quad
\mu=\o_0+\sum_{p=1}^{\ell-1}\mu_p(\o_p-\o_0).
\end{equation}
Note that $\g\in Q^{\sen\len}_\ell$
and that $\hat\g(s,m)=\xi_\g^{-1}(\o_0)'$
by (\ref{gammaxi}).
Using this and (\ref{levelrank}) we get a
$\widehat{\sen\len}_m\times\Hen$-module isomorphism
$$\Fc_{m,\ell}^{(s)}=\bigoplus_{\l\in A(\ell,m)_0}
V_{\hat\g(\l^\dag,\ell)}^{\widehat{\sen\len}_m}\otimes V^{\Hen}_{m\ell}
\otimes\bigl(V_{\hat\g(\l,m)}^{\widehat{\sen\len}_\ell}[\xi_\g^{-1}(\o_0)']\bigr).$$
Thus, by (\ref{coco}), (\ref{coucou}) and (\ref{multm}) we have
$$\Fc_{m,\ell}^{(s)}=V_{\omega_0}^{\widehat{\gen\len}_\ell}[\xi_\g^{-1}(\o_0)],$$
where the bracket indicates the weight subspace for the
$\widehat{\gen\len}_\ell$-action of level 1.
Since the map (\ref{chain}) yields an isomorphism
$[\Oc(\Gamma)]=\Fc_{m,\ell}^{(s)}$, we get
also an isomorphism
\begin{equation}
\label{turlututu}
[\Oc(\Gamma)]=V_{\omega_0}^{\widehat{\gen\len}_\ell}[\xi_\g^{-1}(\omega_0)].
\end{equation}
Under this isomorphism we have
$$[\Oc(\Gamma_n)]
=V_{\omega_0}^{\widetilde{\gen\len}_\ell}[\xi_\g^{-1}(\omega_0)-n\delta]$$
by (\ref{macdo2}) and the following lemma.

\begin{lemma}\label{lem:lemmechinois}
(a) If $\l^c$ is an $\ell$-core such that $\tau(\l^c)=s$
then $n_0(\l^c)={1\over 2}\langle\g,\g\rangle.$

(b) The element $|0,s\rangle$ is an extremal weight vector
of the module
$\Fc_{m,\ell}^{(0)}=V_{\omega_0}^{\widetilde{\gen\len}_\ell}$
with the weight $\xi_\g^{-1}(\omega_0)$.
\end{lemma}

\noindent
The formula
(\ref{conj}) we want to prove is
$$\dim\gr_{i,j}(\Gamma_n)=\sum_\mu
\dim \bigl(V^{\tilde\aen}_\mu[\omega_0-n\delta,j]\otimes
\Hom_{\tilde\aen}
(V_\mu^{\tilde\aen},V_{\omega_0}^{\widetilde{\gen\len}_\ell})\bigr),$$
where the sum is over all weights $\mu\in P^{\tilde\aen}_+$ such that
$\langle\mu,\mu\rangle=-2i$.
The proof consists of three steps.

\vskip1mm

{\it Case 1:} First, let us consider the sum over all $i$'s. We must prove that
$$\dim\,\gr_{\bullet,j}(\Gamma_n)
=\dim\bigl(V_{\omega_0}^{\widetilde{\gen\len}_\ell}
[\omega_0-n\delta,j]\bigr).$$
Note that
$$\dim\bigl(V_{\omega_0}^{\widetilde{\gen\len}_\ell}
[\omega_0-n\delta,j]\bigr)=\dim\bigl(V_{\omega_0}^{\widetilde{\gen\len}_\ell}
[\xi_\g^{-1}(\omega_0)-n\delta,j]\bigr),$$
because the Casimir operator $\partial_m$
commutes with the $\g$-action on
$V_{\omega_0}^{\widetilde{\gen\len}_\ell}$
by (\ref{twoheisenberg}).
Therefore, by Proposition \ref{prop:propinteressante}
we are thus reduced to prove that under
(\ref{turlututu}) we have
$$[\Oc(\Gamma_n)]_{\bullet,j}=
V_{\omega_0}^{\widetilde{\gen\len}_\ell}[\xi_\g^{-1}(\omega_0)-n\delta,j].$$
This follows from the
equality of the Casimir operators
(\ref{casimir}) and (\ref{mcasimir}), see (\ref{twoheisenberg}).

\vskip1mm

{\it Case 2 :}  Next, consider the case $i=0$.
Let $\Theta_{n,0}$ be the image of
\begin{equation}
\label{titi}
\bigoplus_{\tilde\mu}
V^{\tilde\aen}_{\tilde\mu}
[\omega_0-n\delta,j]\otimes
\Hom_{\tilde\aen}(V_{\tilde\mu}^{\tilde\aen},
V_{\omega_0}^{\widetilde{\gen\len}_\ell})
\end{equation}
by the canonical maps
$
V^{\tilde\aen}_{\tilde\mu}\otimes
\Hom_{\tilde\aen}(V_{\tilde\mu}^{\tilde\aen},
V_{\omega_0}^{\widetilde{\gen\len}_\ell})\to
V_{\omega_0}^{\widetilde{\gen\len}_\ell}.$
Here $\tilde\mu$ runs over the set of all weights in $P^{\tilde\aen}_+$
with $\langle\tilde\mu,\tilde\mu\rangle=0$.
By Proposition \ref{prop:propinteressante} and the discussion
above we must prove that
the image of $[\Oc(\Gamma_n)]_{0,j}$
by (\ref{turlututu})
is isomorphic to $\Theta_{n,0}$ as a vector space.
To do that, observe first that
by definition of $[\Oc(\Gamma_n)]_{0,j}$ the map
(\ref{turlututu}) takes $[\Oc(\Gamma_n)]_{0,j}$ onto the subspace
\begin{equation}
\label{forma}
V_{\omega_0}^{\widetilde{\gen\len}_\ell}[\xi_\g^{-1}(\omega_0)-n\delta]
\cap\bigoplus_{\l\in A(\ell,m)_0}
v_{\hat\g(\l^\dag,\ell)}\otimes
V^{\Hen}_{m\ell}[j]\otimes
V_{\hat\g(\l,m)}^{\widehat{\sen\len}_\ell}.
\end{equation}
Note that $v_{\hat\g(\l^\dag,\ell)}\otimes
V^{\Hen}_{m\ell}[j]\otimes
V_{\hat\g(\l,m)}^{\widehat{\sen\len}_\ell}$ is the
submodule of $\Fc^{(0)}_{m,\ell}=V_{\omega_0}^{\widehat{\gen\len}_\ell}$
generated by the vector $|0,\l\rangle$ for the level $m$ action of $\widehat{\gen\len}_\ell$.
Note also that $\tilde\aen_{\o_0}\simeq\widetilde{\gen\len}_\ell$
by Remark \ref{rk:omega0}. Finally, the set of weights of
$V_{\omega_0}^{\widehat{\gen\len}_\ell}$ is
$$\Wt(V_{\omega_0}^{\widehat{\gen\len}_\ell})=\{\o_0+\beta\,;\,\beta\in Q^{\sen\len_\ell}\},$$
see Section \ref{sec:extremal}, and we have the following lemma.

\begin{lemma}\label{lem:lemme}
(a) We have $\nu\in P^{\hat\aen_{\omega_0}}_+$ if and only if $\nu'\in P^{\widehat{\gen\len}_\ell}_+$.

(b) We have
$\{\nu'\,;\,\nu\in P^{\hat\aen_{\omega_0}}_+\cap\Wt(V_{\omega_0}^{\widehat{\gen\len}_\ell})\}=
\{\hat\g(\l,m)\,;\,\l\in A(\ell,m)_0\}.$
\end{lemma}

\noindent
Thus, by Lemmas \ref{lem:lemmechinois}, \ref{lem:lemme} the space (\ref{forma}) is indeed equal to
\begin{equation}
\label{formx}
\bigoplus_{\tilde\nu}
V_{\tilde\nu}^{\tilde\aen_{\omega_0}}[\xi_\g^{-1}(\omega_0)-n\delta,j],
\end{equation}
where the sum is over all extremal weights
$\tilde\nu$ in $P^{\tilde\aen_{\omega_0}}_+\cap\Wt(V_{\omega_0}^{\widetilde{\gen\len}_\ell})$
and $V^{\tilde\aen_{\omega_0}}_{\tilde\nu}$ is identified with the $\tilde\aen_{\omega_0}$-submodule of
$V_{\omega_0}^{\widetilde{\gen\len}_\ell}$ generated by a non zero extremal weight vector of weight
$\tilde\nu$.
Now, let us consider the space $\Theta_{n,0}$.
Recall that $\langle\tilde\mu,\tilde\mu\rangle=0$ if and only if $\tilde\mu$
is an extremal weight of
$V_{\omega_0}^{\widetilde{\gen\len}_\ell}$.
Further an extremal weight have a one-dimensional weight subspace,
see Section \ref{sec:extremal}.
Thus $\Theta_{n,0}$ is equal to the sum
\begin{equation}\label{formb}
\bigoplus_{\tilde\mu}
V^{\tilde\aen}_{\tilde\mu}[\omega_0-n\delta,j],
\end{equation}
where $\tilde\mu$ runs over the set of all extremal weights  such that
$V_{\omega_0}^{\widetilde{\gen\len}_\ell}$
contains an  $\tilde\aen$-primitive vector of weight $\tilde\mu$, say $v_{\tilde\mu}$, and
$V^{\tilde\aen}_{\tilde\mu}$ is identified with the $\tilde\aen$-submodule of
$V_{\omega_0}^{\widetilde{\gen\len}_\ell}$
generated by $v_{\tilde\mu}$.
Now, Remark \ref{rk:basic} yields
$$\tilde\aen=
\ad(\g)^{-1}(\tilde\aen_{\o_0}),\quad
\xi_\g(\tilde\mu)\in P^{\tilde\aen}_+\iff\tilde\mu\in P^{\widetilde\aen_{\omega_0}}_+.$$
Thus the $\g$-action yields a linear automorphism of $V_{\omega_0}^{\widetilde{\gen\len}_\ell}$
such that
$$
\g^{-1}(V^{\tilde\aen_{\omega_0}}_{\tilde\mu}[\xi_\g^{-1}(\omega_0)-n\delta,j])=
V^{\tilde\aen}_{\xi_\g(\tilde\mu)}[\omega_0-n\delta,j],
\quad\forall\tilde\mu\in P^{\tilde\aen_{\omega_0}}_+.$$
Thus (\ref{formx}) is equal to $\Theta_{n,0}$ by the following lemma.

\begin{lemma}\label{lem:lemmebis}
For all weight
$\mu$ in $P^{\hat\aen_{\omega_0}}_+\cap\Wt(V_{\omega_0}^{\widehat{\gen\len}_\ell})$
the module
$V_{\omega_0}^{\widetilde{\gen\len}_\ell}$ contains a
$\tilde\aen_{\o_0}$-primitive vector
of weight $\tilde\mu$.
\end{lemma}

\vskip1mm

{\it Case 3 :}  Finally, consider the general case.
Fix the integers $n,j$. Let $\Theta_{n,i}$ be the image of
\begin{equation*}
\bigoplus_{\tilde\nu}
V^{\tilde\aen}_{\tilde\nu}
[\omega_0-n\delta,j]\otimes
\Hom_{\tilde\aen}(V_{\tilde\nu}^{\tilde\aen},
V_{\omega_0}^{\widetilde{\gen\len}_\ell}),
\end{equation*}
by the canonical maps
$
V^{\tilde\aen}_{\tilde\nu}\otimes
\Hom_{\tilde\aen}(V_{\tilde\nu}^{\tilde\aen},
V_{\omega_0}^{\widetilde{\gen\len}_\ell})\to
V_{\omega_0}^{\widetilde{\gen\len}_\ell}.$
Here the sum is over all weights $\tilde\nu\in P^{\tilde\aen}_+$
such that
$\langle\tilde\nu,\tilde\nu\rangle=-2i$.
The same argument as for Case 2 implies that
$\Theta_{n,i}=\g^{-1}(\Theta'_{n,i})$ where
$\Theta'_{n,i}$ is the image of
\begin{equation*}
\bigoplus_{\tilde\mu}
V^{\tilde\aen_{\omega_0}}_{\tilde\mu}
[\xi_\g^{-1}(\omega_0)-n\delta,j]\otimes
\Hom_{\tilde\aen_{\omega_0}}(V_{\tilde\mu}^{\tilde\aen_{\omega_0}},
V_{\omega_0}^{\widetilde{\gen\len}_\ell}),
\end{equation*}
by the canonical maps
$
V^{\tilde\aen_{\omega_0}}_{\tilde\mu}\otimes
\Hom_{\tilde\aen_{\omega_0}}(V_{\tilde\mu}^{\tilde\aen_{\omega_0}},
V_{\omega_0}^{\widetilde{\gen\len}_\ell})\to
V_{\omega_0}^{\widetilde{\gen\len}_\ell},$
because the composition by the automorphism $\g^{-1}$ of
$V_{\omega_0}^{\widetilde{\gen\len}_\ell}$ yields a linear isomorphism
$$\Hom_{\tilde\aen}(V_{\xi_\g(\tilde\mu)}^{\tilde\aen},
V_{\omega_0}^{\widetilde{\gen\len}_\ell})=
\Hom_{\tilde\aen_{\omega_0}}(V_{\tilde\mu}^{\tilde\aen_{\omega_0}},
V_{\omega_0}^{\widetilde{\gen\len}_\ell}).$$
Here the sum is over all weights $\tilde\mu\in P^{\tilde\aen_{\omega_0}}_+$
such that
$\langle\tilde\mu,\tilde\mu\rangle=-2i$.
Let us prove that (\ref{turlututu}) maps
$[\Oc(\Gamma_n)]_{i,j}$ onto $\Theta'_{n,i}$.
The proof of Case 2 implies that
(\ref{turlututu}) maps
$[\Oc(\Gamma_n)]_{0,j}$ onto $\Theta'_{n,0}$.
By (\ref{map1}) we have
$$U^-(\widehat{\sen\len}_m)_i\,\bigl([\Oc(\Gamma_n)]_{0,j}\bigr)
=[\Oc(\Gamma_n)]_{i,j}.$$
By (\ref{macdo1}) we have also
$$U^-(\widehat{\sen\len}_m)_i\,(\Theta'_{n,0})\subset\Theta'_{n,i},$$
because the actions of
$\widehat{\sen\len}_m$ and $\hat\aen_{\omega_0}$ commute with each other.
Therefore, we have
$$[\Oc(\Gamma_n)]_{i,j}\subset\Theta'_{n,i}.$$
On the other hand, the proof of the first case implies that
$$[\Oc(\Gamma_n)]_{\bullet,j}=\bigoplus_{i\geqslant 0}\Theta'_{n,i}.$$
Thus we have the equality $[\Oc(\Gamma_n)]_{i,j}=\Theta'_{n,i}.$

\vskip3mm

\begin{proof}[Proof of Lemma \ref{lem:lemmechinois}]
A direct computation shows that $${1\over 2}\langle\g,\g\rangle={1\over 2}\sum_{p=1}^\ell s_p^2.$$
Now, consider the partition $\l^c=(\l_1,\ldots,\l_{k\ell})$.
We choose $k$ to be large enough such that
$\l_{k\ell}=0$. Write
$$\l_i-i+1= (a_i-1)\ell +b_i,\quad 1\leqslant b_i\leqslant \ell,$$
$$i-1=a_i'\ell + b_i',\quad 0\leqslant b_i'\leqslant \ell-1.$$
The number of 0-nodes in the $i$-th row of the Young diagram associated with
$\l^c$ is equal to $a_i+a_i'$. So
$$n_0(\l^c)=\sum_{i=1}^{k\ell} (a_i+a_i').$$
We have
$$\sum_{i=1}^{k\ell}a_i'=\frac{-k(-k+1)\ell}{2}.$$
By the definition of the bijection $\tau$,
\begin{eqnarray*} \sum_{i=1}^{k\ell}a_i&=&\sum_{p=1}^\ell\bigl( (-k+1)+(-k+2)+\cdots+s_p\bigr)\\
&=&{1\over 2}\sum_{p=1}^\ell s_p^2 -\frac{-k(-k+1)\ell}{2}.
\end{eqnarray*}
This proves part $(a)$. For part $(b)$, note that $(a)$ and (\ref{macdo00}) yield
$$D(|0,s\rangle)=-{1\over 2}\langle\g,\g\rangle|0,s\rangle$$
Further $|0,s\rangle$ is a weight vector for the level one representation of
$\widehat{\gen\len}_\ell$  with the weight $\o_0-\g$,
see \cite[(28)]{U}. Thus $|0,s\rangle$ is a weight vector for the level one representation of
$\widetilde{\gen\len}_\ell$  with the weight
$$\xi_\g^{-1}(\o_0)=\o_0-\g-{1\over 2}\langle\g,\g\rangle\delta.$$
The latter is an extremal weight, see Section \ref{sec:extremal}.
\end{proof}

\vskip3mm

\begin{proof}[Proof of Lemma \ref{lem:lemme}]
The set of all dominant integral weights of $\widehat{\sen\len}_\ell$ is
$$\aligned
\bigl\{\hat\g(\l,m)\,;\,\l\in A(\ell,m)\bigr\}
&=\bigl\{(m-\l_1+\l_\ell)\,\omega_0+
\sum_{p=1}^{\ell-1}(\l_p-\l_{p+1})\,\omega_p\,;\,\l\in A(\ell,m)\bigr\}\cr
&=\bigl\{m\,\omega_0+
\sum_{p=1}^{\ell-1}(\l_p-\l_{p+1})\,(\omega_p-\omega_0)\,;\,\l\in A(\ell,m)\bigr\}.
\endaligned$$
Set $\beta=\sum_{p=1}^{\ell-1}(\l_p-\l_{p+1})\,(\omega_p-\omega_0)$ with $\l\in A(\ell,m)$.
Identifying
$\omega_p-\omega_0$  with the $\ell$-tuple
(\ref{tuple}),
a short computation
shows that $\beta\in Q^{\sen\len_\ell}$ if and only if $\l\in A(\ell,m)_0$.
\end{proof}

\begin{proof}[Proof of Lemma \ref{lem:lemmebis}]
Fix a weight
$\mu$ in $P^{\hat\aen_{\omega_0}}_+\cap\Wt(V_{\omega_0}^{\widehat{\gen\len}_\ell})$.
Fix a non zero element $v\in V_{\omega_0}^{\widetilde{\gen\len}_\ell}$ of weight
$\tilde\mu=\mu-{1\over 2}\langle \mu,\mu\rangle\delta$. We must prove that
$v$ is $\tilde\aen_{\omega_0}$-primitive. The argument is  taken from \cite[sec.~6.2]{E}.
By Remark \ref{rk:omega0} it is enough to prove that $\tilde\mu+\nu$
is not a weight of $V_{\omega_0}^{\widetilde{\gen\len}_\ell}$ for any element $\nu$ in the set
$$\{\alpha_{p,p+1},\ \tilde\mu-\alpha_{1,\ell}+m\delta\,;\, p=1,2,\dots,\ell-1\}.$$
In fact, since $\tilde\mu\in P^{\tilde\aen}_+$, for such a $\nu$ we have
$$\langle\tilde\mu+\nu,\tilde\mu+\nu\rangle=\langle\nu,\nu\rangle+
2\langle\tilde\mu,\nu\rangle=2+
2\langle\tilde\mu,\nu\rangle>0.$$
Therefore $\tilde\mu+\nu$
is not a weight of $V_{\omega_0}^{\widetilde{\gen\len}_\ell}$ by Section \ref{sec:extremal}.
\end{proof}

\vskip3mm

\begin{rk}
Assume that the parameters $h$, $h_p$ are as in (\ref{h}).
Since $\g$ belongs to $T_\ell(K)$, it acts on any integrable $\widehat{\sen\len}_\ell$-module.
Let $0^\ell$ denote the trivial $\ell$-charge.
The $\g$-action on the representation of $\widehat{\sen\len}_\ell$ on
$\Fc_{m,\ell}$ of level 1 takes $\Fc^{(0^\ell)}_{m,\ell}$ onto $\Fc^{(s)}_{m,\ell}$.
Indeed, since $\g$ is a cocharacter of $T_\ell$ the formula
(\ref{formule78998}) yields the following equality
$$\aligned
\g(\Fc^{(0^\ell)}_{m,\ell})&=\g(\Fc^{(0)}_{m,\ell}[m\o_0])\cr
&=\g(V^{\widehat{\gen\len}_\ell}_{\o_0}[\o_0])\cr
&=V^{\widehat{\gen\len}_\ell}_{\o_0}[\xi_\g^{-1}(\o_0)]\cr
&=\Fc^{(0)}_{m,\ell}[\xi_\g^{-1}(\o_0)'].
\endaligned$$
Here the upper script ${}'$ is as in (\ref{mu'}).
Therefore, by Section \ref{sec:fockl} we are reduced to check the following identity
$$\xi_{\g}^{-1}(\o_0)=\o_0+{1\over h}\sum_{p=1}^{\ell-1}h_p\,(\o_p-\o_0).$$
Recall that $\g=-{1\over h}\sum_{p=1}^{\ell-1}h_p\,(\o_p-\o_0).$
Thus the claim follows from the formula (\ref{formule78997}) for the
$\widehat\Sen_\ell$-action on
$\ten^*_\ell\oplus\CC\o_0\oplus\CC\delta$.
\end{rk}

\vskip1cm

\appendix
\section{Reminder on Hecke algebras}\label{app:A}
\subsection{Affine Hecke algebras}
The affine Hecke algebra of type $GL_n$ with parameter $\zeta\in\CC^\times$
is the $\CC$-algebra $\hat\Hb_\zeta(n)$ generated by the symbols
$X_1,X_2,\dots,X_n$, $T_1,T_2,\dots,T_{n-1}$ modulo the defining relations
$$\gathered
X_iX_j=X_jX_i,\quad 1\leqslant i,j\leqslant n,\cr
T_iX_j=X_jT_i,\quad j\neq i,i+1,\cr
T_iX_iT_i=\zeta X_{i+1},\quad 1\leqslant i\leqslant n-1,\cr
(T_i+1)(T_i-\zeta)=0,\quad 1\leqslant i\leqslant n-1,\cr
T_iT_j=T_jT_i,\quad |i-j|>2,\cr
T_iT_{i+1}T_i=T_{i+1}T_iT_{i+1},\quad 1\leqslant i\leqslant n-2.
\endgathered$$
For $I\subset\{1,2,\dots,n-1\}$ let $\hat\Hb_\zeta(I)\subset\hat\Hb_\zeta(n)$
be the corresponding parabolic subalgebra. It is generated by the elements
$T_i$, $X_j$ with $i\in I$, $j=1,2,\dots,n$.
For a reduced expression $w=s_{i_1}s_{i_2}\cdots s_{i_k}$ of an element
$w\in\frak S_n$ we write $T_w=T_{i_1}T_{i_2}\cdots T_{i_k}$.
We abbreviate $T_{ij}=T_{s_{ij}}$.
Let $D_I$ be the set of minimal length representatives of the left cosets
in $\frak S_n/\frak S_I$. We'll abbreviate
$D_{I,J}=D_I^{-1}\cap D_J$.
For $x\in D_{I,J}$ the map
$$\frak S_{I\cap xJ}\to\frak S_{x^{-1}I\cap J},\quad
w\mapsto x^{-1}wx$$
defines a length preserving homomorphism.
Hence there is a $\CC$-algebra isomorphism
$$\hat\Hb_\zeta(I\cap xJ)\to\hat\Hb_\zeta(x^{-1}I\cap J),\quad
T_w\mapsto T_{x^{-1}wx},\quad
X_j\mapsto X_{x^{-1}(j)}.$$
Let
$$\Rep(\hat\Hb_\zeta(x^{-1}I\cap J))\to\Rep(\hat\Hb_\zeta(I\cap xJ)),\quad
M\mapsto {}^xM$$
be the corresponding twist functor.
The following is well-known.

\begin{lemma}
[Affine Mackey theorem]
Let $M\in\Rep(\hat\Hb_\zeta(J))$. The module
$$\Res^{\hat\Hb_\zeta(n)}_{\hat\Hb_\zeta(I)}
\Ind^{\hat\Hb_\zeta(n)}_{\hat\Hb_\zeta(J)}(M)$$ admits a filtration
with subquotients isomorphic to
$$\Ind^{\hat\Hb_\zeta(I)}_{\hat\Hb_\zeta(I\cap xJ)}
{}^x\Res^{\hat\Hb_\zeta(J)}_{\hat\Hb_\zeta(x^{-1}I\cap J)}(M),$$
one for each $x\in D_{I,J}$. The subquotients are taken in any order refining the Bruhat order on
$D_{I,J}$. In particular we have the inclusion
$$\Ind^{\hat\Hb_\zeta(I)}_{\hat\Hb_\zeta(I\cap J)}\Res^{\hat\Hb_\zeta(J)}_{\hat\Hb_\zeta(I\cap J)}(M)
\subset\Res^{\hat\Hb_\zeta(n)}_{\hat\Hb_\zeta(I)}
\Ind^{\hat\Hb_\zeta(n)}_{\hat\Hb_\zeta(J)}(M).$$
\end{lemma}

\subsection{Cyclotomic Hecke algebras}
\label{app:cyclohecke}
The cyclotomic Hecke algebra $\Hb_\zeta(n,\ell)$
associated with $\Gamma_n$ and the parameters
$\zeta,v_1,v_2,\dots,v_\ell\in\CC^\times$ is the quotient of
$\hat\Hb_\zeta(n)$ by the two-sided ideal generated by the element
$$(X_1-v_1)(X_1-v_2)\dots(X_1-v_\ell).$$
We'll denote the image of the generator $X_1$ in $\Hb_\zeta(n,\ell)$ by the
symbol $T_0$.
For a subset $I\subset\{0,1,\dots,n-1\}$ we define
$\Gamma_I\subset\Gamma_n$ as the subgroup
$\frak S_I$ if $0\not\in I$, or as the subgroup
generated by $\frak S_{I\setminus \{0\}}$ and $\{\g_1;\g\in\Gamma\}$ else.
This yields all parabolic subgroup of $\Gamma_n$. We consider also the
parabolic subalgebra $\Hb_\zeta(I,\ell)\subset\Hb_\zeta(n,\ell)$
which is the subalgebra generated by the elements $T_i$ with $i\in I$.
To unburden the notation, we abbreviate
$$\Hb(\Gamma_n)=\Hb_\zeta(n,\ell),\quad
\Hb(\frak S_m)=\Hb_\zeta(m),\quad
\Hb(\Gamma_I)=\Hb_\zeta(I,\ell).$$
For $r>0$ and
$I=\{0,1,\dots,n+mr-1\}\setminus\{n\}$ we write also
$$\Hb(\Gamma_{n,mr})=\Hb(\Gamma_I).$$

\subsection{Induction/restriction for cyclotomic Hecke algebras}
\label{app:cycloheckeindres}
We'll abbreviate
\begin{equation}
\label{abrevH}
\begin{aligned}
&{}^\Hb\!\Ind_{n}=\Ind_{\Hb(\Gamma_{n-1})}^{\Hb(\Gamma_{n})},
&{}^\Hb\!\Res_{n}=\Res^{\Hb(\Gamma_{n})}_{\Hb(\Gamma_{n-1})},\cr
&{}^\Hb\!\Ind_{n,(m^r)}=\Ind_{\Hb(\Gamma_{n,(m^r)})}^{\Hb(\Gamma_{n+mr})},
&{}^\Hb\!\Res_{n,(m^r)}=\Res^{\Hb(\Gamma_{n+mr})}_{\Hb(\Gamma_{n,(m^r)})},\cr
&{}^\Hb\!\Ind_{n,mr}=\Ind_{\Hb(\Gamma_{n,mr})}^{\Hb(\Gamma_{n+mr})},
&{}^\Hb\!\Res_{n,mr}=\Res^{\Hb(\Gamma_{n+mr})}_{\Hb(\Gamma_{n,mr})}.
\end{aligned}
\end{equation}
We write also
$$\gathered
{}^\Hb\!\Ind_{(m^r)}={}^\Hb\!\Ind_{\Sen_m^r}^{\Sen_{mr}}:\Rep(\Hb(\Sen_m^r))\to\Rep(\Hb(\Sen_{mr})),\cr
{}^\Hb\!\Res_{(m^r)}={}^\Hb\!\Res_{\Sen_m^r}^{\Sen_{mr}}:
\Rep(\Hb(\Sen_{mr}))\to\Rep(\Hb(\Sen_m^r)).
\endgathered$$
Now, we consider the Mackey decomposition of the functor
$${}^\Hb\!\Res_{n+m}\circ\,{}^\Hb\!\Ind_{n,m}:
\Rep(\Hb(\Gamma_{n,m}))\to\Rep(\Hb(\Gamma_{n+m-1})).$$
A short computation shows that a set of representatives of the double cosets
in $$\Gamma_{n+m-1}\setminus\Gamma_{n+m}/\Gamma_{n,m}$$ is
$\{\g_{n+m}, s_{n,n+m}\,;\,\g\in\Gamma\}$.
For
$$I=\{0,\dots,n+m-1\}\setminus\{n-1,n\},\quad
J=\{0,\dots,n+m-2\}\setminus\{n-1\}$$we have
$$\Hb(\Gamma_I)\subset\Hb(\Gamma_{n,m}),\quad
\Hb(\Gamma_J)=\Hb(\Gamma_{n-1,m})\subset\Hb(\Gamma_{n+m-1})
.$$
Further, there is an algebra isomorphism
$$\varphi:\ \Hb(\Gamma_J)\to\Hb(\Gamma_I),\
T_w\mapsto T_{sws^{-1}},\ X_i\mapsto X_{si},$$
where $s=s_ns_{n+1}\cdots s_{n+m-1}$.
For  each $i$, $p$ we write $X_i^p=(X_i)^p$.
We have the following decomposition. It is well known in the case $m=1$,
see e.g., \cite[lem.~5.6.1]{K} in the degenerate case.

\begin{prop}
\label{prop:cycloheckeindres}
(a) We have an isomorphism of $\Hb(\Gamma_{n+m-1})$-modules
$$\Hb(\Gamma_{n+m})=
\bigoplus_{0\leqslant p<\ell}\bigoplus_{1\leqslant j\leqslant n+m}
\Hb(\Gamma_{n+m-1})\,T_{j,n+m}\,X_j^p.$$
(b) We have an isomorphism of $\bigl(\Hb(\Gamma_{n+m-1}),\Hb(\Gamma_{n,m})\bigr)$-bimodules
$$\Hb(\Gamma_{n+m})=
\Hb(\Gamma_{n+m-1})\,T_{n,n+m}\,\Hb(\Gamma_{n,m})\oplus
\bigoplus_{0\leqslant p<\ell}
\Hb(\Gamma_{n+m-1})X_{n+m}^p\Hb(\Gamma_{n,m}).$$
(c) There are isomorphisms of
$\bigl(\Hb(\Gamma_{n+m-1}),\Hb(\Gamma_{n,m})\bigr)$-bimodules
$$\begin{gathered}
\Hb(\Gamma_{n+m-1})\,T_{n,n+m}\Hb(\Gamma_{n,m})=
\Hb(\Gamma_{n+m-1})\otimes_{\Hb(\Gamma_{n-1,m})}\Hb(\Gamma_{n,m}),\cr
\Hb(\Gamma_{n+m-1})\,X_{n+m}^p\Hb(\Gamma_{n,m})=
\Hb(\Gamma_{n+m-1})\otimes_{\Hb(\Gamma_{n,m-1})}\Hb(\Gamma_{n,m}),
\end{gathered}$$
where the algebra homomorphism
$\Hb(\Gamma_{n-1,m})\to\Hb(\Gamma_{n,m})$
is given by $\varphi$.
\end{prop}

\begin{proof}
Part $(a)$ is standard, see e.g., \cite[lem.~5.6.1]{K} in the degenerate case.
Let us concentrate on $(b)$.
Write
$t_{j,i}=T_jT_{j-1}\cdots T_i$  for $1\leqslant i\leqslant j$, and $t_{j,i}=1$ for $i>j$.
By $(a)$ we are reduced to prove the following identities
\begin{equation}\label{form1}
\bigoplus_{0\leqslant p<\ell}\bigoplus_{1\leqslant j\leqslant n}
\Hb(\Gamma_{n+m-1})\,t_{n+m-1,j}X_j^p=
\Hb(\Gamma_{n+m-1})\,t_{n+m-1,n}\,\Hb(\Gamma_{n,m}),
\end{equation}
\begin{equation}\label{form2}
\bigoplus_{0\leqslant p<\ell}\bigoplus_{n< j\leqslant n+m}\hskip-3mm
\Hb(\Gamma_{n+m-1})\,t_{n+m-1,j}X_j^p=
\bigoplus_{0\leqslant p<\ell}\Hb(\Gamma_{n+m-1})\,X_{n+m}^p\,\Hb(\Gamma_{n,m}).
\end{equation}
We have
\begin{equation}\label{form3}
u\,t_{n+m-1,n}=t_{n+m-1,n}\,\varphi(u),\quad u\in\Hb(\Gamma_{n-1,m}),
\end{equation}
because for $i=1,2,\dots,n-1$ and $j\in J\setminus\{0\}$ we have
$$\gathered
T_j\,t_{n+m-1,n}=t_{n+m-1,n}T_{s(j)}=t_{n+m-1,n}\varphi(T_{j}),\cr
 X_i\,t_{n+m-1,n}=t_{n+m-1,n}X_{i}=t_{n+m-1,n}\varphi(X_{i}).
 \endgathered$$
Hence, by $(a)$ the right hand side of (\ref{form1}) is
$$\aligned
&=\bigoplus_{0\leqslant p<\ell}\bigoplus_{1\leqslant j\leqslant n}
\Hb(\Gamma_{n+m-1})\,t_{n+m-1,n}\,\Hb(\Gamma_I)\,t_{n-1,j}\,X_j^p,\cr
&=\bigoplus_{0\leqslant p<\ell}\bigoplus_{1\leqslant j\leqslant n}
\Hb(\Gamma_{n+m-1})\,\Hb(\Gamma_{n-1,m})\,t_{n+m-1,n}\,t_{n-1,j}\,X_j^p,\cr
&=\bigoplus_{0\leqslant p<\ell}\bigoplus_{1\leqslant j\leqslant n}
\Hb(\Gamma_{n+m-1})\,t_{n+m-1,j}\,X_j^p.
\endaligned$$
This proves the first identity.
Next, a short calculation involving the relation
$$X_{j+1}^pT_j-T_jX_j^p\in\CC[X_j,X_{j+1}]$$
proves that the sum
$$\aligned
\sum_{0\leqslant p<\ell}\sum_{n< j\leqslant n+m}
\Hb(\Gamma_{n+m-1})\,t_{n+m-1,j}X_j^p
\endaligned$$
is indeed a direct sum, i.e., it is equal to the
left hand side of (\ref{form2}).
Thus the identity (\ref{form2}) follows from the following equalities
$$\aligned
\Hb(\Gamma_{n+m-1})\,X_{n+m}^p\,\Hb(\Gamma_{n,m})
&=\sum_{n<j\leqslant n+m}\Hb(\Gamma_{n+m-1})\,X_{n+m}^p\,T_{j,n+m}
\cr
&=\sum_{n<j\leqslant n+m}\Hb(\Gamma_{n+m-1})\,X_{n+m}^p\,t_{n+m-1,j}\cr
&=\sum_{n< j\leqslant n+m}
\Hb(\Gamma_{n+m-1})\,t_{n+m-1,j}X_j^p.
\endaligned$$
Finally, let us prove $(c)$. The second claim is obvious because
$$\Hb(\Gamma_{n+m-1})X_{n+m}^p\Hb(\Gamma_{n,m})=
X_{n+m}^p\Hb(\Gamma_{n+m-1})\Hb(\Gamma_{n,m})=
\Hb(\Gamma_{n+m-1})\Hb(\Gamma_{n,m})$$
as $\bigl(\Hb(\Gamma_{n+m-1}),\Hb(\Gamma_{n,m})\bigr)$-bimodules.
For the first one we define a map
$$\gathered
\Hb(\Gamma_{n+m-1})\times\Hb(\Gamma_{n,m})\to
\Hb(\Gamma_{n+m-1})\,T_{n,n+m}\,\Hb(\Gamma_{n,m}),\cr
\ (u,v)\mapsto u\,t_{n+m-1,n}v.
\endgathered$$
By (\ref{form3}) it factors to a surjective homomorphism
$$\psi:\ \Hb(\Gamma_{n+m-1})\otimes_{\Hb(\Gamma_{n-1,m})}\Hb(\Gamma_{n,m})\to
\Hb(\Gamma_{n+m-1})\,T_{n,n+m}\,\Hb(\Gamma_{n,m}).$$
By $(a)$ the left hand side is a free $\Hb(\Gamma_{n+m-1})$-module on basis
$$1\otimes t_{n-1,j}X_j^p,\quad 1\leqslant j\leqslant n,\quad 0\leqslant p<\ell.$$
But $\psi$ maps these elements to
$$t_{n+m-1,j}X_j^p,\quad 1\leqslant j\leqslant n,\quad 0\leqslant p<\ell.$$
Further, the latter form a $\Hb(\Gamma_{n+m-1})$-basis of the right hand side by $(a)$ again.
We are done.
\end{proof}

\section{Reminder on $\zeta$-Schur algebras}\label{app:B}

\subsection{The quantized modified algebra}
Let $v$ be a formal variable. The {\it quantized modified algebra $\dot\Ub(n)$ of $\gen\len_n$}
is the associative $\QQ(v)$-algebra
with generators $E_i$, $F_i$
where $i=1,\dots,n-1$
and $1_\l$ where $\l\in\ZZ^n$,
with the defining relations \cite[sec.~23]{Lu}
\begin{itemize}
\item
$1_\l 1_\mu=\delta_{\l,\mu}1_\l$,

\item
$E_iF_j-F_jE_i=\delta_{ij}\sum_\l[\l_i-\l_{i+1}]1_\l$,

\item
$E_i1_\l=1_{\l+\a_i}E_i$,

\item
$1_\l F_i=F_i 1_{\l+\a_i}$,

\item $E_iE_j=E_jE_i$ if $i\neq j\pm 1$,
$E_i^2E_j-(v+v^{-1})E_iE_jE_i+E_jE_i^2=0$ else,

\item $F_iF_j=F_jF_i$ if $i\neq j\pm 1$,
$F_i^2F_j-(v+v^{-1})F_iF_jF_i+F_jF_i^2=0$ else,
\end{itemize}
where $[m]$ is the usual $v$-analogue of $m$ for any $m\in\NN$.
The comultiplication of $\dot\Ub(n)$ is the $\QQ(v)$-algebra homomorphism
\begin{equation*}
\Delta:\dot\Ub(n)\to\prod_{\l,\l'}
\bigl(\dot\Ub(n)1_\l\otimes \dot\Ub(n)1_{\l'}\bigr)
\end{equation*}
given by
\begin{itemize}
\item
$\Delta(1_\l)=\prod_{\l=\l'+\l''}1_{\l'}\otimes 1_{\l''},$
\item
$\Delta(E_i1_\l)=\prod_{\l=\l'+\l''}
(E_i1_{\l'}\otimes 1_{\l''}+v^{(\a_i,\l')}1_{\l'}\otimes E_i1_{\l''}),$
\item
$\Delta(F_i1_\l)=\prod_{\l=\l'+\l''}
(F_i1_{\l'}\otimes v^{-(\a_i,\l'')}1_{\l''}+1_{\l'}\otimes F_i1_{\l''}).$
\end{itemize}
Set $\Ac=\ZZ[v,v^{-1}]$.
The {\it integral
quantized modified algebra}
is the $\Ac$-subalgebra
$\dot\Ub_\Ac(n)\subset\dot\Ub(n)$
generated by the $1_\l$'s and all quantum divided powers
$E_i^{(d)}$,
$F_i^{(d)}$.
The comultiplication yieds an $\Ac$-algebra homomorphism
$\dot\Ub_\Ac\to\dot\Ub_\Ac\otimes_\Ac\dot\Ub_\Ac$.
For $\eps\in\CC^\times$ we consider the $\CC$-algebra
$$\dot\Ub_\eps(n)=\dot\Ub_\Ac(n)\otimes_\Ac\CC[v,v^{-1}]/(v-\eps).$$
For $V,$ $V'\in\Rep(\dot\Ub_\eps(n))$ let
$\sb_{V,V'}:V\otimes V'\to V'\otimes V$ be the permutation
$v\otimes v'\mapsto v'\otimes v$. The {\it R-matrix}
is a $\CC$-linear endomorphism $R_{V,V'}$ of $V\otimes V'$ such that
the composed map
$$\Rc_{V,V'}=\sb_{V,V'}\circ R_{V,V'}$$ is an isomorphism of
$\dot\Ub_\eps(n)$-modules $V\otimes V'\to V'\otimes V$.
The map $R_{V,V'}$ decomposes in the following form
$$\gathered
R_{V,V'}(v\otimes v')=R(v\otimes v'),\quad
R=\bar\Pi\bar\Theta,\cr
\bar\Pi=\prod_{\l, \l'}v^{-(\l,\l')}\,1_\l\otimes 1_{\l'},\quad
\bar\Theta\in\prod_{\l,\l'}
\bigl(\dot\Ub_\eps(n)1_\l\otimes\dot\Ub_\eps(n)1_{\l'}\bigr).
\endgathered
$$
The notation is chosen to agree with \cite[sec.~32]{Lu}.
We call $R$ the {\it universal $R$-matrix}.
To avoid confusions we may write $R_\eps$ for $R$.
We'll write $\Rc_{V,V'}$ again for the braiding of
right $\dot\Ub_\eps(n)$-modules
$V$, $V'$.
If $\eps$ is a primitive $2d$-th root of 1
then we have $\eps^{d^2}=(-1)^d$. Hence
the {\it quantum Frobenius homomorphism}
\cite[sec.~35.1]{Lu}
is the unique
$\CC$-algebra homomorphism
$$\Fr:\dot\Ub_\eps(n)\to\dot\Ub_{(-1)^d}(n)$$ such that
\begin{itemize}
\item
$\Fr(E_i^{(m)}1_\l)=E_i^{(m/d)}1_{\l/d}$ if $m\in d\ZZ$ and
$\l\in d\ZZ^n$, and 0 otherwise,
\item
$\Fr(F_i^{(m)}1_\l)=F_i^{(m/d)}1_{\l/d}$ if $m\in d\ZZ$ and
$\l\in d\ZZ^n$, and 0 otherwise.
\end{itemize}
The formulas in \cite[sec.~3.1.5]{Lu} imply that
$$\Delta\circ\Fr=\Fr\circ\Delta.$$

\begin{prop}
\label{prop:B1}
We have $(\Fr\otimes\Fr)(R_\eps)=R_{(-1)^d}=
\prod_{\l,\l'}(-1)^{d(\l,\l')}(1_\l\otimes 1_{\l'})$.
\end{prop}

\begin{proof}
To avoid confusions we'll write
$\bar\Theta_\eps$, $\bar\Pi_\eps$ for
$\bar\Theta$, $\bar\Pi$.
If $n=2$ the proposition follows from the formula \cite[sec.~4.1.4]{Lu}.
More precisely, since
$$\bar\Theta_\eps=\prod_{\l,\l'}\sum_{k\geqslant 0}(-1)^k\eps^{-k(k-1)/2}
\{k\}_\eps F^{(k)}1_\l\otimes E^{(k)}1_{\l'},\quad
\{k\}_\eps=\prod_{i=1}^k(\eps^i-\eps^{-i}),$$
we have the following formula
\begin{equation}
\label{F1}
(\Fr\otimes\Fr)(\bar\Theta_\eps)=
\prod_{\l,\l'}(1_\l\otimes 1_{\l'}).
\end{equation}
Further, in $\dot\Ub_{(-1)^d}(n)\otimes\dot\Ub_{(-1)^d}(n)$ we have also
\begin{equation}\label{B2}
(\Fr\otimes\Fr)(\bar\Pi_\eps)=
\prod_{\l,\l'}(-1)^{d(\l,\l')}(1_\l\otimes 1_{\l'}),
\end{equation}
and
\begin{equation}
\label{F2}
\bar\Theta_{(-1)^d}=
\prod_{\l,\l'}(1_\l\otimes 1_{\l'}),\quad
\bar\Pi_{(-1)^d}=
\prod_{\l,\l'}(-1)^{d(\l,\l')}(1_\l\otimes 1_{\l'}).
\end{equation}
This proves the formula for $n=2$.
Now, let $n$ be any integer $\geqslant 2$.
The braid  group of $\Sen_n$ acts on $\dot\Ub_\eps(n)$
via the operators $T''_{1,1},T''_{2,1},\dots,T''_{n-1,1}$
in \cite[sec.~41]{Lu}.
For $i=1,2,\dots,n-1$ we set 
$$S_i=T''_{i,1}\otimes T''_{i,1},\quad
\bar\theta_{i,\eps}=\sum_{k\geqslant 0}(-1)^k\eps^{-k(k-1)/2}
\{k\}_\eps F_i^{(k)}\otimes E_i^{(k)}.$$
For a reduced decomposition $s_{i_1}s_{i_2}\cdots s_{i_r}$ of the longuest
element in $\Sen_n$, the universal R-matrix is given by the following formula,
see \cite[thm.~3]{KR},
$$\gathered
\bar\Theta_\eps=\prod_{\l,\l'}\bar\theta_\eps(1_\l\otimes 1_{\l'}),\quad
\bar\theta_\eps=
S_{i_r}^{-1}\cdots S_{i_3}^{-1}S_{i_2}^{-1}(\bar\theta_{i_1,\eps})
\cdots S_{i_r}^{-1}(\bar\theta_{i_{r-1},\eps})
\bar\theta_{i_{r},\eps}.
\endgathered$$
Thus (\ref{F2}) yields
$$\bar\Theta_{(-1)^d}=\prod_{\l,\l'}(1_\l\otimes 1_{\l'}),$$
Since the braid group action is compatible with the
quantum Frobenius homomorphism, see \cite[sec.~41.1.9]{Lu}, by
(\ref{F1}) we have also
$$(\Fr\otimes \Fr)(\bar\Theta_\eps)
=\prod_{\l,\l'}(1_\l\otimes 1_{\l'}).$$
Finally, a direct computation yields
$$
(\Fr\otimes\Fr)(\bar\Pi_\eps)=
\prod_{\l,\l'}(-1)^{d(\l,\l')}(1_\l\otimes 1_{\l'})=
\bar\Pi_{(-1)^d}.
$$
This proves the proposition.
\end{proof}

\begin{rk}
\label{rk:B2}
It is proved in \cite[prop.~33.2.3]{Lu} that the assignment
$$E_i1_\l\mapsto (-1)^{id(\l_i-\l_{i+1})}E_i1_\l,\quad
F_i1_\l\mapsto (-1)^{(i+1)d(\l_i-\l_{i+1})+d}E_i1_\l,\quad
1_\l\mapsto 1_\l$$ yields a
$\CC$-algebra isomorphism
$\dot\Ub_{1}(n)\to\dot\Ub_{(-1)^d}(n)$.
Thus we can regard $\Fr$ as a map $\dot\Ub_\eps(n)\to\dot\Ub_{1}(n)$.
Note that the isomorphism above does not commute with the comultiplication.
\end{rk}

\subsection{The $\zeta$-Schur algebra}
The $v$-Schur algebra $\Sb(n,m)$
is the associative
$\QQ(v)$-algebra with 1 generated by $E_i$, $F_i$
where $i=1,\dots,n-1$ and by
$1_\l$ where $\l\in\Lambda(n,m)$,
modulo the defining relations
\cite[thm.~2.4]{Dot}
\begin{itemize}
\item
$1_\l 1_\mu=\delta_{\l,\mu}1_\l$,
$\sum_\l 1_\l=1$,

\item
$E_iF_j-F_jE_i=\delta_{ij}\sum_\l[\l_i-\l_{i+1}]1_\l$,

\item
$E_i1_\l=1_{\l+\a_i}E_i$
if $\l+\a_i\in\Lambda(n,m)$,
$0$ else,

\item
$1_\l E_i=E_i 1_{\l-\a_i}$
if $\l-\a_i\in\Lambda(n,m)$,
$0$ else,

\item
$F_i1_\l=1_{\l-\a_i}F_i$
if $\l-\a_i\in\Lambda(n,m)$,
$0$ else,

\item
$1_\l F_i=F_i 1_{\l+\a_i}$
if $\l+\a_i\in\Lambda(n,m)$,
$0$ else,

\item $E_iE_j=E_jE_i$ if $i\neq j\pm 1$,
$E_i^2E_j-(v+v^{-1})E_iE_jE_i+E_jE_i^2=0$ else,

\item $F_iF_j=F_jF_i$ if $i\neq j\pm 1$,
$F_i^2F_j-(v+v^{-1})F_iF_jF_i+F_jF_i^2=0$ else.
\end{itemize}
The integral $v$-Schur algebra is the
$\Ac$-subalgebra
$\Sb_\Ac(n,m)\subset\Sb(n,m)$
generated by the $1_\l$'s and all quantum divided powers
$E_i^{(d)}$,
$F_i^{(d)}$.
In other words, we have a canonical isomorphism
$$\Sb_\Ac(n,m)=
1_m\dot\Ub_\Ac(n)1_m,\quad
1_m=\sum_{\l\in\Lambda(n,m)}1_\l.$$
The comultiplication of $\dot\Ub_\Ac(n)$ factors through an $\Ac$-algebra
homomorphism
\begin{equation}\label{delta}
\Delta:\Sb_\Ac(n,m)\to\bigoplus_{m=m'+m''}
\Sb_\Ac(n,m')\otimes \Sb_\Ac(n,m'').
\end{equation}
For $\zeta,\eps\in\CC^\times$ with $\zeta=\eps^2$ we consider the $\CC$-algebra
$$\aligned
\Sb_\zeta(n,m)
&=\Sb_\Ac(n,m)\otimes_\Ac\CC[v,v^{-1}]/(v-\eps)
\cr
&=1_m\dot\Ub_\eps(n) 1_m.
\endaligned$$
Indeed $\Sb_\zeta(n,m)$ depends only on $\zeta$ and not on the choice of $\eps$.
If $\zeta$ is a primitive $d$-th root of 1,
we choose $\eps$ to be a primitive $2d$-th root of 1.
Then the {\it quantum Frobenius homomorphism}
$\Fr:\dot\Ub_\eps(n)\to\dot\Ub_{1}(n)$ factors through a
$\CC$-algebra homomorphism
\begin{equation}\label{frob}\Fr:\Sb_\zeta(n,dm)\to\Sb_1(n,m).
\end{equation}
Note that we have used the identification 
$\dot\Ub_{(-1)^d}(n)=\dot\Ub_{1}(n)$ in Remark \ref{rk:B2}.

\subsection{The module category of $\Sb_\zeta(n,m)$}
For $\l\in\ZZ^n_+$
let $\Delta_\l^U, L_\l^U\in\Rep(\dot\Ub_\eps(n))$
denote the Weyl module and the simple module
with highest weight $\l$.
Set
$$\Lambda(n,m)_+=\Lambda(n,m)\cap\ZZ^n_+.$$
The category
$\Rep(\Sb_\zeta(n,m))$ is equivalent to the full subcategory of
$\Rep(\dot\Ub_\eps(n))$
consisting of the modules such that all
constituents have a highest weight in the set
$\Lambda(n,m)_+.$
It is
quasi-hereditary with respect to the dominance order,
the standard objects being the modules $\Delta^S_\l$
with $\l\in\Lambda(n,m)_+$. Here, for
$\l\in\Lambda(n,m)_+$, we write
$$\Delta_\l^S=\Delta_\l^U, \quad L_\l^S=L_\l^U,$$
regarded as objects in $\Rep(\Sb_\zeta(n,m)).$

\subsection{The Schur functor}
Assume that  $n\geqslant m$. There is a $\CC$-algebra isomorphism
\cite[sec.~11]{Dot}
$$\Hb_\zeta(m)=f\,\Sb_\zeta(n,m)\,f,\quad f=1_{(1^m0^{n-m})}.$$
Thus the vector space $\Tb_\zeta(n,m)=\Sb_\zeta(n,m)f$ is a $(\Sb_\zeta(n,m),\Hb_\zeta(m))$-bimodule,
and $\Vb_\zeta(n,m)=f\Sb_\zeta(n,m)$ is a $(\Hb_\zeta(m),\Sb_\zeta(n,m))$-bimodule.
Consider the triple of adjoint functors $(\Phi_!,\Phi^*,\Phi_*)$
$$\begin{aligned}
&\Phi^*:\Rep(\Sb_\zeta(n,m))\to\Rep(\Hb_\zeta(m)),\quad M\mapsto fM,\cr
&\Phi_*:\Rep(\Hb_\zeta(m))\to\Rep(\Sb_\zeta(n,m)),\quad
N\mapsto\Hom_{\Hb_\zeta(m)}(\Vb_\zeta(n,m),N),\cr
&\Phi_!:\Rep(\Hb_\zeta(m))\to\Rep(\Sb_\zeta(n,m)),\quad M\mapsto
\Tb_\zeta(n,m)\otimes_{\Hb_\zeta(m)}M.
\end{aligned}$$
We call $\Phi^*$ the {\it Schur functor}.
It is a {\it quotient functor}, i.e., it is exact and the
counit $\Phi^*\,\Phi_*\to 1$ is invertible.
The {\it double centralizer property} holds, i.e., we have
$$\Sb_\zeta(n,m)=\End_{\Hb_\zeta(m)}(\Vb_\zeta(n,m)).$$
Equivalently, the functor $\Phi^*$ is fully faithful on projectives, or, equivalently again,
the unit $P\to\Phi_*\,\Phi^*(P)$ is invertible whenever $P$ is  projective.
See \cite[prop.~4.33]{R} for details.
Since $\Phi^*$ is a quotient functor, the functor
$\Phi_!$ takes projectives to projectives and
the unit $1\to\Phi^*\Phi_!$ is an isomorphism of functors.
For $m=m'+m''$ the comultiplication (\ref{delta}) yields a functor
\begin{equation}\label{tensorschur}
\dot\otimes:
\Rep(\Sb_\zeta(n,m'))\otimes\Rep(\Sb_\zeta(n,m''))\to\Rep(\Sb_\zeta(n,m)).
\end{equation}
We'll abbreviate
${}^\Hb\Ind_{m',m''}=\Ind^{\Hb_\zeta(m)}_{\Hb_\zeta(m')\otimes\Hb_\zeta(m'')}$.

\begin{prop} \label{prop:intertwinestensor}
(a) We have a
$(\Sb_\zeta(n,m),\Hb_\zeta(m')\otimes\Hb_\zeta(m''))$-bimodules isomorphism
$\can:\Tb_\zeta(n,m')\dot\otimes\Tb_\zeta(n,m'')\to
\Tb_\zeta(n,m).$
For $M'\in\Rep(\Hb_\zeta(m'))$, $M''\in\Rep(\Hb_\zeta(m''))$ the map $\can$ yields an isomorphism
\begin{equation*}
\can:\Phi_!\bigl({}^\Hb\Ind_{m',m''}
(M'\otimes M'')\bigr)
\to
\Phi_!(M')\dot\otimes \Phi_!(M'').
\end{equation*}

(b) We have an isomorphism of
$(\Hb_\zeta(m')\otimes\Hb_\zeta(m''),\Sb_\zeta(n,m))$-bimodules 
$\can:\Vb_\zeta(n,m')\dot\otimes\Vb_\zeta(n,m'')\to
\Vb_\zeta(n,m).$
For $M'\in\Rep(\Hb_\zeta(m'))$, $M''\in\Rep(\Hb_\zeta(m''))$ the map $\can$ yields an isomorphism
\begin{equation*}
\label{tensor*}
\can:\Phi_*\bigl({}^\Hb\Ind_{m',m''}
(M'\otimes M'')\bigr)
\to
\Phi_*(M')\dot\otimes \Phi_*(M'').
\end{equation*}
\end{prop}

\begin{proof}
By definition $\Tb_\zeta(n,m)$ is the {\it v-tensor space} in \cite[def.~2.6]{DJ}.
According to \cite[sec.~3.3, 4.4]{DD} it is identified with the $m$-th tensor power of the natural representation
of the (modified) quantized enveloping algebra of $\gen\len_n$, in such a way that the $\Hb_\zeta(m)$-action
comes from the R-matrix, see also \cite{Ji}. This proves part $(a)$. 
Part $(b)$ follows also by taking the dual spaces.
\end{proof}

\begin{cor}\label{cor:tensor*}
We have
an isomorphism
\begin{equation*}
\can:{}^\Hb\Ind_{m',m''}
(\Phi^*M'\otimes \Phi^*M'')
\to
\Phi^*(M'\dot\otimes M'')
\end{equation*}
for $M'\in\Rep(\Sb_\zeta(n,m'))$
and $M''\in\Rep(\Sb_\zeta(n,m''))$.
\end{cor}

\begin{proof}
For $M'\in\Rep(\Sb_\zeta(n,m'))$
and $M''\in\Rep(\Sb_\zeta(n,m''))$,
Proposition \ref{prop:intertwinestensor} yields an isomorphism
$$\Phi_*{}^\Hb\Ind_{m',m''}
(\Phi^*M'\otimes \Phi^*M'')
=
\Phi_*\Phi^*M'\dot\otimes \Phi_*\Phi^*M''.$$
Composing it with $\Phi^*$ we get an isomorphism
$${}^\Hb\Ind_{m',m''}
(\Phi^*M'\otimes \Phi^*M'')
=
\Phi^*\bigl(\Phi_*\Phi^*M'\dot\otimes \Phi_*\Phi^*M''\bigr).$$
Composing it with the unit $1\to\Phi_*\Phi^*$ we get a functorial map
$$
\Phi^*(M'\dot\otimes M'')
\to{}^\Hb\Ind_{m',m''}
(\Phi^*M'\otimes \Phi^*M'')
$$
which is invertible whenever $M'$, $M''$ are projectives, because the unit is invertible on projective modules.
Thus it is always invertible, because $\Phi^*$ and ${}^\Hb\Ind_{m',m''}$ are exact and because
there are enough projectives in  $\Rep(\Sb_\zeta(n,m))$.
\end{proof}

\subsection{The braiding and the Schur functor}
For $M'\in\Rep(\Hb_\zeta(m'))$
and $M''\in\Rep(\Hb_\zeta(m''))$ the R-matrix yields an isomorphism
of $\Sb_\zeta(n,m)$-modules
$$\Rc_{\Phi_*M',\Phi_*M''}:
\Phi_*M'\dot\otimes\Phi_*M''
\to
\Phi_*M''\dot\otimes\Phi_*M'.
$$
Let $\tau\in\Sen_m$ be the unique element such that
\begin{itemize}
\item $\tau$ is minimal in the coset
$(\Sen_{m'}\times\Sen_{m''})\tau(\Sen_{m''}\times\Sen_{m'})$,
\item we have
$\tau^{-1}(\Sen_{m'}\times\Sen_{m''})\tau=\Sen_{m''}\times\Sen_{m'}.$
\end{itemize}
We have the following formula in $\Hb_\zeta(m)$
\begin{equation}\label{formuleutile}T_\tau(h''\otimes h')=(h'\otimes h'')\,T_\tau,
\quad
h'\in\Hb_\zeta(m'),\,
h''\in\Hb_\zeta(m'').
\end{equation}
Thus there is a unique functorial $\Hb_\zeta(m)$-module isomorphism
$$\Sc_{M',M''}:
{}^\Hb\Ind_{m',m''}(M'\otimes M'')\to
{}^\Hb\Ind_{m'',m'}(M''\otimes M')$$
given by
$$\Sc_{M',M''}(h\otimes(v'\otimes v''))=
hT_\tau\otimes(v''\otimes v'),\quad
h\in\Hb_\zeta(m),\, v'\in M',\, v''\in M''.$$

\begin{prop}
For $M'\in\Rep(\Hb_\zeta(m'))$, $M''\in\Rep(\Hb_\zeta(m''))$ the
following square is commutative
$$\xymatrix{
\Phi_*{}^\Hb\Ind_{m',m''}
(M'\otimes M'')
\ar[d]_\can
\ar[rr]^
{\Phi_*(\Sc_{M',M''})}&&
\Phi_*{}^\Hb\Ind_{m'',m'}(M''\otimes M')
\ar[d]_\can\cr
\Phi_*M'\dot\otimes\Phi_*M''
\ar[rr]^
{\Rc_{\Phi_*M',\Phi_*M''}}&&
\Phi_*M''\dot\otimes\Phi_*M'.
}$$
\end{prop}

\begin{proof}
We abbreviate
$\Hb=\Hb_\zeta(m)$,
$\Hb'=\Hb_\zeta(m')$,
$\Hb''=\Hb_\zeta(m'')$,
$\Vb=\Vb_\zeta(n,m)$,
$\Vb'=\Vb_\zeta(n,m')$ and
$\Vb''=\Vb_\zeta(n,m'')$.
First, we have a commutative square
\begin{equation}\label{square}\begin{split}
\xymatrix{
\Vb''\dot\otimes\Vb'\ar[d]_\can\ar[rr]^{\Rc_{\Vb'',\Vb'}}&&
\Vb'\dot\otimes\Vb''\ar[d]_\can\cr
\Vb\ar[rr]^{T_\tau}&&\Vb
}
\end{split}
\end{equation}
where the lower map is the left multiplication with $T_\tau$.
See \cite{Ji} and the discussion in the proof of Proposition \ref{prop:intertwinestensor}.
In particular, we have
$$
\Rc_{\Vb'',\Vb'}(h''v''\otimes h'v')=
(h''\otimes h')\Rc_{\Vb'',\Vb'}(v'\otimes v''),\quad
v'\in\Vb',\, v''\in\Vb'',\,
h'\in\Hb',\, h''\in\Hb''.$$
Therefore, the composition by $\Rc_{\Vb'',\Vb'}$ yields a linear map
$$\gathered
\Hom_{\Hb'\otimes\Hb''}(\Vb,M'\otimes M'')=
\Phi_*{}^\Hb\Ind_{m',m''}(M'\otimes M'')
\to\cr
\to\Hom_{\Hb''\otimes\Hb'}(\Vb,M''\otimes M')=
\Phi_*{}^\Hb\Ind_{m'',m'}(M''\otimes M').
\endgathered$$
The commutativity of the square (\ref{square})
implies that this map is equal to
$\Phi_*(\Sc_{M',M''})$.
It is easy to see that this map coincides also with
$\Rc_{\Phi_*M',\Phi_*M''}$.
\end{proof}

\begin{cor}
For $M'\in\Rep(\Sb_\zeta(n,m'))$, $M''\in\Rep(\Sb_\zeta(n,m''))$ the
following square is commutative
$$\xymatrix{
{}^\Hb\Ind_{m',m''}
(\Phi^*M'\otimes \Phi^*M'')
\ar[d]_\can
\ar[rr]^
{\Sc_{\Phi^*M',\Phi^*M''}}&&
{}^\Hb\Ind_{m'',m'}(\Phi^*M''\otimes \Phi^*M')
\ar[d]_\can\cr
\Phi^*(M'\dot\otimes M'')
\ar[rr]^
{\Phi^*(\Rc_{M',M''})}&&
\Phi^*(M''\dot\otimes M')
}$$
\end{cor}

\begin{proof} Use the same argument as in the proof of
Corollary \ref{cor:tensor*}.
\end{proof}

\vskip3mm

\noindent
Let $r\geqslant 1$ and $i=1,2,\dots ,r-1$.
For $M\in\Rep(\Hb_\zeta(m))$ we consider the
automorphism of the $\Hb_\zeta(mr)$-module
${}^\Hb\Ind_{(m^r)}(M^{\otimes r})$ given by
\begin{equation}
\label{Sc}
\gathered
\Sc_{M,i}=
{}^\Hb\Ind_\Hb^{\Hb_\zeta(mr)}
(\oneb^{\otimes i-1}\otimes
\Sc_{M,M}\otimes\oneb^{\otimes r-i-1}),\cr
\Hb=\Hb_\zeta(m)^{\otimes i-1}\otimes\Hb_\zeta(2m)
\otimes\Hb_\zeta(m)^{\otimes r-i-1}.
\endgathered
\end{equation}
For $M\in\Rep(\Sb_\zeta(n,m))$ we consider the
automorphism of the $\Sb_\zeta(n,mr)$-module
$M^{\dot\otimes r}$ given by
\begin{equation}
\label{Rc}
\Rc_{M,i}=\oneb^{\dot\otimes i-1}\dot\otimes
\Rc_{M,M}\dot\otimes \oneb^{\dot\otimes r-i-1}.
\end{equation}

\begin{cor}
\label{cor:B7}
For $M\in\Rep(\Sb_\zeta(n,m))$, $r\geqslant 1$ and $i=1,2,\dots ,r-1$ we
have a commutative square with invertible vertical maps
$$\xymatrix{
{}^\Hb\Ind_{(m^r)}
\Phi^*(M)^{\otimes r}
\ar[d]
\ar[rr]^
{\Sc_{\Phi^*(M),i}}&&
{}^\Hb\Ind_{(m^r)}\Phi^*(M)^{\otimes r}
\ar[d]\cr
\Phi^*(M^{\dot\otimes r})
\ar[rr]^
{\Phi^*(\Rc_{M,i})}&&
\Phi^*(M^{\dot\otimes r})
}$$
\end{cor}

\subsection{The braiding and the quantum Frobenius homomorphism}
Recall that if $\zeta$ is a primitive $d$-th root of 1
then the quantum Frobenius homomorphism (\ref{frob}) yields a functor
$$\Fr^*:\Rep(\Sb_1(n,m))=\Rep(\Sb_{(-1)^d}(n,m))\to\Rep(\Sb_\zeta(n,dm)).$$
Here we have identified
$\Sb_{(-1)^d}(n,m)$ and $\Sb_1(n,m)$ as in Remark \ref{rk:B2}.
Let $m',m''>0$ with $m=m'+m''$.
By Proposition \ref{prop:B1}, for $M\in\Rep(\Sb_{(-1)^d}(n,m)),$
$M'\in\Rep(\Sb_{(-1)^d}(n,m'))$
the braiding operator
$$\Rc_{M,M'}:M\dot\otimes M'\to M'\dot\otimes M$$
is  the composition of the permutation
$\sb_{M,M'}$ and of the operator $$R_{M,M'}=\prod_{\l,\l'}(-1)^{d(\l,\l')}(1_\l\dot\otimes 1_{\l'}).$$

\begin{prop}
\label{prop:B8bis}
For $r\geqslant 1$, $i,j=1,2,\dots r-1$, and
$M\in\Rep(\Sb_{(-1)^d}(n,m))$ the following relations hold
in $\End_{\Sb_{(-1)^d}(n,mr)}(M^{\dot\otimes r})$
\begin{itemize}
\item $\Rc_{M,i}^2=1,$
\item
$\Rc_{M,i}\Rc_{M,j}=\Rc_{M,j}\Rc_{M,i}\ \text{if}\ j\neq i-1,i+1,$
\item
$\Rc_{M,i}\Rc_{M,i+1}\Rc_{M,i}=\Rc_{M,i+1}\Rc_{M,i}\Rc_{M,i+1}\ \text{if}\ i\neq r-1.$
\end{itemize}
\end{prop}

\begin{proof}
The first relation is obvious by definition of the braiding operator, see above.
The other relations are consequences of the general properties of a braiding.
\end{proof}

\vskip2mm

\noindent
Further, the functor $\Fr^*$ is a braided tensor functor, i.e., we have
the following.

\begin{prop}
\label{prop:B8}
For 
$M\in\Rep(\Sb_{(-1)^d}(n,m')),$
$M'\in\Rep(\Sb_{(-1)^d}(n,m''))$
we have a functorial isomorphism
$\Fr^*(M\dot\otimes M')=\Fr^*(M)\dot\otimes\Fr^*(M')$
in $\Rep(\Sb_\zeta(n,dm))$ such that
$\Fr^*(\Rc_{M,M'})=\Rc_{\Fr^*(M),\Fr^*(M')}.$
\end{prop}

\begin{proof}
Obvious by Proposition \ref{prop:B1}.
\end{proof}

\subsection{The algebra $\Sb_\zeta(m)$}
\label{sec:B7}
We'll abbreviate $\Sb_\zeta(m)=\Sb_\zeta(m,m)$.
If $n\geqslant m$
the algebra $\Sb_\zeta(n,m)$ is Morita equivalent to
$\Sb_\zeta(m)$, see e.g., \cite[lem.~1.3]{DJ}.
Thus $\dot\otimes$
can be viewed as a functor (choosing $n\geqslant m=m'+m''$)
$$\dot\otimes:\Rep(\Sb_\zeta(m'))\otimes
\Rep(\Sb_\zeta(m''))\to\Rep(\Sb_\zeta(m)).$$
If $\zeta$ is a primitive $d$-th root of 1
then the quantum Frobenius homomorphism can be viewed as a functor
(choosing $n\geqslant dm$)
$$\Fr^*:\Rep(\Sb_1(m))=\Rep(\Sb_{(-1)^d}(m))\to\Rep(\Sb_\zeta(dm)).$$

\end{document}